\documentclass[11pt,reqno]{amsart}
\usepackage{fullpage}
\usepackage{mathrsfs,amssymb,graphicx,verbatim,amsmath,amsfonts}
\usepackage{paralist}
\usepackage[breaklinks,pdfstartview=FitH]{hyperref}
\usepackage{upgreek}
\renewcommand{\le}{\leqslant}
\renewcommand{\ge}{\geqslant}

\renewcommand{\setminus}{\smallsetminus}
\renewcommand{\gamma}{\upgamma}

\newcommand{\Rad}{\mathrm{\bf Rad}}
\newcommand{\BMW}{\mathrm{BMW}}
\renewcommand{\P}{\mathfrak{P}}
\newcommand{\sign}{\mathrm{sign}}

\newcommand{\n}{\{1,\ldots,n\}}

\renewcommand{\d}{\delta}

\newcommand{\e}{\varepsilon}
\newcommand{\R}{\mathbb R}
\newcommand{\1}{\mathbf 1}

\renewcommand{\b}{\mathfrak{b}}

\newtheorem{theorem}{Theorem}[section]
\newtheorem{proposition}[theorem]{Proposition}
\newtheorem{lemma}[theorem]{Lemma}
\newtheorem{corollary}[theorem]{Corollary}

\theoremstyle{remark}
\newtheorem{conjecture}[theorem]{Conjecture}

\newtheorem{remark}[theorem]{Remark}
\newtheorem{question}[theorem]{Question}
\theoremstyle{definition}
\newtheorem{definition}[theorem]{Definition}

\renewcommand{\subset}{\subseteq}

\newcommand{\C}{\mathbb C}

\DeclareMathOperator{\dist}{\bf dist}
\DeclareMathOperator{\trace}{\bf Tr}
\newcommand{\E}{\mathbb{ E}}

\newcommand{\N}{\mathbb N}
\newcommand{\Z}{\mathbb Z}

\newcommand{\eqdef}{\stackrel{\mathrm{def}}{=}}

\begin{document}


\title[Metric $X_p$ inequalities]{Metric ${X}_p$ inequalities}

\author{Assaf Naor}
\address{Mathematics Department\\ Princeton University\\ Fine Hall, Washington Road, Princeton, NJ 08544-1000, USA}
\email{naor@math.princeton.edu}

\author{Gideon Schechtman}
\address {Department of Mathematics\\
Weizmann Institute of Science\\
Rehovot 76100, Israel}
\email{gideon@weizmann.ac.il}.
\date{}

\begin{abstract} For every $p\in (0,\infty)$ we associate to every metric space $(X,d_X)$ a numerical invariant $\mathfrak{X}_p(X)\in [0,\infty]$ such that if $\mathfrak{X}_p(X)<\infty$ and a metric space $(Y,d_Y)$ admits a bi-Lipschitz embedding into $X$ then also $\mathfrak{X}_p(Y)<\infty$. We prove that if $p,q\in (2,\infty)$ satisfy $q<p$ then $\mathfrak{X}_p(L_p)<\infty$ yet $\mathfrak{X}_p(L_q)=\infty$. Thus our new bi-Lipschitz invariant certifies that $L_q$ does not admit a bi-Lipschitz embedding into $L_p$ when $2<q<p<\infty$. This completes the long-standing search for bi-Lipschitz invariants that serve as an obstruction to the embeddability of $L_p$ spaces into each other, the previously understood cases of which were metric notions of type and cotype, which however fail to certify the nonembeddability of $L_q$ into $L_p$ when $2<q<p<\infty$. Among the consequences of our results are new quantitative restrictions on the bi-Lipschitz embeddability into $L_p$ of snowflakes of $L_q$ and integer grids in $\ell_q^n$, for $2<q<p<\infty$. As a byproduct of our investigations, we also obtain results on the geometry of the Schatten $p$ trace class $S_p$ that are new even in the linear setting.
\end{abstract}

\maketitle

\section{Introduction}

\subsection{Nontechnical overview}\label{sec:overview} As a special case of the main contribution of the present article, for $p\in (0,\infty)$ we associate to every metric space $(X,d_X)$ a numerical invariant $\mathfrak{X}_p(X)\in [0,\infty]$; a precise description  of this quantity appears in Definition~\ref{def:Xpmetric} below.  Given $p\in (0,\infty)$ and two metric spaces $(X,d_X)$ and $(Y,d_Y)$,  any  $f:X\to Y$ incurs distortion at least $\mathfrak{X}_p(X)/\mathfrak{X}_p(Y)$. Thus, $\mathfrak{X}_p(\cdot)$ is a bi-Lipschitz invariant. We shall prove that for $2<q<p<\infty$  we have $\mathfrak{X}_p(L_p)\asymp p/\log p$, while $\mathfrak{X}_p(L_q)=\infty$. Consequently, $L_q$ does not admit a bi-Lipschitz embedding into $L_p$.

Qualitatively, the above nonembedding conclusion is well known. Namely, the fact that $L_q$ fails to admit a bi-Lipschitz embedding into $L_p$ when $2<q<p<\infty$  follows from a differentiation argument that allows one to reduce the question to the linear theory.  Specifically, every Lipschitz mapping $f:L_q\to L_p$ must have~\cite{Man72,Chr73,Aro76} a point of G\^ateaux differentiability $x_0\in L_q$. The derivative $f'(x_0):X\to Y$ is a bounded linear operator, and if $f$ were bi-Lipschitz then it would follow that $f'(x_0)$ is invertible with a bounded inverse, and therefore $L_q$ would be isomorphic to the  linear subspace $f'(x_0)L_q$ of $L_p$. However, a classical theorem of Paley~\cite{Pal36} asserts that $L_q$ is not isomorphic to any subspace of $L_p$, so it follows that $L_q$ also fails to admit a bi-Lipschitz embedding into $L_p$. The above reasoning is due to Mankiewicz~\cite[Theorem~4]{Man72}; Section~\ref{sec:intro details} below contains a more detailed description of the relevant background.

Such differentiation arguments rely on an existential statement (a point of differentiability must exist), followed by a limiting procedure (differentiation itself) that uses the linear structure. As such, they do not apply in many settings, examples of which include understanding the $L_p$ distortion of certain (often discrete) subsets of $L_q$, as well as treating non-Lipschitz (e.g. H\"older) mappings, a setting in which  the mapping may be non-differentiable  at every point~\footnote{By~\cite[Remark~5.10]{MN04}  there does exist a bi-H\"older embedding of $L_q$ into $L_p$ when $2<q<p<\infty$. Hence, the pertinent question is to determine which H\"older exponents are possible here. The non-Lipschitz setting therefore exhibits phenomena that are truly nonlinear and cannot be explained by a direct reduction to the linear theory.}. Crucially, such arguments also fail to give any indication as to how to devise an invariant of metric spaces that certifies that the geometry of certain subsets of $L_q$ is incompatible with the geometry of any subset of $L_p$.

The search for such metric invariants has been an important theme in modern metric geometry, underpinned by a classical rigidity theorem of Ribe~\cite{Rib76} that laid the groundwork for what is known today as the Ribe program; for more on this research program see its original formulation by Bourgain~\cite{Bou86} as well as the recent (though by now not quite up-to-date) surveys~\cite[Section~3]{Kal08},\cite{Bal13} and~\cite{Nao12}. It suffices to say here that Ribe's theorem indicates that certain types of linear properties of Banach spaces (including those properties that are used in some, but not all, of the known proofs that $L_q$ is not isomorphic to any linear subspace of $L_p$ when $2<q<p<\infty$), may in fact be metric properties in disguise, i.e., they could be reformulated without making any reference to the linear structure whatsoever, so as to make sense in any metric space and thus provide a dictionary that allows one to apply linear intuitions in purely metric contexts. This paradigm is very powerful, leading to solutions of questions in a wide variety of areas, ranging from the nonlinear geometry of Banach spaces themselves, to settings that a priori  have seemingly nothing to do with Banach spaces, such as group theory, harmonic analysis, probability and combinatorial optimization.

Among the first questions that one would ask about bi-Lipschitz embeddings is to characterise those $p,q\in [1,\infty)$ such that $L_q$ fails to admit a bi-Lipschitz embedding into $L_p$. Not surprisingly, efforts to understand this question influenced some of the most important developments in the Ribe program. By a reduction to the linear theory through differentiation in a manner that is similar to what we described above, the qualitative answer here is known: $L_q$ does not admit a bi-Lipschitz embedding into $L_p$ if and only if $p,q\in [1,\infty)$ satisfy one of the following three conditions.
\begin{equation}\label{eq:ranges}
q<\min\{p,2\}\qquad \mathrm{or}\qquad q>\max\{p,2\}\qquad\mathrm{or}\qquad 2<q<p<\infty.
\end{equation}

The search for metric invariants that explain the first range in~\eqref{eq:ranges} was an important impetus in the development of the theory of type of metric spaces, with notable contributions by Enflo~\cite{Enf69,Enf70,Enf76}, Bourgain--Milman--Wolfson~\cite{BMW86}, Pisier~\cite{Pis86} and Ball~\cite{Bal92}; see also~\cite{NS02,NPSS06,MN07,NP08,Oht09,GN10,NP11,DLP13,HN13,Nao14}. The search for metric invariants that explain the second range in~\eqref{eq:ranges} was an important impetus in the development of the theory of cotype of metric spaces; see the work of Mendel and Naor~\cite{MN08} as well as~\cite{Bal92,GMN11,MN13-barycentric,MN14-calculus}. The second range in~\eqref{eq:ranges} could also  be explained through a metric invariant called Markov convexity; see~\cite{Bou86,LNP09,MN13-convexity}. Over the years, many applications  of the above invariants (metric type, metric cotype, Markov convexity) to a wide range of areas were discovered; the above mentioned references contain examples of such results, and a variety of additional examples appears in~\cite{LMN02,BLMN05,MN06,ANP09,VW10,Li14,AN09,VW10,MN15-hadamard,ANN15,BZ15}. Despite these developments, the question of formulating a metric invariant that explains the third range in~\eqref{eq:ranges} remained unresolved for many years. Here we settle this remaining case by introducing an invariant of metric spaces that serves as an obstruction to the embeddability of $L_q$ into $L_p$ when $2<q<p<\infty$, thus completing the repertoire of metric invariants that classify those $p,q\in [1,\infty)$ for which  $L_q$ admits a bi-Lipschitz embedding into $L_p$.

Our new metric invariant is described in the following definition, in which (and in what follows)  for every $n\in \N$  we let $e_1,\ldots,e_n$ denote the standard basis of $\R^n$, and for $S\subset \n$ and $\e=(\e_1,\ldots,\e_n)\in \{-1,1\}^n$ we denote $\e_S=\sum_{j\in S}\e_j e_j$.

\begin{definition}[$X_p$ metric space]\label{def:Xpmetric} Let $(X,d_X)$ be a metric space and  $p\in (0,\infty)$. Say that $(X,d_X)$ is an $X_p$ metric space if there exists  $\mathfrak{X}\in (0,\infty)$ such that for every $n\in \N$ and $k\in \n$ there exists $m\in \N$ such that  every mapping $f:\Z_{2m}^n\to X$ satisfies
\begin{multline}\label{eq:metric space is X_p beginning def2}
\bigg(\frac{1}{\binom{n}{k}}\sum_{\substack{S\subset
\n\\|S|= k}}
\E \Big[d_X\big(f(x+m\e_S),f(x)\big)^p\Big]\bigg)^{\frac{1}{p}}\\\le \mathfrak{X}m\Bigg(
\frac{k}{n}\sum_{j=1}^n\E\Big[d_X\left(
f(x+e_j),f(x)\right)^p\Big]+\left(\frac{k}{n}\right)^{\frac{p}{2}}
\E\Big[d_X(f(x+ \e),f(x))^p\Big]\bigg)^{\frac{1}{p}},
\end{multline}
where the expectations in~\eqref{eq:metric space is X_p beginning def2} are with respect to $(x,\e)\in \Z_{2m}^n\times \{-1,1\}^n$ chosen uniformly at random. The infimum over those $\mathfrak{X}\in (0,\infty)$ for which~\eqref{eq:metric space is X_p beginning def2} holds true is denoted  $\mathfrak{X}_p(X,d_X)$, or simply $\mathfrak{X}_p(X)$ if the metric is clear from the context.
\end{definition}

Theorem~\ref{thm:main} below establishes that $L_p$ is an $X_p$ metric space when $p\ge 2$. We shall also check that $L_q$ is not an $X_p$ metric space when $q\in (2,p)$. Since for a metric space $(X,d_X)$ the property of being an $X_p$ metric space is obviously inherited by all the metric spaces that admit a bi-Lipschitz embedding into $X$, we thus obtain a new proof of the fact that $L_q$ fails to admit a bi-Lipschitz embedding into $L_p$ when $2<q<p<\infty$. We shall show that the metric $X_p$ invariant yields results that were beyond the reach of previous methods. For example, we shall obtain the first nontrivial upper bound on those $\theta\in (0,1]$ for which $L_q$ admits a bi-$\theta$-H\"older embedding into $L_p$.

The above overview covered the context of our results without going into various technicalities, and as such it did not provide an explanation of how  we arrived at Definition~\ref{def:Xpmetric}. There are also technical subtleties that partially explain (in hindsight) why understanding the third range in~\eqref{eq:ranges} remained open for so much longer than the same question for the first two ranges in~\eqref{eq:ranges}. These matters will be clarified in the remainder of this introduction starting from Section~\ref{sec:intro details} below, where we shall also describe consequences of our work, including new results even within the linear theory, as well as intriguing open questions that it raises.

\subsection{Detailed statements and technical background}\label{sec:intro details}

The ensuing discussion uses standard notation and terminology from Banach space theory, as in~\cite{LT77}. In particular, for $p\in [1,\infty]$ and $n\in \N$, the space $\ell_p^n$ (respectively $\ell_p^n(\C)$) denotes the vector space $\R^n$ (respectively $\C^n$), equipped with the standard $\ell_p$ norm. Our results apply equally well to any infinite dimensional Lebesgue function space $L_p(\mu)$, but for concreteness we fix (as usual) the space $L_p$ to be equal to $L_p([0,1],\mathscr{L})$, where $\mathscr{L}$ is the Lebesgue measure. Banach spaces are assumed to be over real scalars unless stated otherwise, though our results hold true mutatis mutandis for complex Banach spaces as well.

We shall also use standard notation and terminology from the theory of metric embeddings, as in~\cite{Mat02,Ost13}. In particular, a metric space $(X,d_X)$ is said to admit a bi-Lipschitz embedding into a metric space $(Y,d_Y)$ if there exist $s\in (0,\infty)$, $D\in [1,\infty)$ and a mapping $f:X\to Y$ such that
\begin{equation}\label{eq:distortion definition}
\forall\, x,y\in X,\qquad sd_X(x,y)\le d_Y(f(x),f(y))\le Dsd_X(x,y)
\end{equation}
When this happens we say that $(X,d_X)$ embeds into $(Y,d_Y)$ with distortion at most $D$. Given $f:X\to Y$, the infimum over those $D\in [1,\infty)$ for which there exists $s\in (0,\infty)$ such that~\eqref{eq:distortion definition} holds true is called the distortion of $f$ and is denoted $\dist(f)$. If no such $D$ exists set $\dist(f)=\infty$. We denote by $c_{(Y,d_Y)}(X,d_X)$ (or simply $c_Y(X)$ if the metrics are clear from the context) the infimum over those $D\in [1,\infty]$ for which $(X,d_X)$ embeds into $(Y,d_Y)$ with distortion at most $D$. If $(X,d_X)$ does not admit a bi-Lipschitz embedding into $(Y,d_Y)$ then we set $c_{(Y,d_Y)}(X,d_X)=\infty$. When $Y=L_p$ we use the shorter notation $c_{L_p}(X,d_X)=c_p(X,d_X)$.

As we discussed in Section~\ref{sec:overview}, among the simplest and most basic questions that one could ask in the context of metric embeddings is to determine those $p,q\in [1,\infty)$ for which $L_q$ admits a bi-Lipschitz embedding into $L_p$. This is well understood via a reduction to the linear theory, from which we deduce that $L_q$ admits a bi-Lipschitz embedding into $L_p$ if and only if either $q=2$ or $1\le p\le q\le 2$ (moreover, in these cases we have $c_p(L_q)=1$). Indeed, by general principles (see Chapter~7 of~\cite{BL00} and the references therein), relying mainly on differentiation theorems for Lipschitz mappings between Banach spaces (the case $p=1$ being somewhat different from the reflexive range), it suffices to understand when $L_q$ is isomorphic to a subspace of $L_p$, a question that is perhaps among the first issues that one would investigate when studying linear embeddings of Banach spaces. Chapter 12 of Banach's book~\cite{Ban32} is devoted to this topic. Banach proved there that if $L_q$ is isomorphic to a subspace of $L_p$ then necessarily either $p\le q\le 2$ or $2\le q\le p$, and that $L_2$ is isomorphic to a subspace of $L_p$ for all $p\in [1,\infty)$. Banach also conjectured~\cite[page~205]{Ban32} that $L_q$ is isomorphic to a subspace of $L_p$ if $p<q<2$ or $2<q<p$. In the range $p<q<2$, Banach's question was answered affirmatively by Kadec~\cite{Kad58}, who showed that in this case $L_q$ is linearly isometric to a subspace of $L_p$. When $2<q<p$, Banach's question was answered negatively by Paley~\cite{Pal36}, i.e., $L_q$ is not isomorphic to a subspace of $L_p$ when $2<q<p$.

As we explained above, our goal here is to obtain a nonlinear version of Paley's theorem, i.e., the formulation of a bi-Lipschitz invariant that serves as an obstruction to the embeddability of $L_q$ into $L_p$ when $2<q<p$. This invariant allows us to obtain nonembeddability results that were beyond the reach of previously available methods, and in addition it leads to interesting open questions. Our new invariant thus completes a long line of work on the bi-Lipschitz classification of $L_p$ spaces, because the remaining cases, namely the bi-Lipschitz nonembeddability of $L_q$ into $L_p$ when either $q\in [1,2)$ and $p>q$, or $q\in (2,\infty)$ and $p<q$, were previously understood through notions of metric type and cotype that were introduced over the past four decades (see below for more on this topic).

Our main result is the following theorem, which, using the notation and terminology of Definition~\ref{def:Xpmetric}, asserts that if $p\in (2,\infty)$ then $L_p$ is an $X_p$ metric space, with $\mathfrak{X}_p(L_p)\lesssim p/\log p$.

\begin{theorem}[Metric $X_p$ inequality]\label{thm:main} Fix $p\in [2,\infty)$. Suppose that $m,n\in \N$ and $k\in \n$ satisfy $$m\ge \frac{n^{\frac32}\log p}{\sqrt{k}}+pn.$$ Then, for every $f:\Z_{4m}^n\to L_p$ we have
\begin{multline}\label{eq:Xp in theorem1}
\frac{(p/\log p)^{-p}}{\binom{n}{k}}\sum_{\substack{S\subset
\n\\|S|= k}}
\frac{\E\left[\left\|f\left(x+2m\e_S\right)-f(x)\right\|_p^p\right]}{m^p}\\\lesssim_p
\frac{k}{n}\sum_{j=1}^n\E\left[\left\|
f(x+e_j)-f(x)\right\|_p^p\right]+\left(\frac{k}{n}\right)^{\frac{p}{2}}
\E\left[\left\|f\left(x+ \e\right)-f(x)\right\|_p^p\right],
\end{multline}
where the expectation is with respect to $(x,\e)\in \Z_{4m}^n\times \{-1,1\}^n$  chosen uniformly at random.
\end{theorem}
\noindent{\bf Asymptotic notation.} In Theorem~\ref{thm:main}, and in what follows, we use the (somewhat nonstandard) convention that for $a,b\in [0,\infty)$ and $p\in [1,\infty)$ the notation $a\lesssim_p b$ (respectively $a\gtrsim_p b$) stands for $a\le c^pb$ (respectively $a\ge c^pb$) for some universal constant $c\in (0,\infty)$.  The notation $a\lesssim b$ (respectively $a\gtrsim b$) stands for $a\le c b$ (respectively $a\ge c b$) for some universal constant $c\in (0,\infty)$. The notation $a\asymp b$ stands for $(a\lesssim b)\wedge (b\lesssim a)$. At times our discussion will be in the presence of an auxiliary Banach (or metric) space $X$, in which case the notation $a\lesssim_X b$ will stand for $a\le c(X)b$, where $c(X)\in (0,\infty)$ is allowed to depend only on $X$ (in fact, $c(X)$ will always depend on certain numerical geometric invariants of $X$ that will be clear from the context).

\medskip

The term $p/\log p$ in the left-hand side of~\eqref{eq:Xp in theorem1} is sharp up to a universal constant factor.  We defer the explanation of why~\eqref{eq:Xp in theorem1} is called a metric $X_p$ inequality to the ensuing discussion. Note that since~\eqref{eq:Xp in theorem1} involves the $p$'th power of $L_p$ norms, it suffices to prove its validity when $f$ is real-valued, but we stated Theorem~\ref{thm:main} for functions with values in $L_p$ since this is the way by which we will apply it to prove  new nonembeddability results. The fact that in Theorem~\ref{thm:main} the function $f$ is assumed to be defined on the discrete torus $\Z_{4m}^n$ rather than on $\Z_m^n$ is not important: for notational reasons it is beneficial to work with $\Z_m^n$ when the modulus $m$ is divisible by $4$, and this suffices for all of the applications of~\eqref{eq:Xp in theorem1} that we can imagine. However, it is straightforward to modify our proof of Theorem~\ref{thm:main} so as to obtain variants of~\eqref{eq:Xp in theorem1} for functions defined on discrete tori whose modulus is not necessarily divisible by $4$.

\begin{remark}\label{rem:smaller m worse constant}
If one makes the weaker assumption $m\ge n^{3/2}/\sqrt{k}$ in Theorem~\ref{thm:main} then~\eqref{eq:Xp in theorem1} holds true with the (sharp) term $p/\log p$ in the left-hand side replaced by $p^2/\log p$. This, and additional tradeoffs of this type, can be deduced from an inspection of our proof of Theorem~\ref{thm:main}.
\end{remark}

\subsection{Quantitative nonembeddability}\label{sec:quantitative sketch intro} The above
classification of those $p,q\in [1,\infty)$ for which $L_q$ admits a
bi-Lipschitz embedding into $L_p$ is based on an abstract reduction
to linear embeddings, and as such it fails to yield a metric
invariant that serves as an obstruction to bi-Lipschitz embeddings.
This argument also does not imply various quantitative estimates
that are inherently nonlinear and cannot be deduced from the linear
theory. For example, given a metric space $(X,d_X)$ and $\theta\in
(0,1]$, the $\theta$-snowflake of $(X,d_X)$ is defined (see e.g.~\cite{DS97}) to be the
metric space $(X,d_X^\theta)$. A natural quantitative refinement of
the assertion that $L_q$ does not admit a bi-Lipschitz embedding
into $L_p$ is that if the $\theta$-snowflake of $L_q$ admits a
bi-Lipschitz embedding into $L_p$ then necessarily $\theta$ must be
bounded away from $1$ by a definite constant (depending on $p,q$). While such statements are known (through the theory
of metric type and cotype; see below) when either $q\in [1,2)$ and
$p>q$, or $p\in (2,\infty)$ and $q>p$, in the range $2<q<p$ no such
quantitative refinement of bi-Lipschitz nonembeddability was
previously known. For $2<q<p$, in Theorem~\ref{thm:our snowflake}
below we obtain, as a consequence of Theorem~\ref{thm:main}, an
explicit $\d(p,q)\in (0,1)$ such that if the $\theta$-snowflake of
$L_q$ admits a bi-Lipschitz embedding into $L_p$ then necessarily
$\theta\le 1-\d(p,q)$. In Section~\ref{sec:best m conv} we formulate
a conjectural convolution inequality that is shown to yield the
sharp value $\d(p,q)$ in this context. Since H\"older mappings need
not be differentiable anywhere, and moreover continuous linear
mappings are necessarily Lipschitz, it seems impossible to obtain a
restriction on those snowflakes of $L_q$ that embed into $L_p$ via a
reduction to linear embeddings as above.

Another natural quantitative refinement of the bi-Lipschitz nonembeddability of $L_q$ into $L_p$ is, given $m,n\in \N$, to ask for a lower bound on $c_p([m]_q^n)$, where here, and in what follows, $[m]_q^n$ denotes the grid $\{0,\ldots,m\}^n\subset \R^n$, equipped with the metric inherited from $\ell_q^n$. While such an estimate can be obtained from general principles, namely Bourgain's discretization theorem~\cite{Bou87,GNS12} (see Remark~\ref{rem:Bourgain discretization} below), in Theorem~\ref{thm:our grid distortion} we obtain, as a consequence of Theorem~\ref{thm:main}, the best known lower bound on $c_p([m]_q^n)$ when $2<q<p$. The convolution inequality that is conjectured in Section~\ref{sec:best m conv} is shown to imply an asymptotically sharp evaluation of $c_p([m]_q^n)$, exhibiting a striking phase transition when $m\asymp n^{(p-q)/(q(p-2))}$; see  Theorem~\ref{thm:best m implies sharp nonembeddability} below.

\subsection{Local invariants}\label{sec:local ribe} Suppose that $p,q\in [1,\infty)$ are such that $L_q$ does not admit a bi-Lipschitz
embedding into $L_p$. This assertion is local in the sense that the
smallest possible distortion of a linear embedding of $\ell_q^n$
into $L_p$ tends to $\infty$ with $n$. Thus, there is a finite
dimensional linear obstruction (which will be stated explicitly in
Section~\ref{sec:type cotype monster} below) showing that no
$n$-dimensional subspace of $L_p$ can be close to $\ell_q^n$. As we discussed in Section~\ref{sec:overview}, an
important rigidity theorem of Ribe~\cite{Rib76} suggests that such
finite dimensional linear obstructions can be reformulated while
only referring to distances between pairs of points. This is the
basis for the Ribe program~\cite{Bou86,Nao12,Bal13}, and our work
constitutes a completion of this program for $L_p$ spaces, the
previously missing case being when $2<q<p$.  The next step in the
Ribe program, a step that has proven in the past to be useful for
various questions in metric geometry, would be to study $X_p$ metric
spaces in their own right. However, unlike
previous advances in the Ribe program, in the present setting it
seems more natural for the linear theory to be developed
further before its metric counterpart is investigated; we discuss
this matter  and formulate some related open problems in
Section~\ref{sec:metric} below.

\subsection{Type, cotype and symmetric structures}\label{sec:type cotype monster} For
$r,s\in [1,\infty)$, a Banach space $(X,\|\cdot\|_X)$ is said to
have Rademacher type $r$ and cotype $s$ if for every $n\in \N$ and
$x_1,\ldots,x_n\in X$ we have
\begin{equation}\label{eq:type cotype def}
\bigg(\E\bigg[\Big\|\sum_{j=1}^n\e_jx_j\Big\|_X^r\bigg]\bigg)^{\frac{1}{r}}\lesssim_X \bigg(\sum_{j=1}^n \|x_j\|_X^r\bigg)^{\frac{1}{r}}
\quad \mathrm{and}
\quad \bigg(\sum_{j=1}^n \|x_j\|_X^s\bigg)^{\frac{1}{s}}\lesssim_X \bigg(\E\bigg[ \Big\|\sum_{j=1}^n\e_jx_j\Big\|_X^s\bigg]\bigg)^{\frac{1}{s}},
\end{equation}
where the expectation is with respect to $\e\in \{-1,1\}^n$ chosen uniformly at random. The infimum over the implicit constants for which~\eqref{eq:type cotype def} holds true are denoted $T_r(X)$ and $C_s(X)$, respectively. See~\cite{Mau03} and the references therein for more on these
important notions. It suffices to say here that if $p\in [1,\infty)$ then $L_p$ has type
$\min\{p,2\}$ and cotype $\max\{q,2\}$, from which one deduces that
there exists $\kappa(p)\in (0,\infty)$ such that if $T:\ell_q^n\to
L_p$ is an invertible linear operator then necessarily
\begin{equation}\label{eq:type cotype distortions}
\dist(T)=\|T\|\cdot\|T^{-1}\|\ge \kappa(p)\cdot\left\{ \begin{array}{ll}
n^{\frac{1}{q}-\frac{1}{p}}& \mathrm{if}\ 1\le q\le p\le 2,\\
n^{\frac{1}{q}-\frac{1}{2}}& \mathrm{if}\ 1\le q\le 2\le p<\infty,\\
n^{\frac{1}{p}-\frac{1}{q}}& \mathrm{if}\ 2\le p\le q,\\
n^{\frac{1}{2}-\frac{1}{q}}& \mathrm{if}\ 1\le p\le 2\le q.
\end{array}\right.
\end{equation}
\eqref{eq:type cotype distortions} follows from an application
of~\eqref{eq:type cotype def} with $X=L_p$, $r=\min\{p,2\}$,
$s=\max\{p,2\}$ and $x_j=Te_j$. The bounds in~\eqref{eq:type cotype
distortions} cannot be improved up to the value of $\kappa(p)$.
Thus, type and cotype constitute the finite dimensional linear
invariants that were alluded to in Section~\ref{sec:local ribe},
i.e., they certify (in a sharp way) that if either $q\in [1,2)$ and
$p>q$ or $q\in (2,\infty)$ and $q>p$, then any linear embedding of
$\ell_q^n$ into $L_p$ incurs large distortion.

The usefulness of the notions of Rademacher type and cotype goes far
beyond their relevance to embeddings of $L_p$ spaces. For this
reason (in addition to the intrinsic geometric interest arising from
the Ribe program) there has been considerable effort to reformulate
these notions while using only distances between pairs of points
rather than linear combinations of vectors as in~\eqref{eq:type
cotype def}, thereby understanding when a metric space has type $r$
and cotype $s$. We will quickly recall now a very small part of what
is known in this direction, stating only those results that are
needed for the present discussion on metric $X_p$ inequalities.

Following Enflo~\cite{Enf76}, a metric space $(X,d_X)$ is said to
have Enflo type $r\in [1,\infty)$ if for every $n\in \N$ and
$f:\{-1,1\}^n\to X$,
\begin{equation}\label{eq:def enflo type intro}
\E\left[d_X(f(\e),f(-\e))^r\right]\lesssim_X\sum_{j=1}^n\E\left[d_X(f(\e),f(\e_1,\ldots,\e_{j-1},-\e_j,
\e_{j+1},\ldots,\e_n))^r\right],
\end{equation}
where the expectation is with respect to $\e\in \{-1,1\}^n$ chosen uniformly at random. Note that if $X$ is a Banach space then~\eqref{eq:def enflo type
intro} coincides with the leftmost inequality in~\eqref{eq:type
cotype def} when $f$ is the linear function given by
$f(\e)=\sum_{j=1}^n \e_j x_j$. For $p\in [1,\infty)$, $L_p$ actually
has Enflo type $r=\min \{p,2\}$, i.e., $X=L_p$
satisfies~\eqref{eq:def enflo type intro} with $f:\{-1,1\}^n\to L_p$
allowed to be an arbitrary mapping rather than only a linear mapping. This statement was first proved
for $p\in [1,2]$ in~\cite{Enf69} and for $p\in (2,\infty)$
in~\cite{NS02}.

One is tempted to define when a metric space $(X,d_X)$ has cotype
$s\in (0,\infty)$ by reversing the inequality in~\eqref{eq:def enflo
type intro} (with $r$ replaced by $s$). But, note that if
$d_X(f(\e),f(\d))=1$ for every distinct $\e,\d\in \{-1,1\}^n$ (this can occur
even if $X$ is a Hilbert space) then the right-hand side
of~\eqref{eq:def enflo type intro} grows linearly with $n$ as
$n\to \infty$, while the left hand side of~\eqref{eq:def enflo type
intro} remains bounded. Thus, there are truly nonlinear phenomena
that do not occur in the linear setting of Rademacher cotype which
do not allow for the straightforward reversal of the inequality
in~\eqref{eq:def enflo type intro}. In essence, the total mass of
the measure that appears in the right-hand side of~\eqref{eq:def
enflo type intro} is too large in comparison to the total mass of
the measure that appears in the left-hand side of~\eqref{eq:def
enflo type intro} for an inequality that is the reverse
of~\eqref{eq:def enflo type intro} to make any sense even in Hilbert
space (it actually fails in any non-singleton metric space;
see~\cite{MN08}).

The solution to this problem comes by considering functions defined
on $\Z_m^n$ rather than on $\{-1,1\}^n$, and {\em scaling} the
argument of the function. Specifically, following~\cite{MN08} say
that a metric space $(X,d_X)$ has metric cotype $s\in (0,\infty)$ if
for every $n\in \N$ there exists $m\in \N$ such that
\begin{equation}\label{eq:def metric cotype}
\forall\, f:\Z_{2m}^n \to X,\qquad \sum_{j=1}^n \frac{\E\left[d_X(f(x+me_j),f(x))^s\right]}{m^s}
\lesssim_X \E\left[ d_X(f(x+\e),f(x))^s\right],
\end{equation}
where the expectation is with respect to $(x,\e)\in \Z_{2m}^n\times \{-1,0,1\}^n$ chosen uniformly at random.  It was proved in~\cite{MN08} that a Banach space $(X,\|\cdot\|_X)$
has Rademacher cotype $s$ if and only if it has metric cotype $s$,
in particular $L_p$ has metric cotype $\max\{p,2\}$. ``Scaling"
refers to the fact that in~\eqref{eq:def metric cotype} we consider
displacements of the argument of $f$ by a multiple of $m$, i.e., we
consider distances between $f(x+me_j)$ and $f(x)$ rather than
distances between $f(x+e_j)$ and $f(x)$, and then we compensate for
this by normalizing the distances appropriately. This idea makes its
appearance also in the left-hand side of our metric $X_p$
inequality~\eqref{eq:Xp in theorem1}, but we shall see below that
the need for scaling in the context of Theorem~\ref{thm:main} is due
to a more subtle reason than the above explanation of why scaling is
needed in the context of metric cotype (compare the total masses of
the measures that appear in both sides of~\eqref{eq:Xp in theorem1}
to see that it doesn't cause the problem that we presented above).

\subsubsection{The case $2<q<p$} While Paley's work~\cite{Pal36} from 1936 established that $L_q$ is not isomorphic to a subspace for $L_p$ when $2<q<p$, several decades later more structural approaches to this theorem were developed. In 1962, Kadec and Pe\l czy\'nski~\cite{KP62}  introduced an influential way  to solve this problem through a structural study of basic sequences in $L_p$ spaces. In particular, it follows from~\cite{KP62}   that for $p\in (2,\infty)$, any infinite symmetric basic sequence in $L_p$ is equivalent to either the standard basis of $\ell_p$ or the standard basis of $\ell_2$. Consequently, for $q\in (2,p)$ there does not exist a symmetric basic sequence in $L_p$ that is equivalent to the unit basis of $\ell_q$, and therefore $\ell_q$ cannot be isomorphic to a subspace of $\ell_p$. In 1979, Johnson, Maurey, Schechtman and Tzafriri~\cite{JMST79} obtained a proof of Paley's theorem via a classification  of finite symmetric bases in function spaces, leading to a comprehensive theory of symmetric structures in Banach spaces to which the research monograph~\cite{JMST79} is devoted. In particular, in~\cite{JMST79} a ``local" version of the above theorem of Kadec and Pe\l czy\'nski is studied, leading to a classification of all {\em finite} symmetric bases in $L_p$. It turns out that in this finitary setting the classification involves more structures than those that are allowed (by the Kadec--Pe\l czy\'nski theorem) for infinite symmetric sequences in $L_p$, namely, a one-parameter family of such sequences can occur, yet any finite symmetric sequence in $L_p$ is equivalent to a member of this one-parameter family. This theorem of~\cite{JMST79} is the starting point of our work here.

Given a Banach space $(X,\|\cdot\|_X)$, $n\in \N$ and $K\in [1,\infty)$, recall that a linearly independent sequence of vectors $(x_1,\ldots,x_n)\in X^n$ is said to be $K$-symmetric if for every sequence of scalars $a_1,\ldots,a_n\in \R$, every permutation $\pi\in S_n$ and every sequence of signs $\e=(\e_1,\ldots,\e_n)\in \{-1,1\}^n$, we have $\|\e_1a_{\pi(1)}x_1+\ldots+\e_na_{\pi(n)} x_n\|_X\le K\|a_1x_1+\ldots+a_n x_n\|_X$. The sequence $(x_1,\ldots,x_n)\in X^n$ is said to be normalized if $\|x_j\|_X=1$ for all $j\in \n$. Given two Banach spaces $(X,\|\cdot\|_X)$ and $(Y,\|\cdot\|_Y)$, two sequences  $(x_1,\ldots,x_n)\in X^n$ and $(y_1,\ldots,y_n)\in Y^n$  are said to be $K$-equivalent if there exists $s\in (0,\infty)$ such that  $s\|a_1x_1+\ldots+a_n x_n\|_X\le \|a_1y_1+\ldots+a_n y_n\|_Y\le Ks\|a_1x_1+\ldots+a_n x_n\|_X$ for all choices of scalars $a_1,\ldots,a_n\in \R$.

Fixing $p\in (2,\infty)$, examples of symmetric sequences in $L_p$ are furnished by Rosenthal's $X_p^n(\omega)$ symmetric bases~\cite{Ros70}, as $\omega$ ranges over $[0,\infty)$. The definition of these bases is very simple: let $u_1,\ldots,u_n$ be an orthonormal basis of $\ell_2^n$ and define $\{x_j(p,\omega)\}_{j=1}^n\subset (\ell_p^n\oplus \ell_2^n)_p$ by

\begin{equation}\label{eq:def x(p,omega)}
x_j(p,\omega)\eqdef \frac{1}{{\left(1+\omega^p\right)^{\frac{1}{p}}}}\cdot e_j+\frac{\omega}{\left(1+\omega^p\right)^{\frac{1}{p}}}\cdot  u_j.
\end{equation}
The $1$-symmetric sequence $\{x_j(p,\omega)\}_{j=1}^n$ is known in the literature as Rosenthal's $X_p^n(\omega)$ basis. Note that since $\ell_2$ is isometric to a subset of $L_p$ (see e.g.~\cite{Woj91}), the sequence $\{x_j(p,\omega)\}_{j=1}^n$  can be realized as elements of $L_p$.

In~\cite{JMST79} it was proved that for every $K\in [1,\infty)$ and $p\in (2,\infty)$ there exists $D(p,K)\in (0,\infty)$ such that every $K$-symmetric sequence $(x_1,\ldots,x_n)$ in $L_p$ is $D(p,K)$-equivalent to an $X_p^n(\omega)$ basis for some $\omega\in [0,\infty)$. This classification theorem has immediate relevance to linear embeddings of $\ell_q^n$ into $L_p$. Indeed, if $T:\ell_q^n\to L_p$ is injective and linear then $(Te_1,\ldots, Te_n)$ is a $\dist(T)$-symmetric sequence in $L_p$, and is therefore $D(p,\dist(T))$-equivalent to an $X_p^n(\omega)$ basis for some $\omega\in (0,\infty)$. Direct inspection now reveals that this is only possible if $\dist(T)$ tends to $\infty$ as $n\to \infty$. In fact, by computing the various bounds explicitly and optimizing over $\omega\in [0,\infty)$, as done in~\cite{FJS88} (relying in part on a computation from~\cite{GPP80}), one can deduce that for every $2<q<p<\infty$ there exists $\sigma(p,q)\in (0,\infty)$ such that for every invertible linear mapping $T:\ell_q^n\to L_p$ we have
\begin{equation}\label{eq:banach mazur q p}
\dist(T)\ge \sigma(p,q)\cdot n^{\frac{(p-q)(q-2)}{q^2(p-2)}}.
\end{equation}

 The lower bound in~\eqref{eq:banach mazur q p} is asymptotically sharp (up to the implicit dependence on $p,q$), as exhibited by the embedding $J^R_{(q\to p;n)}:\ell_q^n\to (\ell_p^n\oplus \ell_2^n)_p\subset L_p$ given by\footnote{The superscript in the notation $J^R_{(q\to p;n)}(\cdot)$ refers to Rosenthal.}
\begin{equation}\label{eq:def JR}
\forall\, j\in \n,\qquad J^R_{(q\to p;n)}(e_j)\eqdef n^{\frac12}\cdot e_j+n^{\frac{1}{q}} \cdot u_j,
\end{equation}
where, as in~\eqref{eq:def x(p,omega)}, $u_1,\ldots,u_n$ is an orthonormal basis of $\ell_2^n$. Indeed, by a straightforward Langrange multiplier argument (see Section~\ref{sec:prelim} below), for every $2<q\le p$ we have
\begin{equation}\label{eq:distortion of JR}
\dist\!\left(J^R_{(q\to p;n)}\right)\asymp n^{\frac{(p-q)(q-2)}{q^2(p-2)}}.
\end{equation}

A sequence of random variables $\{Y_j\}_{j=1}^n$ is said to be symmetrically exchangeable if for every $\pi\in S_n$ and $\e\in \{-1,1\}^n$ the random vectors $(\e_1Y_{\pi(1)},\ldots,\e_nY_{\pi(n)})$ and $(Y_1,\ldots,Y_n)$ are identically distributed. The proof of the above classification of finite symmetric sequences in $L_p$ relies on the following inequality~\cite{JMST79}. Fix $p\in [2,\infty)$ and suppose that $\{Y_j\}_{j=1}^n$ are symmetrically exchangeable random variables with $\E[|Y_j|^p]=1$ for all $j\in \n$. Then for every $t_1,\ldots,t_n\in \R$,
\begin{equation}\label{eq:quote JSZ}
\left(\frac{\log p}{p}\right)^p\cdot \E\left[\Big|\sum_{j=1}^n t_j Y_j\Big|^p\right]\lesssim_p \sum_{j=1}^n |t_j|^p+\left(\frac{1}{n}\sum_{j=1}^n t_j^2\right)^{\frac{p}{2}}\E\left[\left(\sum_{j=1}^n Y_j^2\right)^{\frac{p}{2}}\right].
\end{equation}
The term $(\log p)/p$ in the left-hand side of~\eqref{eq:quote JSZ} is sharp up to a universal constant factor: in this sharp form  the inequality~\eqref{eq:quote JSZ} is due to~\cite{JSZ85}. Without a sharp dependence on $p$, inequality~\eqref{eq:quote JSZ} was first proved in~\cite{JMST79}. The proof of~\eqref{eq:quote JSZ} with sharp dependence on $p$ is significantly more involved than the proof in~\cite{JMST79}. The dependence on $p$ is not of major importance for us here, but it is worthwhile to state the above sharp form of~\eqref{eq:quote JSZ} since it is available in the literature.

Fix $p\in [2,\infty)$, $n\in \N$ and $a_1,\ldots,a_n\in \R$. For $(\e,\pi)\in \{-1,1\}^n\times S_n$ chosen uniformly at random, define
$$
Y_j(\e,\pi)\eqdef \frac{\e_ja_{\pi(j)}}{\left(\frac{1}{n}\sum_{s=1}^n |a_s|^p\right)^{\frac{1}{p}}}.
$$
Then $\{Y_j\}_{j=1}^n$ are symmetrically exchangeable random variables (the underlying probability space being the uniform measure on $\{-1,1\}^n\times S_n$), with $\E[|Y_j|^p]=1$. For $k\in \n$, an application of~\eqref{eq:quote JSZ} with $t_1=\ldots=t_k=1$ and $t_{k+1}=\ldots=t_n=0$ therefore yields the following inequality.
\begin{equation}\label{eq:k set square function form}
\frac{(p/\log p)^{-p}}{\binom{n}{k}}\sum_{\substack{S\subset\n\\|S|=k}}\E\left[ \Big|\sum_{j\in S} \e_j a_j\Big|^p\right] \lesssim_p \frac{k}{n}\sum_{j=1}^n |a_j|^p+\left(\frac{k}{n}\right)^{\frac{p}{2}}\left(\sum_{j=1}^n a_j^2\right)^{\frac{p}{2}},
\end{equation}
where in~\eqref{eq:k set square function form}, as well as in~\eqref{eq:jensen comment p}, \eqref{eq:k set quote}, \eqref{eq:hope without scaling} and~\eqref{eq:enflo type 2 used on S} below,  the expectation is with respect to $\e\in \{-1,1\}^n$ chosen uniformly at random. Since, by Jensen's inequality,
\begin{equation}\label{eq:jensen comment p}
\left(\sum_{j=1}^n a_j^2\right)^{\frac{p}{2}}=\left(\E\left[\Big|\sum_{j=1}^n \e_j a_j\Big|^2\right]\right)^{\frac{p}{2}}\le \E\left[\Big|\sum_{j=1}^n \e_j a_j\Big|^p\right],
\end{equation}
it follows from~\eqref{eq:k set square function form} that
\begin{equation}\label{eq:k set quote}
\frac{(p/\log p)^{-p}}{\binom{n}{k}}\sum_{\substack{S\subset\n\\|S|=k}}\E\left[\Big|\sum_{j\in S} \e_j a_j\Big|^p\right] \lesssim_p \frac{k}{n}\sum_{j=1}^n |a_j|^p+\left(\frac{k}{n}\right)^{\frac{p}{2}}\E\left[\Big|\sum_{j=1}^n \e_j a_j\Big|^p\right].
\end{equation}
An inspection of the argument in~\cite{JSZ85} reveals that the term $p/\log p$ in~\eqref{eq:k set quote} is sharp up to a constant factor even in this special case of~\eqref{eq:quote JSZ} (this is true if one requires the validity of~\eqref{eq:k set quote}  for all $k\in \n$, while for a fixed $k$ there might be a better dependence as a function of $k,n, p$).


Our main result, namely Theorem~\ref{thm:main}, is a nonlinear version of~\eqref{eq:k set quote}. By following the reasoning that led to the definition~\eqref{eq:def enflo type intro} of Enflo type, one is tempted to try to establish the validity of the following inequality, which should hold true for every $f:\{-1,1\}^n\to \R$ and for some $\alpha(p)\in (0,\infty)$.
\begin{multline}\label{eq:hope without scaling}
\frac{\alpha(p)}{\binom{n}{k}}\sum_{\substack{S\subset\n\\|S|=k}}\E\Big[\left|f(\e)-f\left(\e_{\n\setminus S}-\e_S\right)\right|^p\Big] \\\le \frac{k}{n}\sum_{j=1}^n\E\Big[\big|f(\e)-f(\e_1,\ldots,\e_{j-1},-\e_j,\e_{j+1},\ldots,\e_n)\big|^p\Big]+
\left(\frac{k}{n}\right)^{\frac{p}{2}}\E\Big[\left|f(\e)-f(-\e)\right|^p\Big].
\end{multline}
Inequality~\eqref{eq:hope without scaling} holds true when $p=2$. Indeed, the fact that the real line has Enflo type $2$ with constant $1$ (as shown by Enflo in~\cite{Enf69}) implies that for every $S\subset\n$ we have
\begin{equation}\label{eq:enflo type 2 used on S}
\E\Big[ \left|f(\e)-f\left(\e_{\n\setminus S}-\e_S\right)\right|^2\Big]\le \sum_{j\in S}\E\Big[\left|f(\e)-f(\e_1,\ldots,\e_{j-1},-\e_j,\e_{j+1},\ldots,\e_n)\right|^2\Big].
\end{equation}
By averaging~\eqref{eq:enflo type 2 used on S} over all of those $S\subset \n$ satisfying $|S|=k$ we see that~\eqref{eq:hope without scaling} holds true when $p=2$, with $\alpha(2)=1$ and even without the final term in the right-hand side of~\eqref{eq:hope without scaling}.

The validity of~\eqref{eq:hope without scaling} for $p=2$ indicates that the reason why scaling is needed for the definition~\eqref{eq:def metric cotype} of metric cotype does not arise in the context of~\eqref{eq:hope without scaling}. However, Proposition~\ref{prop:lower lemma} below shows that scaling is nevertheless necessary in the context of metric $X_p$ inequalities, thus explaining our formulation of Theorem~\ref{thm:main}. Note that the conclusion of Theorem~\ref{thm:main} implies the linear $X_p$ inequality~\eqref{eq:k set quote}. Roughly speaking, this follows by applying~\eqref{eq:Xp in theorem1} to the linear function $f:\Z_{4m}^n\to \R$ given by $f(x)=\sum_{j=1}^n x_j a_j$. However, this reasoning isn't quite accurate because this  $f$ isn't well defined as a function on the discrete torus $\Z_{4m}^n$; for a precise argument see Proposition~\ref{prop:metric to linear} below.

\begin{proposition}[Scaling is necessary]\label{prop:lower lemma}
Fix $p\in (2,\infty)$, $\alpha\in (0,1)$,  $m,n\in \N$ and $k\in
\{1,\ldots,n\}$. Suppose that for every $f:\Z_{2m}^n\to \R$ we have
\begin{multline}\label{eq:Xp alpha in lower lemma}
\frac{\alpha^p}{\binom{n}{k}}\sum_{\substack{S\subset \n\\|S|=
k}}\frac{\E\Big[\left|f\left(x+m\e_S\right)-f(x)\right|^p\Big]}{m^p}\\\le
\frac{k}{n}\sum_{j=1}^n\E\Big[\left|
f(x+e_j)-f(x)\right|^p\Big]+\left(\frac{k}{n}\right)^{\frac{p}{2}}
\E\Big[\left|f\left(x+ \e\right)-f(x)\right|^p\Big],
\end{multline}
where the expectation is with respect to $(x,\e)\in \Z_{2m}^n\times
\{-1,1\}^n$  chosen uniformly at random. Then
\begin{equation}\label{eq:if k big m big}
k\ge \left(\frac{5}{\alpha}\right)^{\frac{2p}{p-2}}\implies m\ge \frac{\alpha}{3}\sqrt{\frac{n}{k}}.
\end{equation}
\end{proposition}
The proof of Proposition~\ref{prop:lower lemma} appears in Section~\ref{sec:prelim}. We conjecture that the dependence of $m$ on $n$ and $k$ that appears
in Proposition~\ref{prop:lower lemma} is sharp, up to the (possibly
$p$-dependent) constant. This is the content of
Conjecture~\ref{conj:best m} below. It seems that in order to prove
Conjecture~\ref{conj:best m} one would need to exploit cancelations
that are more subtle than those that we used to prove
Theorem~\ref{thm:main}.

\begin{conjecture}\label{conj:best m} For every $p\in (2,\infty)$
there exist $\alpha_p\in (0,1)$ and $C_p\in [1,\infty)$ such that if $m,n\in \N$ and $k\in \n$
satisfy $m\ge C_p\sqrt{n/k}$ then for every $f:\Z_{4m}^n\to \R$ we
have
\begin{multline}\label{eq:Xp alpha(p) version}
\frac{\alpha_p}{\binom{n}{k}}\sum_{\substack{S\subset \n\\|S|= k}}
\frac{\E\Big[\left|f\left(x+2m\e_S\right)-f(x)\right|^p\Big]}{m^p}\\\lesssim_p
\frac{k}{n}\sum_{j=1}^n\E\Big[\left|
f(x+e_j)-f(x)\right|^p\Big]+\left(\frac{k}{n}\right)^{\frac{p}{2}}
\E\Big[\left|f\left(x+ \e\right)-f(x)\right|^p\Big],
\end{multline}
where the expectation is with respect to $(x,\e)\in \Z_{4m}^n\times
\{-1,1\}^n$  chosen uniformly at random.
\end{conjecture}
We will see in Section~\ref{consequences of conjecture} below that,
in addition to its intrinsic interest, a positive resolution of
Conjecture~\ref{conj:best m} would have striking consequences in the theory of metric
embeddings. A conjectural convolution inequality (of independent interest) that we formulate in Question~\ref{Q:convolution} below is shown in Proposition~\ref{prop:convolution implies best m} below to imply a positive answer to Conjecture~\ref{conj:best m}.

Before passing to a description of the geometric consequences of Theorem~\ref{thm:main}, we note that the linear $X_p$ inequality~\eqref{eq:k set quote} also has a (much easier) converse~\cite{JMST79}. Specifically, for every $p\in (2,\infty)$ there exists $K(p)\in (0,\infty)$ such that for every $a_1,\ldots,a_n\in \R$ and $k\in \n$ we have
\begin{equation}\label{eq:reverse linear Xp}
\frac{k}{n}\sum_{j=1}^n |a_j|^p+\left(\frac{k}{n}\right)^{\frac{p}{2}}\E\left[\Big|\sum_{j=1}^n \e_j a_j\Big|^p\right]\le\frac{K(p)^p}{\binom{n}{k}}\sum_{\substack{S\subset\n\\|S|=k}}\E\left[\Big|\sum_{j\in S} \e_j a_j\Big|^p\right],
\end{equation}
where the expectation is over $\e\in \{-1,1\}^n$ chosen uniformly at random.  An inspection of the proof of~\eqref{eq:reverse linear Xp}  in~\cite{JMST79} (or in~\cite{JSZ85}) reveals that one can take $K(p)\lesssim \sqrt{p}$ in~\eqref{eq:reverse linear Xp}.  Theorem~\ref{thm:reverse} below is a nonlinear version of~\eqref{eq:reverse linear Xp}. Although we do not have a new geometric application of the reverse metric $X_p$ inequality that appears in Theorem~\ref{thm:reverse}, it is worthwhile to establish it so as to obtain a complete picture of the $X_p$ phenomenon in the metric setting. As a side product, our proof of Theorem~\ref{thm:reverse} yields some new information on metric cotype; see Theorem~\ref{thm:cotype Rademacher version} below and the discussion immediately preceding it.

\begin{theorem}[Reverse metric $X_p$ inequality]\label{thm:reverse} Fix $p\in [2,\infty)$ and $k,m\in \N$ with $m\ge \frac{k^{1/p}}{\sqrt{p}}$. Fix also an integer $n\ge k$. Then for every $f:\Z_{8m}^n\to L_p$ we have
\begin{multline}\label{eq:reverse in theorem1}
\frac{k}{n}\sum_{j=1}^n\frac{\E\left[\left\| f(x+4me_j)-f(x)\right\|_p^p\right]}{m^p}  +\left(\frac{k}{n}\right)^{\frac{p}{2}}
\E\left[\left\|f\left(x+ \e\right)-f(x-\e)\right\|_p^p\right] \\ \lesssim_p \frac{p^{\frac{p}{2}}}{\binom{n}{k}}\sum_{\substack{S\subset \n\\|S|= k}} \E\left[\left\|f\left(x+\e_S\right)-f(x)\right\|_p^p\right],
\end{multline}
where the expectation is with respect to $(x,\e)\in \Z_{8m}^n\times \{-1,1\}^n$  chosen uniformly at random.
\end{theorem}
\subsection{Metric $X_p$ inequalities as obstructions to embeddings}

Theorem~\ref{thm:main} yields a bi-Lipschitz invariant that can be used to obtain new nonembeddability results which we shall now describe.

\subsubsection{Snowflakes} Fix $p,q\in [1,\infty)$. Sharp restrictions on those $\theta\in (0,1]$ for which the $\theta$-snowflake of $L_q$ admits a bi-Lipschitz embedding into $L_p$ follow from the theory of metric type and cotype when either $q\in [1,2]$ and $p\ge q$, or $q\in [2,\infty)$ and $p\le q$; see~\cite{Lov88,MN06,GMN11}. Here we obtain, as a consequence of Theorem~\ref{thm:main}, the first such result when $2<q<p$.

\begin{theorem}[$L_q$ snowflakes in $L_p$]\label{thm:our snowflake} For every $2<q<p$ there exists $\d(p,q)>0$ such that if $\theta\in (0,1)$ is such that the metric space $(L_q,\|x-y\|_q^\theta)$ admits a bi-Lipschitz embedding into $L_p$ then necessarily $\theta \le 1-\d(p,q)$. Specifically, $\theta$ must satisfy
\begin{equation}\label{eq:theta upper in theorem}
\theta\le \frac{2q(p-q)+q^2(p-1)(p-2)}{2p^2(q-2)}\left(\sqrt{1+\frac{4p(p-2)(q-2)}{(pq-3q+2)^2}}-1\right)
\le1-\frac{(p-q)(q-2)}{2p^3}.
\end{equation}
\end{theorem}
It was shown in~\cite[Remark~5.10]{MN04} that for $2< q< p$ that the $(q/p)$-snowflake of $L_q$ is isometric to a subset of $L_p$. We conjecture that this is sharp, i.e., that the upper bound on $\theta$ that appears in~\eqref{eq:theta upper in theorem} can be improved to $\theta\le q/p$.

\begin{conjecture}\label{conj:q/p}
Suppose that $2<q<p$ and  $\theta\in (0,1)$ is such that the metric space $(L_q,\|x-y\|_q^\theta)$ admits a bi-Lipschitz embedding into $L_p$. Then necessarily $\theta \le q/p$.
\end{conjecture}

In fact, when  $2< q\le p$, we ask whether or not $L_q$ has a {\em unique} snowflake that admits a bi-Lipschitz embedding into $L_p$. If true, this would be manifestly different than the case $1\le q\le p\le 2$, where it is known~\cite{BDD65} (see also~\cite{WW75})  that the metric space  $(L_q,\|x-y\|_q^\theta)$ admits an isometric embedding into $L_p$ for every $0<\theta\le q/p$.
\begin{question}[Uniqueness of snowflakes]\label{Q:uniqueness of snowflakes} Suppose that $2<q\le p$ and $\theta\in (0,1)$. Is it true that if the metric space $(L_q,\|x-y\|_q^\theta)$ admits a bi-Lipschitz embedding into $L_p$ then necessarily $\theta = q/p$?
\end{question}
The case $q=p$ of Question~\ref{Q:uniqueness of snowflakes}  is a well-known problem that has been open for many years (though apparently not stated explicitly in the literature): is it true that if $p\in (2,\infty)$ then for no $\theta\in (0,1)$ the metric space $(L_p,(\|x-y\|_p^\theta)$ admits a bi-Lipschitz embedding into $L_p$? Related results appear in~\cite[Section~5]{MN04}.

\begin{remark}
The analogue of Conjecture~\ref{conj:q/p} for sequence spaces has a positive answer. Indeed, a combination of~\cite[Cor.~2.19]{Bau12} and \cite[Cor.~2.23]{Bau12} shows that for every $1\le q\le p<\infty$, if $\theta\in (0,1]$ is such that the metric space  $(\ell_q,\|x-y\|_q^\theta)$ admits a bi-Lipschitz embedding into $\ell_p$ then necessarily $\theta\le q/p$. The proof of this result in~\cite{Bau12} relies on an infinite dimensional argument of~\cite{KR08} that is specific to sequence spaces (the above statement from~\cite{Bau12}  becomes false if $q\in [1,2]$, $p\in (2,\infty)$ and $\ell_p$ is replaced by $L_p$). Conversely, in~\cite{AB12} (see also~\cite[Exercise~1.61]{Ost13}) it was shown that for every $1\le q\le p<\infty$ the $(q/p)$-snowflake of $\ell_q$ does admit a bi-Lipschitz embedding into $\ell_p$.
\end{remark}

\subsubsection{Grids} Recall that for $q\in [1,\infty)$ and $m,n\in \N$  the integer grid $\{1,\ldots,m\}^n$, equipped with the metric inherited from $\ell_q^n$, is denoted $[m]_q^n$. Theorem~\ref{thm:our grid distortion} below, which is a consequence of Theorem~\ref{thm:main}, contains the best-known lower bound on $c_p([m]_q^n)$ when $2<q<p$, thus yielding another quantitative version of the fact that $L_q$ does not admit a bi-Lipschitz embedding into $L_p$.

\begin{theorem}[$L_p$ distortion of $L_q$ grids]\label{thm:our grid distortion}  For every $p\in (2,\infty)$ there exists $\alpha_p\in (0,\infty)$ such that for every $q\in (2,p)$ and $m,n\in \N$ we have
\begin{equation}\label{eq:our lower grid}
c_p\!\left([m]_q^n\right)\gtrsim \alpha_p\left(\min\left\{m^{\frac{q(p-2)}{q(p-2)+p-q}},n\right\}\right)^{\frac{(p-q)(q-2)}{q^2(p-2)}}.
\end{equation}
In particular,
\begin{equation}\label{eq:when our m is large enough}
m\ge n^{1+\frac{p-q}{q(p-2)}}\implies c_p\left([m]_q^n\right)\ge \alpha_pn^{\frac{(p-q)(q-2)}{q^2(p-2)}}\gtrsim \alpha_pc_p\!\left(\ell_q^n\right).
\end{equation}
\end{theorem}

The fact that the lower bound in~\eqref{eq:our lower grid} becomes weaker for smaller $m$ is necessary, as exhibited by the following embedding from~\cite{MN06}. First, let $G_{2,p}:L_2\to L_p$ be an isometric embedding of $L_2$ into $L_p$. By a classical theorem of Schoenberg~\cite{Sch38} (see also~\cite{WW75}) there exists an isometric embedding of the $(2/q)$-snowflake of $\ell_2^n$ into $L_2$, i.e., there exists $\psi_q^n:\ell_2^n\to L_2$ such that
$$
\forall\, x,y\in \ell_2^n,\qquad \|\psi_q^n(x)-\psi_q^n(y)\|_2=\|x-y\|_2^{\frac{2}{q}}.
$$
Finally, let $I_{q\to 2}^n:\ell_q^n\to \ell_2^n$ be the identity mapping, and define\footnote{The superscript in the notation $J^S_{(q\to p;n)}(\cdot)$ refers to Schoenberg.}
\begin{equation}\label{eq:def JS}
J_{(q\to p;n)}^S\eqdef G_{2,p}\circ \psi_q^n\circ I_{q\to 2}^n: \ell_q^n\to L_p.
\end{equation}
As argued in~\cite{MN06}, the distortion of the restriction of  $J^S_{(q\to p;n)}$ to $[m]_q^n$ satisfies
$$
\dist\!\left(\left.J^S_{(q\to p;n)}\right|_{[m]_q^n}\right)\le m^{1-\frac{2}{q}}.
$$
Recalling the definition of the embedding $J^R_{(q\to p;n)}$ in~\eqref{eq:def JR}, we therefore have
\begin{equation}\label{eq:better of two embeddings}
c_p\!\left([m]_q^n\right)\le \min\left\{\dist\!\left(\left.J^R_{(q\to p;n)}\right|_{[m]_q^n}\right),\dist\!\left(\left.J^S_{(q\to p;n)}\right|_{[m]_q^n}\right)\right\}\lesssim \min\left\{n^{\frac{(p-q)(q-2)}{q^2(p-2)}},m^{1-\frac{2}{q}}\right\}.
\end{equation}
We conjecture that~\eqref{eq:better of two embeddings} is asymptotically sharp up to constant factors that depend only on $p, q$.
\begin{conjecture}\label{conj:sharp grid} For $2<q<p$ and $m,n\in \N$, the better of the embeddings $J^R_{(q\to p;n)}$ and $J^S_{(q\to p;n)}$ appearing in~\eqref{eq:def JR} and~\eqref{eq:def JS}, respectively, is the best possible bi-Lipschitz embedding of the $L_q$ integer grid  $[m]_q^n$ into $L_p$. Equivalently, $c_p([m]_q^n)$ is bounded from above and from below by positive constants that may depend only on $p$ and $q$ times the quantity
\begin{equation}\label{eq:sharp distortion grid transition}
\min\left\{n^{\frac{(p-q)(q-2)}{q^2(p-2)}},m^{1-\frac{2}{q}}\right\}.
\end{equation}
In particular, there exists $\eta(p,q)\in (0,1)$ such that
\begin{equation*}
m\ge n^{\frac{p-q}{q(p-2)}}\implies c_p\!\left([m]_q^n\right)\ge \eta(p,q) c_p(\ell_q^n),
\end{equation*}
yet
\begin{equation*}
m=o\!\left(n^{\frac{p-q}{q(p-2)}}\right)\implies c_p\!\left([m]_q^n\right)= o\!\left(c_p(\ell_q^n)\right)\ \ (\mathrm{as}\ n\to\infty).
\end{equation*}
\end{conjecture}
An affirmative answer to Conjecture~\ref{conj:sharp grid} would imply that if the linear embedding $J^R_{(q\to p;n)}$ of $\ell_q^n$ into an appropriate Rosenthal $X_p(\omega)$ space fails to yield the best possible bi-Lipschitz embedding of $[m]_q^n$ into $L_p$ then (up to constant factors that are independent of $m,n$), the best possible way to embed $[m]_q^n$ into $L_p$ would be to embed it into $L_2$ (ignoring the fact that we are seeking an embedding into the larger space $L_p$), via the (highly nonlinear) Schoenberg embedding $\psi_q^n$. Admittedly, if true, this phenomenon would be quite exotic, but we conjecture that it indeed occurs partially because it is a consequence of Conjecture~\ref{conj:best m}, as we shall see in Section~\ref{consequences of conjecture} below.

\begin{remark}
There are also interesting open problems related to embeddings of $[m]_p^n$ into $L_q$ when $p>q>2$. Specifically, by combining the upper bound in~\cite{MN06} with the metric cotype-based lower bound in~\cite{MN08}, we see that
\begin{equation}\label{eq:from cotype}
\frac{1}{\sqrt{q}}\cdot \min\left\{n^{\frac{1}{q}-\frac{1}{p}},m^{1-\frac{q}{p}}\right\}\lesssim c_q([m]_p^n)\le \min\left\{n^{\frac{1}{q}-\frac{1}{p}},m^{1-\frac{2}{p}}\right\}.
\end{equation}
The bounds in~\eqref{eq:from cotype} match only when $q=2$, and it remains open to evaluate  $c_q([m]_p^n)$ up to constant factors that are independent of $m,n$. An inspection of the argument in~\cite{MN06} reveals that the lower bound on $c_q([m]_p^n)$ in~\eqref{eq:from cotype} would be sharp (up to constant factors that may depend only on $p,q$) if the $(q/p)$-snowflake of $L_q$ admitted a bi-Lipschitz embedding into $L_q$. When $q=2$ this is indeed the case due to the theorem of Schoenberg that was quoted above, but for $q>2$ a positive answer to Question~\ref{Q:uniqueness of snowflakes} (see also the paragraph immediately following Question~\ref{Q:uniqueness of snowflakes}) would imply that no nontrivial snowflake of $L_q$ admits a bi-Lipschitz embedding into $L_q$. In the spirit of Conjecture~\ref{conj:sharp grid}, one is tempted to ask whether or not the upper bound on $c_q([m]_p^n)$ in~\eqref{eq:from cotype} is asymptotically sharp, i.e., if also in this setting it is best to embed $[m]_p^n$ into $L_q$ via an appropriate embedding into the smaller space $L_2$. However, if this were true then one would need to find a better lower bound on $c_q([m]_p^n)$ than what follows from the fact that $L_q$ has metric cotype $q$. For this reason, at present we do not have a concrete conjecture as to the sharp asymptotics of $c_q([m]_p^n)$ when $p>q>2$.
\end{remark}

\subsubsection{Consequences of Conjecture~\ref{conj:best m}}\label{consequences of conjecture}
The following theorem asserts that Conjecture~\ref{conj:best m} implies a positive solution of Conjecture~\ref{conj:q/p} and Conjecture~\ref{conj:sharp grid}. Thus, obtaining the conjecturally sharp value of $m$ in the metric $X_p$ inequality of Theorem~\ref{thm:main}, in addition to its intrinsic analytic interest, would yield striking nonembeddability results. As we mentioned earlier, in Section~\ref{sec:best m conv} we present a concrete convolution inequality (that is interesting on its own right) and prove that it implies an affirmative answer to Conjecture~\ref{conj:best m}, and hence also to Conjecture~\ref{conj:q/p} and Conjecture~\ref{conj:sharp grid}.

\begin{theorem}\label{thm:best m implies sharp nonembeddability} If Conjecture~\ref{conj:best
m} holds true then for every $2<q<p$ and $\theta\in (0,1)$,
\begin{equation}\label{eq:sharp snowflake bound}
c_p\!\left(L_q,\|x-y\|_q^\theta\right)<\infty\implies \theta\le \frac{q}{p}.
\end{equation}
Moreover, for every $m,n\in \N$ the $L_p$ distortion of the $L_q$ grid $[m]_q^n$ is bounded from above and below by a constant that may depend only on $p$ times the quantity appearing in~\eqref{eq:sharp distortion grid transition}.
\end{theorem}

\subsection{$X_p$ metric spaces?}\label{sec:metric} For $p\in (0,\infty)$, by pursuing the Ribe program in light of Theorem~\ref{thm:main}, one arrives at Definition~\ref{def:Xpmetric}  of when  a metric space $(X,d_X)$ is an $X_p$ metric space. One would then want to investigate the structure of such metric spaces, motivated in part by analogies from the linear theory. However, in contrast to previous successful steps in the Ribe program, in the present setting the linear theory of $X_p$ spaces hasn't been studied yet, and it therefore seems to be more natural to first understand what makes a Banach space an $X_p$ Banach space. Specifically, say that a Banach space $(X,\|\cdot\|_X)$ is an $X_p$ Banach space if for every $n\in \N$ and $k\in \n$, every  $v_1,\ldots,v_n\in X$ satisfy
\begin{equation*}
\frac{1}{\binom{n}{k}}\sum_{\substack{S\subset\n\\|S|=k}}\E\left[\Big\|\sum_{j\in S} \e_j v_j\Big\|_X^p\right] \lesssim_X \frac{k}{n}\sum_{j=1}^n \|v_j\|_X^p+\left(\frac{k}{n}\right)^{\frac{p}{2}}\E\left[\Big\|\sum_{j=1}^n \e_j v_j\Big\|_X^p\right],
\end{equation*}
where the expectation is over $\e\in \{-1,1\}^n$ chosen uniformly at random.  Being an $X_p$ Banach space is clearly a local property. Our proof of Theorem~\ref{thm:main} shows that a Banach space is an $X_p$ metric space if and only if it is an $X_p$ Banach space, thus completing the Ribe program in this setting.

For $p>2$, it seems that the only Banach spaces that were previously known to be $X_p$ Banach spaces were those that are isomorphic to subspaces of $L_p$. However, there exist separable $X_p$ Banach spaces that are not isomorphic to a subspace of $L_p$. In Section~\ref{sec:appendix} we prove that for $p\in [2,\infty)$ the Schatten $p$ trace class $S_p$ is an $X_p$ Banach space. The fact that $S_p$ is not isomorphic to a subspace of $L_p$ was proved in~\cite{McC67} (see also~\cite{Pis78}). Obtaining a satisfactory understanding of those Banach spaces that are $X_p$ spaces remains an interesting, though probably quite difficult, research challenge.

Since $S_p$ is an $X_p$ Banach space,  our work here shows that it is also an $X_p$ metric space. The nonembeddability results that were stated above for embeddings  into  $L_p$ therefore hold true for embeddings into $S_p$ as well. In the setting of $S_p$, these nonembeddability results are new even in the linear category. It was known that for $2<q<p$ the Banach--Mazur distance of $\ell_q^n$ to any subspace of $S_p$ must tend to $\infty$ with $n$: this follows from the non-commutative Kadec--Pe\l czy\'nski result in~\cite{Suk96}; see also Theorem~10.7 in~\cite{PX03}. The literature gives no information on the rate at which $c_{S_p}(\ell_q^n)$ tends to infinity with $n$ (extracting quantitative estimates from the proof in~\cite{Suk96}, if at all possible, would probably require significant effort and yield weak bounds). Here we see that  $c_{S_p}(\ell_q^n)$ is asymptotically $n^{(p-q)(q-2)/(q^2(p-2))}$, up to constant factors that may depend only on $p, q$.

\section{Preliminaries}\label{sec:prelim}

Here we establish some initial facts and prove some of the simpler statements that were presented in the Introduction. The results of the present section will not be used for the proofs of Theorem~\ref{thm:main} and its consequences, so they could be skipped on first reading.

We shall start with the proof of Proposition~\ref{prop:lower lemma}, i.e., that scaling is needed for the metric $X_p$ inequality of Theorem~\ref{thm:main} to hold true.

\begin{proof}[Proof of Proposition~\ref{prop:lower lemma}] We shall use here the notation that was introduced in the statement of Proposition~\ref{prop:lower lemma}. Since
$\ell_2$ embeds isometrically into $L_p$, by~\cite[Lem.~5.2]{MN04}
there exists $F:\Z_{2m}^n\to L_p$ such that for every distinct
$x,y\in \Z_{2m}^n$ we have
\begin{equation}\label{eq:truncation}
1\le \frac{\|F(x)-F(y)\|_p}{\min\left\{2\sqrt{k},\sqrt{\sum_{j=1}^n
\left|e^{\pi i(x_j-y_j)/m}-1\right|^2}\right\}}\le
2.
\end{equation}
By integrating~\eqref{eq:Xp alpha in lower lemma} we see that
\begin{multline}\label{eq:Xp alpha Lower in Lp}
\frac{\alpha^p}{\binom{n}{k}}\sum_{\substack{S\subset \n\\|S|=
k}}\frac{\E\left[\left\|F\left(x+m\e_S\right)-F(x)\right\|_p^p\right]}{m^p}\\\le
\frac{k}{n}\sum_{j=1}^n\E\left[\left\|
F(x+e_j)-F(x)\right\|_p^p\right]+\left(\frac{k}{n}\right)^{\frac{p}{2}}
\E\left[\left\|F\left(x+ \e\right)-F(x)\right\|_p^p\right].
\end{multline}
It follows from~\eqref{eq:truncation} that if $S\subseteq \n$
satisfies $|S|=k$ then $\|F(x+m\e_S)-F(x)\|_p\ge 2\sqrt{k}$ for
every $x\in \Z_{2m}^n$. Also, the elementary inequality $|e^{\pi
i/m}-1|^2\le \pi^2/m^2$ implies that for every $j\in \n$ we have
$\|F(x+e_j)-F(x)\|_p\le 2\pi/m$, and for every $(x,\e)\in
\Z_{2m}^n\times  \{-1,1\}^n$ we have $\|F(x+\e)-F(x)\|_p\le
4\sqrt{k}$. In conjunction with~\eqref{eq:Xp alpha Lower in Lp}
these estimates show that
$$
\frac{2^p \alpha^p k^{\frac{p}{2}}}{m^p}\le \frac{2^{p}\pi^pk}{m^p}+
\frac{(4k)^p}{n^{\frac{p}{2}}},
$$
which yields the desired implication~\eqref{eq:if k big m big}.
\end{proof}

Next, we shall check the validity of~\eqref{eq:distortion of JR}, i.e., evaluate the distortion of the mapping $J^R_{(q\to p;n)}$ given in~\eqref{eq:def JR}. This is a known (and easy) statement which is included here only because we could not locate a clean reference for it.

\begin{proof}[Proof of~\eqref{eq:distortion of JR}] The definition~\eqref{eq:def JR} implies that for every $x\in \ell_q^n$ we have
$$
\left\|J^R_{(q\to p;n)}(x)\right\|_{(\ell_p^n\oplus\ell_2^n)_p}^p=n^{\frac{p}{2}}\|x\|_p^p+n^{\frac{p}{q}}\|x\|_2^p.
$$
Consequently, it suffices to show that for every $x\in \ell_q^n$ we have
\begin{equation}\label{eq:lagrange goal}
\frac{n^{\frac{p}{2}}\|x\|_q^p}{2^{\frac{p}{q}}n^{\frac{p(p-q)(q-2)}{q^2(p-2)}}}\le n^{\frac{p}{2}}\|x\|_p^p+n^{\frac{p}{q}}\|x\|_2^p\le 2n^{\frac{p}{2}}\|x\|_q^p.
\end{equation}
The rightmost inequality in~\eqref{eq:lagrange goal} is an immediate consequence of the estimates $\|x\|_2\le n^{\frac12-\frac{1}{q}}\|x\|_q$ and  $\|x\|_p\le \|x\|_q$, which hold true  because $2<q<p$.

Let $x\in \ell_q^n$  with $\|x\|_q=1$ be such that $n^{\frac{p}{2}}\|x\|_p^p+n^{\frac{p}{q}}\|x\|_2^p$ is minimal. We may also assume that the number of nonzero entries of $x$ is minimal, and that $x_1,\ldots,x_k>0$ and $x_{k+1}=\ldots=x_n=0$ for some $k\in \n$. Hence, there exists (a Lagrange multiplier) $\lambda\in \R$ such that
\begin{equation}\label{eq:lagrange identity}
\forall\, j\in \{1,\ldots,k\},\qquad n^{\frac{p}{2}}x_j^{p-1}+n^{\frac{p}{q}}\|x\|_2^{p-2}x_j=\lambda x_j^{q-1}.
\end{equation}
For $s\in [0,\infty)$ write $\psi(s)\eqdef n^{\frac{p}{2}}s^{p-2}-\lambda s^{q-2}+n^{\frac{p}{q}}\|x\|_2^{p-2}$. Since $p,q>2$ we have $\psi(0)>0$, and since $p>q$ we have $\lim_{s\to\infty} \psi(s)=\infty$. It follows from~\eqref{eq:lagrange identity} that $\lambda>0$, and therefore there is a unique $s_0\in (0,\infty)$ for which $\psi'(s_0)=0$. This means that $\psi$ starts at a positive value, decreases on $(0,s_0)$, and then increases to $\infty$. Consequently, there exist $a,b\in (0,\infty)$ such that $\psi(s)=0\implies s\in \{a,b\}$ for  every $s\in (0,\infty)$. Since  by~\eqref{eq:lagrange identity} we have $\psi(x_j)=0$ for every $j\in \{1,\ldots,k\}$, it follows that there exists $S\subset\{1,\ldots k\}$ such that $x_j=a\1_S(j)+b\1_{\{1,\ldots,k\}\setminus S}(j)$ for all $j\in \{1,\ldots,k\}$. Since $\|x\|_q=1$, we may assume without loss of generality that $a^q|S|\ge 1/2$, i.e., that $a\ge 1/(2|S|)^{1/q}$. Consequently,
$$
n^{\frac{p}{2}}\|x\|_p^p+n^{\frac{p}{q}}\|x\|_2^p\ge n^{\frac{p}{2}}|S|a^p+n^{\frac{p}{q}}|S|^{\frac{p}{2}}a^p\ge \frac{n^{\frac{p}{2}}}{2^{\frac{p}{q}}|S|^{\frac{p}{q}-1}}+
\frac{n^{\frac{p}{q}}|S|^{\frac{p}{2}-\frac{p}{q}}}{2^{\frac{p}{q}}}\ge 2^{-\frac{p}{q}}n^{\frac{p}{2}-\frac{p(p-q)(q-2)}{q^2(p-2)}},
$$
where the last step follows by computing the minimum of $n^{\frac{p}{2}}/s^{\frac{p}{q}-1}+n^{\frac{p}{q}}s^{\frac{p}{2}-\frac{p}{q}}$ over $s\in (0,\infty)$.
\end{proof}

In the present work, Banach spaces are assumed to be over real scalars unless stated otherwise. However, it will sometimes be notationally convenient to work with complex Banach spaces, and in fact all the results presented below hold true for Banach spaces over the complex numbers as well. This follows from a straightforward complexification argument. Specifically, given a real Banach space $(Z,\|\cdot\|_Z)$  and $p\in [1,\infty)$ denote by $Z_p(\C)$ the following $p$-complexification of $Z$. As a vector space, $Z_p(\C)=Z\times Z$. As usual,  we consider $Z_p(\C)$ as a vector space over $\C$ by setting $(a+bi)(u,v)=(au-bv,av+bu)$ for every $u,v\in Z$ and $a,b\in \R$. The norm on $Z_p(\C)$ is given by
\begin{equation}\label{eq:def complexification norm}
\forall(u,v)\in Z\times Z,\qquad \|(u,v)\|_{Z_p(\C)} \eqdef \left(\int_0^{2\pi} \| (\cos \theta)u-(\sin\theta)v\|_Z^pd\theta \right)^{\frac{1}{p}}.
\end{equation}
This turns $(Z_p(\C),\|\cdot\|_{Z_p(\C)})$ into a Banach space over
the complex numbers, which is isometric as a real Banach space to a
subspace of $L_p([0,2\pi],Z)$. For every $z\in Z$ we have
\begin{equation}\label{eq:z0}
\|(z,0)\|_{Z_p(\C)}^p=\|z\|_Z^p\int_0^{2\pi} |\cos\theta|^pd\theta= \frac{4\sqrt{\pi}\Gamma\left(\frac{p}{2}+\frac{1}{2}\right)}{\Gamma\left(\frac{p}{2}+1\right)}\|z\|_Z^p.
\end{equation}
Hence, by considering an appropriate rescaling of the first coordinate of elements of $Z_p(\C)$, we see that $Z$ is isometric to a subspace of $Z_p(\C)$. Since $Z_p(\C)$ is a subspace of $L_p([0,2\pi],Z)$, all properties that are closed under $\ell_p$ sums are inherited by $Z_p(\C)$ from $Z$.

The final matter that will be treated in the present section is to show that the metric $X_p$ inequality of Theorem~\ref{thm:main} implies the linear $X_p$ inequality~\eqref{eq:k set quote}. We shall show this in the context of general Banach spaces, i.e., if a Banach space is an $X_p$ metric space then it is also an $X_p$ Banach space. The converse of this assertion, i.e., that an $X_p$ Banach space is also an $X_p$ metric space, follows from the proof of Theorem~\ref{thm:main}  that can be found in Section~\ref{sec:proof of main}.

\begin{proposition}[Metric $X_p$ inequalities imply linear $X_p$ inequalities]\label{prop:metric to linear}
Let $(Z,\|\cdot\|_Z)$ be a Banach space. Fix $p\in [2,\infty)$ and $\gamma\in (0,1)$. Fix also $m,n\in \N$ and $k\in \n$. Suppose that for every $f:\Z_{2m}^n\to Z$ we have
\begin{multline}\label{eq:Xp gamma in lemma}
\frac{\gamma}{2^n\binom{n}{k}}\sum_{\substack{S\subset
\n\\|S|=k}}\sum_{\e\in \{-1,1\}^n}\sum_{x\in
\Z_{2m}^n}\frac{\|f(x+m\e_S)-f(x)\|_Z^p}{m^p}\\\le
\frac{k}{n}\sum_{j=1}^n \sum_{x\in \Z_{2m}^n}
\|f(x+e_j)-f(x)\|_Z^p+\frac{(k/n)^{\frac{p}{2}}}{2^n}\sum_{\e\in
\{-1,1\}^n} \sum_{x\in \Z_{2m}^n} \|f(x+\e)-f(x)\|_Z^p.
\end{multline}
Then for every $z_1,\ldots,z_n\in Z$ we have
\begin{equation}\label{eq:rademacher some from metric}
\frac{(2/\pi)^{2p}\gamma}{2^n\binom{n}{k}}\sum_{\substack{S\subset
\n\\|S|=k}}\sum_{\e\in \{-1,1\}^n} \left\|\sum_{j\in S}\e_j
z_j\right\|_Z^p\le
\frac{k}{n}\sum_{j=1}^n\|z_j\|_Z^p+\frac{(k/n)^{\frac{p}{2}}}{2^n}\sum_{\e\in
\{-1,1\}^n}\left\|\sum_{j=1}^n \e_j z_j\right\|_Z^p.
\end{equation}
\end{proposition}

\begin{proof} Since~\eqref{eq:Xp gamma in lemma} holds true in $Z$, it also holds true in its $p$-complexification $Z_p(\C)$. Fixing $z_1,\ldots,z_n\in Z$ and $\d\in \{-1,1\}^n$, apply~\eqref{eq:Xp gamma in lemma} to the function $f_\d:\Z_{2m}^n\to Z_p(\C)$ given by
$$
\forall\, x\in \Z_{2m}^n,\qquad f_\d(x)\eqdef \sum_{j=1}^n \d_je^{\frac{\pi ix_j}{m}}(z_j,0)\in Z\times Z.
$$
By averaging the resulting inequality over $\d\in \{-1,1\}^n$, we deduce that
\begin{align}\label{eq:use Xp on complexification}
\frac{2^p\gamma}{2^n\binom{n}{k}}&\sum_{\substack{S\subset \n\\|S|=k}}\sum_{\d\in \{-1,1\}^n}\sum_{x\in \Z_{2m}^n}\frac{\left\|\sum_{j\in S}\d_je^{\frac{\pi ix_j}{m}}(z_j,0)\right\|_{Z_p(\C)}^p}{m^p}\nonumber
\\\nonumber&\le \frac{k(2m)^n}{n} \sum_{j=1}^n\left|1-e^{\frac{\pi i}{m}}\right|^p\cdot\|(z_j,0)\|_{Z_p(\C)}^p\\
&\qquad +\frac{(k/n)^{\frac{p}{2}}}{4^n}\sum_{x\in \Z_{2m}^n}\sum_{\e,\d\in \{-1,1\}^n}\left\|\sum_{j=1}^n \d_j\left(e^{\frac{\pi i(x_j+\e_j)}{m}}-e^{\frac{\pi ix_j}{m}}\right) (z_j,0)\right\|_{Z_p(\C)}^p,
\end{align}
where for the left-hand side of~\eqref{eq:use Xp on complexification} we used the fact that $e^{\frac{\pi i(x+m\sigma)}{m}}-e^{\frac{\pi ix}{m}}=-2e^{\frac{\pi ix}{m}}$ for every $\sigma\in \{-1,1\}$ and $x\in \Z_{2m}$.

Recalling the definition~\eqref{eq:def complexification norm} of the norm of $Z_p(\C)$, for every $S\subset \n$ we have
\begin{align}
\sum_{\d\in \{-1,1\}^n}\sum_{x\in \Z_{2m}^n}\left\|\sum_{j\in S}\d_je^{\frac{\pi ix_j}{m}}(z_j,0)\right\|_{Z_p(\C)}^p&=\sum_{x\in \Z_{2m}^n}\int_0^{2\pi}\sum_{\d\in \{-1,1\}^n}\left\|\sum_{j\in S} \d_j\cos\left(\theta+\frac{\pi x_j}{m}\right)z_j\right\|_Z^pd\theta\nonumber\\
&\nonumber=\sum_{x\in \Z_{2m}^n}\int_0^{2\pi}\sum_{\d\in \{-1,1\}^n}\left\|\sum_{j\in S} \d_j\left|\cos\left(\theta+\frac{\pi x_j}{m}\right)\right|z_j\right\|_Z^pd\theta\\
&\ge2\pi (2m)^n\sum_{\d\in \{-1,1\}^n} \left\|\sum_{j\in S} \frac{\d_j}{2\pi}\left(\int_0^{2\pi}\left|\cos\theta\right|d\theta\right)z_j\right\|_Z^p\label{eq:jensen for contraction}\\
&=\frac{2^{p+1}(2m)^n}{\pi^{p-1}}\sum_{\d\in \{-1,1\}^n} \left\|\sum_{j\in S} \d_jz_j\right\|_Z^p,
\label{eq:to lower bound for contraction}
\end{align}
where in~\eqref{eq:jensen for contraction} we used Jensen's inequality.

To bound the first term in the right-hand side of~\eqref{eq:use Xp on complexification}, use the fact that $|1-e^{i\theta}|\le \theta$ for every $\theta\in [0,\pi]$, and the identity~\eqref{eq:z0} to get
\begin{equation}\label{eq:use z0 identity}
 \sum_{j=1}^n\left|1-e^{\frac{\pi i}{m}}\right|^p\cdot\|(z_j,0)\|_{Z_p(\C)}^p\le \frac{\pi^p}{m^p}\left(\int_0^{2\pi}|\cos\theta|^pd\theta\right)\sum_{j=1}^n \|z_j\|_Z^p\le \frac{\pi^{p+1}}{m^p}\sum_{j=1}^n \|z_j\|_Z^p,
\end{equation}
where we used the fact that, since $p\ge 2$, we have $\int_0^{2\pi}|\cos\theta|^pd\theta\le \int_0^{2\pi}\cos^2\theta d\theta=\pi$. To bound the second term in the right-hand side of~\eqref{eq:use Xp on complexification}, recall the contraction principle (see~\cite[Thm.~4.4]{LT91}), which asserts that
for every $a_1,\ldots,a_n\in \R$ we have
\begin{equation}\label{eq:contraction}
\sum_{\d\in \{-1,1\}^n} \left\|\sum_{j=1}^n a_j\d_j z_j \right\|_Z^p\le \left(\max_{j\in \n}|a_j|^p\right)\sum_{\d\in \{-1,1\}^n} \left\|\sum_{j=1}^n\d_j z_j \right\|_Z^p.
\end{equation}
Hence, for every $x\in \Z_{2m}^n$ and $\e\in \{-1,1\}^n$ we have
\begin{align}
\nonumber&\sum_{\d\in \{-1,1\}^n}\left\|\sum_{j=1}^n
\d_j\left(e^{\frac{\pi i(x_j+\e_j)}{m}}-e^{\frac{\pi ix_j}{m}}\right) (z_j,0)\right\|_{Z_p(\C)}^p\\\nonumber&=
\int_0^{2\pi} \sum_{\d\in \{-1,1\}^n}
\left\|\sum_{j=1}^n\left(\cos\left(\theta+\frac{\pi x_j}{m}+\frac{\pi\e_j}{m}\right)-
\cos\left(\theta+\frac{\pi x_j}{m}\right)\right)\d_jz_j \right\|_Z^pd\theta\\
&\le 2\pi\left(\max_{\theta\in [0,2\pi]}\max_{j\in \n}\left|\cos\left(\theta+\frac{\pi x_j}{m}+\frac{\pi\e_j}{m}\right)-
\cos\left(\theta+\frac{\pi x_j}{m}\right)\right|^p\right)\sum_{\d\in \{-1,1\}^n}\left\|\sum_{j=1}^n\d_jz_j \right\|_Z^p\label{eq:use contraction}\\
&\le \frac{2\pi^{p+1}}{m^p}\sum_{\d\in \{-1,1\}^n}\left\|\sum_{j=1}^n\d_jz_j \right\|_Z^p,\label{eq:use contraction principle}
\end{align}
where~\eqref{eq:use contraction} uses~\eqref{eq:contraction} and~\eqref{eq:use contraction principle} uses
$\left|\cos\left(\alpha\pm\pi/m\right)-\cos\alpha\right|=
\left|\int_{\alpha}^{\alpha\pm\pi/m}\sin tdt\right|\le \pi/m$, which holds true for every $\alpha\in [0,2\pi]$. The
desired inequality~\eqref{eq:rademacher some from metric}
follows by combining~\eqref{eq:use Xp on complexification},
\eqref{eq:to lower bound for contraction}, \eqref{eq:use z0
identity}, \eqref{eq:use contraction principle}.
\end{proof}

\section{Nonembeddability}

Here we assume for the moment the validity of Theorem~\ref{thm:main}, whose proof appears in Section~\ref{sec:proof of main}, and proceed to deduce its geometric consequences that were stated in the Introduction. Namely, we will prove here Theorem~\ref{thm:our snowflake}, Theorem~\ref{thm:our grid distortion} and Theorem~\ref{thm:best m implies sharp nonembeddability}.

\begin{proof}[Proof of Theorem~\ref{thm:our snowflake}] We first make some preparatory elementary estimates that explain the origin of the quantities that appear in~\eqref{eq:theta upper in theorem}. Define $\psi_{p,q}:\R\to \R$ by
\begin{align}
\psi_{p,q}(t)&\eqdef \frac{3t p}{q}-3+\left(t p+2-\frac{2t p}{q}-p\right)\left(1+\frac{t p-q}{q(p-2)}\right)\nonumber
=\frac{p^2(q-2)}{q^2(p-2)}t^2+\frac{p(pq-3q+2)}{q(p-2)}t-p.
\end{align}
Then for every $s\in (0,1)$ we have
\begin{multline*}
\frac{q^2(p-2)}{p}\cdot \psi_{p,q}(1-s)-(p-q)(q-2)=-(2pq+2q-4p+pq^2-3q^2)s+p(q-2)s^2\\> -(2pq+2q-4p+pq^2)s> -(2p^2-2p+p^3)s> -2p^3s.
\end{multline*}
Hence $\psi_{p,q}(1-(p-q)(q-2)/(2p^3))> 0$. Note that $\psi_{p,q}(0)=-p<0$ and $\psi_{p,q}(q/p)=-(p-q)<0$. Since $\psi$ is quadratic with $\lim_{t\to \pm \infty}\psi_{p,q}(t)=\infty$, it follows that $\psi_{p,q}$ has exactly one positive zero that lies in the interval $(q/p,1-(p-q)(q-2)/(2p^3))$. One checks that $\psi_{p,q}(\theta_{p,q})=0$, where
$$
\theta_{p,q}\eqdef \frac{2q(p-q)+q^2(p-1)(p-2)}{2p^2(q-2)}\left(\sqrt{1+\frac{4p(p-2)(q-2)}{(pq-3q+2)^2}}-1\right).
$$
Consequently, $q/p<\theta_{p,q}<1-(p-q)(q-2)/(2p^3)$ (in particular, the rightmost inequality in~\eqref{eq:theta upper in theorem} is valid), and
\begin{equation}\label{eq:psi pq neg}
\forall\, \theta\in (0,1),\qquad \psi_{p,q}(\theta)\le 0\implies \theta\le \theta_{p,q}.
\end{equation}

Now, suppose that $(L_q,\|x-y\|_q^\theta)$ admits a bi-Lipschitz embedding into $L_p$. If $\theta\le q/p<\theta_{p,q}$ then we are done, so we may assume below that $\theta>q/p$. Since $\ell_q(\C)$ embeds isometrically into $L_q$, there exists $\Lambda\in [1,\infty)$ such that for every $m,n\in \N$ there is a mapping $f_{m,n}:\Z_{4m}^n\to L_p$ that satisfies for every $x,y\in \Z_{4m}^n$,
\begin{equation}\label{eq:Lambda bi lipschitz}
\left(\sum_{j=1}^n \left|e^{\frac{\pi i x_j}{2m}}-e^{\frac{\pi i y_j}{2m}}\right|^q\right)^{\frac{\theta}{q}}\le \|f_{m,n}(x)-f_{m,n}(y)\|_p\le \Lambda\left(\sum_{j=1}^n \left|e^{\frac{\pi i x_j}{2m}}-e^{\frac{\pi i y_j}{2m}}\right|^q\right)^{\frac{\theta}{q}}.
\end{equation}
Suppose that $m\ge n$ and define $k\in \n$ by $k\eqdef \lceil n^3/m^2\rceil$. By Theorem~\ref{thm:main} and Remark~\ref{rem:smaller m worse constant}, in conjunction with~\eqref{eq:Lambda bi lipschitz}, we have
\begin{equation}\label{eq:plug for snowlflake}
\frac{n^{\frac{3\theta p}{q}}}{{m^{\frac{2\theta p}{q}+p}}}=\frac{k^{\frac{\theta p}{q}}}{m^p}\le (c(p)\Lambda)^p\left(\frac{k}{m^{\theta p}}+\left(\frac{k}{n}\right)^{\frac{p}{2}}\cdot \frac{n^{\frac{\theta p}{q}}}{m^{\theta p}}\right)=(2c(p)\Lambda)^p\left(\frac{n^3}{m^{\theta p+2}}+\frac{n^{p+\frac{\theta p}{q}}}{m^{(1+\theta)p}}\right),
\end{equation}
where $c(p)\in (1,\infty)$ may depend only on $p$.

Choose $m\in \N$ by setting
$$
m\eqdef \left\lceil n^{\frac{p-3+\theta p/q}{p-2}}\right\rceil= \left\lceil n^{1+\frac{\theta p-q}{q(p-2)}}\right\rceil.
$$
Observe that since $\theta >q/p$ and $p>2$ we have $m\ge n$. The above choice of $m$ ensures that
$$
\frac{n^3}{m^{\theta p+2}}+\frac{n^{p+\frac{\theta p}{q}}}{m^{(1+\theta)p}}\lesssim_p \frac{n^3}{m^{\theta p+2}},
$$
and therefore by~\eqref{eq:plug for snowlflake} (and our choice of $m$) we have
\begin{equation}\label{eq:n complicated power}
n^{\frac{3\theta p}{q}-3+\left(\theta p+2-\frac{2\theta p}{q}-p\right)\left(1+\frac{\theta p-q}{q(p-2)}\right)}\lesssim_p (c(p)\Lambda)^p.
\end{equation}
Since~\eqref{eq:n complicated power} is supposed to hold true for $n$ that can be arbitrarily large, we necessarily have
$$
\psi_{p,q}(\theta)=\frac{3\theta p}{q}-3+\left(\theta p+2-\frac{2\theta p}{q}-p\right)\left(1+\frac{\theta p-q}{q(p-2)}\right)\le 0.
$$
Recalling~\eqref{eq:psi pq neg}, this implies that $\theta\le\theta_{p,q}$, as required.
\end{proof}

Before proving Theorem~\ref{thm:our grid distortion}  we record for future use the following very simple lemma.

\begin{lemma}\label{lem:approximate cosine}
For every two integers $m,n\ge 2$ there exists a mapping $h_m^n:\Z_{m}^n\to \{0,\ldots,4m\}^{2n}$ such that for every $q\in [2,\infty)$ and $x,y\in \Z_m^n$ we have
$$
m\left(\sum_{j=1}^n \left|e^{\frac{2\pi i x_j}{m}}-e^{\frac{2\pi i y_j}{m}}\right|^q\right)^{\frac{1}{q}}\le \left\|h_m^n(x)-h_m^n(y)\right\|_q\le 3m\left(\sum_{j=1}^n \left|e^{\frac{2\pi i x_j}{m}}-e^{\frac{2\pi i y_j}{m}}\right|^q\right)^{\frac{1}{q}}.
$$
\end{lemma}

\begin{proof}
For every $u\in \Z_m$ choose $a_m(u),b_m(u)\in \{0,\ldots, 4m\}$ such that
$$
\left|2m+ 2m\cos\left(\frac{2\pi u}{m}\right)-a_m(u)\right|\le \frac12\qquad\mathrm{and}\qquad \left|2m+ 2m\sin\left(\frac{2\pi u}{m}\right)-b_m(u)\right|\le \frac12.
$$
Then, for every distinct $u,v\in \Z_m$ we have
\begin{align*}
\left(|a_m(u)-a_m(v)|^q+|b_m(u)-b_m(v)|^q\right)^{\frac1{q}}&\le \sqrt{|a_m(u)-a_m(v)|^2+|b_m(u)-b_m(v)|^2}\\&\le 2m\left|e^{\frac{2\pi i u}{m}}-e^{\frac{2\pi i v}{m}}\right|+\frac{2}{\sqrt{2}}\le 3m\left|e^{\frac{2\pi i u}{m}}-e^{\frac{2\pi i v}{m}}\right|,
\end{align*}
since for distinct $u,v\in \Z_m$ we have $\left|e^{\frac{2\pi i u}{m}}-e^{\frac{2\pi i v}{m}}\right|\ge \left|e^{\frac{2\pi i }{m}}-1\right|\ge \frac{4}{m}$. Similarly,
\begin{align*}
\left(|a_m(u)-a_m(v)|^q+|b_m(u)-b_m(v)|^q\right)^{\frac1{q}}&\ge \frac{1}{\sqrt{2}}\sqrt{|a_m(u)-a_m(v)|^2+|b_m(u)-b_m(v)|^2}\\&\ge
 \frac{2-\frac{\sqrt{2}}{4}}{\sqrt{2}}m\left|e^{\frac{2\pi i u}{m}}-e^{\frac{2\pi i v}{m}}\right|\ge m\left|e^{\frac{2\pi i u}{m}}-e^{\frac{2\pi i v}{m}}\right|.
\end{align*}
Hence $h_m^n(x)\eqdef (a_m(x_1),b_m(x_1),a_m(x_2),b_m(x_2),\ldots,a_m(x_n),b_m(x_n))$ has the desired property.
\end{proof}

\begin{proof}[Proof of Theorem~\ref{thm:our grid distortion}]
We shall show that for an appropriate choice of $\beta_p\in (0,\infty)$ we have
\begin{equation}\label{eq:8m}
m\ge n^{1+\frac{p-q}{q(p-2)}} \implies c_p\left([16m]^{2n}_q\right)\ge  \beta_pn^{\frac{(p-q)(q-2)}{q^2(p-2)}}.
\end{equation}
Since $[M]_q^n\supseteq [m]_q^n$ for every integer $M\ge m$ and $[m]_q^N$ contains an isometric copy of $[m]_q^n$ for every integer $N\ge n$, the validity of~\eqref{eq:8m} implies the desired estimate~\eqref{eq:our lower grid}.

Fix $D\in [1,\infty)$ and suppose that $f:[16m]_q^{2m}\to L_p$ satisfies
\begin{equation}\label{eq:distortion f D for lower bound}
\forall\, x,y\in [16m]_q^{2n},\qquad \|x-y\|_q\le \|f(x)-f(y)\|_p\le D\|x-y\|_q.
\end{equation}
Our goal is to bound $D$ from below. Define $F:\Z_{4m}^n\to L_p$ by $F=f\circ h_{4m}^n$, where $h_{4m}^n$ is the mapping from Lemma~\ref{lem:approximate cosine}. Then for every $x\in \Z_{4m}^n$, every $j\in \n$, every $\e\in \{-1,1\}^n$ and every $S\subset \n$ we have
\begin{equation}\label{eq:partialj F}
\|F(x+e_j)-F(x)\|_p\le 3mD\left|e^{\frac{\pi i}{2m}}-1\right|\lesssim D,
\end{equation}
\begin{equation}\label{eq:partial eps F}
\|F(x+\e)-F(x)\|_p\le 3mD\left(\sum_{j=1}^n \left|e^{\frac{\pi i \e_j}{2m}}-1\right|^q\right)^{\frac{1}{q}}\lesssim Dn^{\frac1{q}},
\end{equation}
and
\begin{equation}\label{eq:eps S F}
\|F(x+2m\e_S)-F(x)\|_p\ge m \left(\sum_{j\in S} \left|e^{\pi i}-1\right|^q\right)^{\frac{1}{q}}\gtrsim m|S|^{\frac{1}{q}}.
\end{equation}

Denote
\begin{equation}\label{eq:choice of k for embeddings}
k=\left\lceil n^{\frac{p(q-2)}{q(p-2)}} \right\rceil.
\end{equation}
Then $k\le n$ and the assumption on $m$ in~\eqref{eq:8m} implies that $m\ge n^{3/2}/\sqrt{k}$. Hence, by Theorem~\ref{thm:main} and Remark~\ref{rem:smaller m worse constant},  combined with~\eqref{eq:partialj F}, \eqref{eq:partial eps F} and~\eqref{eq:eps S F}, there exists $K_p\in (0,\infty)$ such that
$$
n^{\frac{p^2(q-2)}{q^2(p-2)}}\le k^{\frac{p}{q}}\le K_p^pD^p\left(k+\frac{k^{\frac{p}{2}}}{n^{\frac{p}{2}-\frac{p}{q}}}\right)\lesssim_p K_p^pD^pn^{\frac{p(q-2)}{q(p-2)}}.
$$
Consequently,
\begin{equation*}
D\gtrsim \frac{n^{\frac{(p-q)(q-2)}{q^2(p-2)}}}{K_p}.\qedhere
\end{equation*}
\end{proof}

\begin{remark}\label{rem:Bourgain discretization}
Lower bounds on $c_p([m]_q^n)$ that are weaker than those of Theorem~\ref{thm:our grid distortion} can also be deduced from general discretization principles (combined with the asymptotic computation of $c_p(\ell_q^n)$ in~\cite{FJS88}), namely from Bourgain's discretization theorem~\cite{Bou87} and its quantitative improvement for $L_p$ spaces in~\cite{GNS12}. Specifically, let $B_q^n$ denote the unit ball of $\ell_q^n$. Observe that $\frac{1}{m}\{-m,\ldots,m\}^n$ contains a $\d$-dense subset of $B_q^n$, with $\d\le n^{1/q}/m$. By Theorem~1.3 in~\cite{GNS12} (and the discussion immediately following it) we see that there exists a universal constant $\gamma\in (0,1)$ such that if
$$
\frac{n^{\frac{1}{q}}}{m}\le   \frac{\gamma}{\sigma(p,q)n^{2+\frac{(p-q)(q-2)}{q^2(p-2)}}} \le \frac{\beta}{n^2c_p(\ell_q^n)}
$$
then
$$
c_p\left([2m]_q^n\right)\ge \frac{c_p(\ell_q^n)}{2}\gtrsim \sigma(p,q) n^{\frac{(p-q)(q-2)}{q^2(p-2)}},
$$
where $\sigma(p,q)\in (0,\infty)$ is as in~\eqref{eq:banach mazur q p}. Consequently,
\begin{equation}\label{eq:m large full distrotion from discretization}
m\ge \frac{\sigma(p,q)}{\gamma}\cdot n^{2+\frac{1}{q}+\frac{(p-q)(q-2)}{q^2(p-2)}}= \frac{\sigma(p,q)}{\gamma}\cdot n^{2+\frac{p-q}{q(p-2)}+\frac{p(q-2)}{q^2(p-2)}}\implies c_p([2m]_q^n)\gtrsim \sigma(p,q) n^{\frac{(p-q)(q-2)}{q^2(p-2)}}.
\end{equation}
We note that a direct application of Bourgain's discretization theorem~\cite{Bou87} (which holds true also for target spaces that need not be $L_p$ spaces) would imply the same bound on $c_p([2m]_q^n)$ as in~\eqref{eq:m large full distrotion from discretization}, provided that $m$ is much larger than the requirement appearing in~\eqref{eq:m large full distrotion from discretization} (specifically, $m$ would have to be at least doubly exponential in $n\log n$).
\end{remark}


\begin{proof}[Proof of Theorem~\ref{thm:best m implies sharp nonembeddability}]
The proof  follows the proofs of Theorem~\ref{thm:our snowflake} and Theorem~\ref{thm:our grid distortion} with a different (optimal) setting of parameters that is made possible due to the assumed validity of Conjecture~\ref{conj:best m}. Specifically, we are now assuming that~\eqref{eq:Xp alpha(p) version} holds true provided $m\ge C_p\sqrt{n/k}$.

Dealing first with~\eqref{eq:sharp snowflake bound}, fix $\theta\in (q/p,1]$ and $n\in \N$. Choose $m,k\in \N$ as follows.
\begin{equation}\label{eq:choice mk}
m\eqdef \left\lfloor n^{\frac{\theta p -q}{q(p-2)}}\right\rfloor \qquad\mathrm{and}\qquad k\eqdef \left\lceil \frac{C_p^2 n}{m^2}\right\rceil.
\end{equation}
Since $\theta>q/p$ we may assume that $n$ is large enough so that $m\ge C_p$, in which case we have $k\in \n$ and $m\ge C_p\sqrt{n/k}$. Suppose for the sake of obtaining a contradiction that there exists $f_{m,n}:\Z_{4m}^n\to L_p$ satisfying~\eqref{eq:Lambda bi lipschitz}. An application of~\eqref{eq:Xp alpha(p) version} then yields the following estimate.
\begin{equation}\label{eq:optimal theta constraint}
\alpha_pC_p^{\frac{2\theta p}{q}}n^{\frac{p(q^2-2\theta^2p)}{q^2(p-2)}}\stackrel{\eqref{eq:choice mk}}{\le}\frac{\alpha_pk^{\frac{\theta p}{q}}}{m^p}\stackrel{\eqref{eq:Xp alpha(p) version}\wedge \eqref{eq:Lambda bi lipschitz}}{\le} \Lambda^p\left(\frac{k}{m^{\theta p}}+\left(\frac{k}{n}\right)^{\frac{p}{2}}\cdot\frac{n^{\frac{\theta p}{q}}}{m^{\theta p}}\right)
\stackrel{\eqref{eq:choice mk}}{\lesssim_p} (C_p\Lambda)^pn^{1-(2+\theta p)\frac{\theta p-q}{q(p-2)}}.
\end{equation}
Since~\eqref{eq:optimal theta constraint} holds true for arbitrarily large $n$, we conclude that
$$
\frac{p(q^2-2\theta^2p)}{q^2(p-2)}\le 1-(2+\theta p)\frac{\theta p-q}{q(p-2)}= \frac{p(q^2-2\theta^2p)}{q^2(p-2)}-\frac{\theta p(q-2)(\theta p-q)}{q^2(p-2)}.
$$
Consequently $\theta\le q/p$, contradicting the initial assumption that $\theta>q/p$. This proves~\eqref{eq:sharp snowflake bound}.

Next, we have already seen in~\eqref{eq:better of two embeddings} that $c_p([m]_q^n)$ is bounded from above by a constant multiple of the quantity appearing in~\eqref{eq:sharp distortion grid transition}. By arguing as in the beginning of the proof of Theorem~\ref{thm:our grid distortion}, it therefore suffices to show that for every $m,n\in \N$ we have
\begin{equation}\label{eq:8m-sharp}
m\ge n^{\frac{p-q}{q(p-2)}} \implies c_p\left([16m]^{2n}_q\right)\ge  \xi(p)n^{\frac{(p-q)(q-2)}{q^2(p-2)}}
\end{equation}
for some $\xi(p)\in (0,\infty)$. To this end, suppose that there exists $f:[16m]_q^{2n}\to L_p$ satisfying~\eqref{eq:distortion f D for lower bound}, our goal being to bound $D$ from below. As explained in the proof of  Theorem~\ref{thm:our grid distortion}, this implies the existence of $F:\Z_{4m}^n\to L_p$ that satisfies~\eqref{eq:partialj F}, \eqref{eq:partial eps F} and~\eqref{eq:eps S F}. Similarly to~\eqref{eq:choice of k for embeddings}, choose $k\in \N$ to be
\begin{equation}\label{eq:new k Cp}
k\eqdef \left\lceil C_p^2 n^{\frac{p(q-2)}{q(p-2)}}\right\rceil.
\end{equation}
We may suppose that $n$ is large enough so that $k\in \n$, since otherwise~\eqref{eq:8m-sharp} is vacuous. The lower bound on $m$ that is assumed in~\eqref{eq:8m-sharp} implies that $m\ge C_p\sqrt{n/k}$, so  we may apply~\eqref{eq:Xp alpha(p) version}, yielding, in conjunction with~\eqref{eq:partialj F}, \eqref{eq:partial eps F} and~\eqref{eq:eps S F}, that the following holds true.
\begin{equation*}
\alpha_p C_p^{\frac{p}{q}}n^{\frac{p^2(q-2)}{q^2(p-2)}}\stackrel{\eqref{eq:new k Cp}}{\le}\alpha_p k^{\frac{p}{q}}\lesssim_p D^p\left(k+\frac{k^{\frac{p}{2}}}{n^{\frac{p}{2}-\frac{p}{q}}}\right)\stackrel{\eqref{eq:new k Cp}}{\lesssim_p} (C_pD)^pn^{\frac{p(q-2)}{q(p-2)}}\implies D\gtrsim \frac{\alpha_p^{\frac1{p}}}{C_p^{1-\frac{1}{q}}}\cdot n^{\frac{(p-q)(q-2)}{q^2(p-2)}}.\qedhere
\end{equation*}
\end{proof}

\section{Proof of Theorem~\ref{thm:main}}\label{sec:proof of main}

Suppose from now on that $m,n\in \N$ satisfy $m\ge n$ and that $R\in [n,2m]$ is an odd integer. In what follows we shall use the canonical identification of $\Z_{4m}^n$ with $[-(2m-1),2m-1]^n\cap \Z^n$. Fix $S\subset \n$ and define $U_S\subset \Z_{4m}^n$ by
\begin{equation}\label{eq:def US}
U_S\eqdef \left\{y\in [-R,R]^n:\ \forall(i,j)\in S\times (\n\setminus S),\quad (y_i,y_j)\in (2\Z)\times (1+2\Z)\right\}.
\end{equation}
Thus $U_S$ consists for those $y\in \Z_{4m}^n$ satisfying $|y_j|\le R$ for every $j\in \n$, and such that $y_j$ is even for every $j\in S$ and $y_j$ is odd for every $j\in\n\setminus S$. Observe that since $R$ is odd, for every $y\in U_S$ we actually have $|y_j|<R$ if $j\in S$. Hence $|U_S|=R^{|S|}(R+1)^{n-|S|}$. Given a Banach space $(X,\|\cdot\|_X)$, the averaging operator corresponding to $U_S$ will be denoted below by $D_{S}:L_2(\Z_{4m}^n,X)\to L_2(\Z_{4m}^n,X)$, i.e., for every $f:\Z_{4m}^n\to X$ and $x\in \Z_{4m}^n$ we set
\begin{equation}\label{eq:def DS}
D_{S}f(x)\eqdef \frac{1}{|U_S|}\sum_{y\in U_S}f(x+y).
\end{equation}

The following lemma extends Lemma~5.1 in~\cite{MN08}, which corresponds to the special case $|S|=1$.

\begin{lemma}\label{lem:displacement first version} Suppose that $m,n\in \N$, and that $R\in \{1,\ldots,2m-1\}$ is odd. Let $(X,\|\cdot\|_X)$ be a Banach space and $p\in [1,\infty)$. Then for every $f:\Z_{4m}^n\to X$ and $S\subset\n$ we have
\begin{multline}\label{eq:average displacement ineq}
\sum_{x\in \Z_{4m}^n}\left\|f(x)-D_{S}f(x)\right\|_X^p\\\lesssim_p \frac{R^p}{2^n}\sum_{\e\in \{-1,1\}^n}\sum_{x\in \Z_{4m}^n} \left\|f(x+\e)-f(x)\right\|_X^p+\frac{1}{2^n}\sum_{\e\in \{-1,1\}^n}\sum_{x\in \Z_{4m}^n}\left\|f(x+\e_S)-f(x)\right\|_X^p.
\end{multline}
\end{lemma}

\begin{proof} For every $w\in \Z^n$ all of whose coordinates are odd fix $\gamma_w:\N\cup\{0\}\to \Z^n$ that satisfies $\gamma_w(0)=0$, $\gamma_w(\|w\|_\infty)=w$ and  $\gamma_w(t)-\gamma_w(t-1)\in \{-1,1\}^n$ for every $t\in \N$. The existence of such $\gamma_w$ is explained in~\cite[Lem.~5.1]{MN08}, and we shall quickly recall now why this is so for the sake of completeness. We may assume without loss of generality that all the coordinates of $w$ are positive, since for general $w$  we could then define $\gamma_w=\sign(w)\cdot \gamma_{|w|}$, where the multiplication is coordinate-wise and we denote $\sign(w)=(\sign(w_1),\ldots,\sign(w_n))$ and $|w|=(|w_1|,\ldots,|w_n|)$. Now, supposing that all the coordinates of $w$ are positive, define $\gamma_w(0)=0$ and, inductively, for every $t\in \N$ such that $\gamma_w(2t-2)$ has already been defined, set
$$
\gamma_w(2t-1)\eqdef \gamma_w(2t-2)+\sum_{j=1}^n e_j\quad\mathrm{and}\quad \gamma_w(2t)\eqdef \gamma_w(2t-1)+
\sum_{\substack{j\in \n\\ \gamma_w(2t-1)_j<w_j}}e_j-\sum_{\substack{j\in \n\\ \gamma_w(2t-1)_j=w_j}}e_j.
$$
This explicit definition of $\gamma_w$ is not used below; we shall only need to know that $\gamma_w$ exists, and that, as our construction guarantees, we have $\e\gamma_w=\gamma_{\e w}$ for every $\e\in \{-1,1\}^n$. Note that, since the restriction of $\gamma_w$ to $\{0,\ldots,\|w\|_\infty\}$ is an $\ell_\infty$ geodesic joining $0$ and $w$, for every distinct $s,t\in \{0,\ldots,\|w\|_\infty\}$ we have $\gamma_w(s)\neq \gamma_w(t)$.

If $y\in U_S$ and $\eta\in \{-1,1\}^n$ then all the coordinates of $y-\eta_S$ are odd, and we can therefore consider $\gamma_{y-\eta_S}$. For every $x\in \Z_{4m}^n$ define $\gamma_{x, y}^\eta:\N\to \Z^n$ by $\gamma_{x,y}^\eta=x+\eta_S+\gamma_{y-\eta_S}$. Thus $\gamma_{x,y}^\eta(0)=x+\eta_S$, $\gamma_{x,y}^\eta(\|y-\eta_S\|_\infty)=x+y $ and $\gamma_{x,y}^\eta(t)-\gamma_{x,y}^\eta(t-1)\in \{-1,1\}^n$ for all $t\in \N$. Note that $\gamma_{x,y}^\eta$ depends only on those coordinates of $\eta$ that belong to $S$.

For every $z\in \Z_{4m}^n$ and $\e,\eta\in \{-1,1\}^n$ define
$$
F_\eta(z,\e)\eqdef\left\{(x,y)\in \Z_{4m}^n\times U_S:\ \gamma_{x,y}^\eta(t-1)=z\ \mathrm{and} \ \gamma_{x,y}^\eta(t)=z+\e \mathrm{\ for\ some\ }t\in [1,\|y-\eta_S\|_\infty]\right\}.
$$
Observe that for every $(x,y)\in \Z_{4m}^n\times U_S$ and $\eta\in \{-1,1\}^n$ there is at most one $t\in \{1,\ldots, \|y-\eta_S\|_\infty\}$ for which $\gamma_{x,y}^\eta(t-1)=z$ and $\gamma_{x,y}^\eta(t)=z+\e$.

We claim that
\begin{equation}\label{eq:def N integer}
N\eqdef \sum_{\eta\in \{-1,1\}^n}|F_\eta(z,\e)|
\end{equation}
is independent of $z\in \Z_{4m}^n$ and $\e\in \{-1,1\}^n$. Indeed, for every $\e,\d\in \{-1,1\}^n$ and
$z,w\in \Z_{4m}^n$ define a bijection $\psi_{z,w}^{\e,\d}: \Z_{4m}^n\times U_S\to \Z_{4m}^n\times U_S$ by $$\psi_{z,w}^{\e,\d}(x,y)\eqdef (w-\e\d z +\e\d x,\e\d y).$$
Then for every $\eta\in \{-1,1\}^n$ we have $\gamma_{\psi_{z,w}^{\e,\d}(x,y)}^{\e\d\eta}=w-\e\d z+\e\d\gamma_{x,y}^\eta$. Consequently,
$$
\left(\gamma_{x,y}^\eta(t-1),\gamma_{x,y}^\eta(t)\right)=(z,z+\e)\iff \left(\gamma_{\psi_{z,w}^{\e,\d}(x,y)}^{\e\d\eta}(t-1),\gamma_{\psi_{z,w}^{\e,\d}(x,y)}^{\e\d\eta}(t)\right)=(w,w+\d)
$$
for every $t\in \{1,\ldots,\|y-\eta_S\|_\infty\}$. This shows that for every $\eta\in \{-1,1\}^n$ the mapping $\psi_{z,w}^{\e,\d}$ is a bijection between $F_\eta(z,\e)$ and $F_{\e\d\eta}(w,\d)$, whence $|F_\eta(z,\e)|=|F_{\e\d\eta}(w,\d)|$. Consequently,
$$
\sum_{\eta\in \{-1,1\}^n}|F_\eta(z,\e)|=\sum_{\eta\in \{-1,1\}^n}|F_{\e\d\eta}(w,\d)|=\sum_{\eta\in \{-1,1\}^n}|F_\eta(w,\d)|,
$$
implying that the integer $N$ defined in~\eqref{eq:def N integer} is indeed independent of $(z,\e)\in \Z_{4m}^n\times \{-1,1\}^n$.

We shall need an estimate on $N$, which is proved by double counting as follows.
\begin{align*}
N(8m)^n&=\sum_{z\in \Z_{4m}^n}\sum_{\e,\eta\in \{-1,1\}^n}|F_\eta(z,\e)|\\&=
\sum_{z\in \Z_{4m}^n}\sum_{\e,\eta\in \{-1,1\}^n}\sum_{(x,y)\in \Z_{4m}^n\times U_S} \sum_{t=1}^{\|y-\eta_S\|_\infty}\1_{\{\gamma_{x,y}^\eta(t-1)=z\ \wedge \ \gamma_{x,y}^\eta(t)=z+\e\}}\\&=\sum_{\eta\in \{-1,1\}^n}\sum_{(x,y)\in \Z_{4m}^n\times U_S} \|y-\eta_S\|_\infty\\&\le R(8m)^n|U_S|.
\end{align*}
Consequently,
\begin{equation}\label{eq:N estimate}
N\le R|U_S|.
\end{equation}

Now, fix $f:\Z_{4m}^n\to X$. For every $x\in \Z_{4m}^n$, $y\in U_S$ and $\eta\in \{-1,1\}^n$ we have
\begin{align*}
\|f(x)&-f(x+y)\|_X^p\\&\lesssim_p \left\|f(x)-f(x+\eta_S)\right\|_X^p+\left\|f\left(\gamma_{x,y}^\eta(0)\right)-f\left(\gamma_{x,y}^\eta(\|y-\eta_S\|_\infty)\right)\right\|_X^p\\
&\lesssim_p \left\|f(x)-f(x+\eta_S)\right\|_X^p+\|y-\eta_S\|_\infty^{p-1} \sum_{t=1}^{\|y-\eta_S\|_\infty}\left\|f\left(\gamma_{x,y}^\eta(t-1)\right)-f\left(\gamma_{x,y}^\eta(t)\right)\right\|_X^p\\
&\le \left\|f(x)-f(x+\eta_S)\right\|_X^p+R^{p-1}\sum_{t=1}^{\|y-\eta_S\|_\infty}\left\|f\left(\gamma_{x,y}^\eta(t-1)\right)-f\left(\gamma_{x,y}^\eta(t)\right)\right\|_X^p.
\end{align*}
By averaging this inequality over $\eta\in \{-1,1\}^n$ we see that
\begin{multline*}
\!\!\!\!\!\|f(x)-f(x+y)\|_X^p\\\lesssim_p \frac{1}{2^n}\sum_{\eta\in \{-1,1\}^n} \left\|f(x)-f(x+\eta_S)\right\|_X^p+\frac{R^{p-1}}{2^n}\sum_{\eta\in \{-1,1\}^n} \sum_{t=1}^{\|y-\eta_S\|_\infty}\left\|f\left(\gamma_{x,y}^\eta(t-1)\right)-f\left(\gamma_{x,y}^\eta(t)\right)\right\|_X^p.
\end{multline*}
Consequently, using the definition of the operator $D_S$ and convexity, we see that
\begin{align*}
\sum_{x\in Z_{4m}^n}&\left\|f(x)-D_S f(x)\right\|_X^p\\&\le \frac{1}{|U_S|} \sum_{x\in Z_{4m}^n}\sum_{y\in U_S} \left\|f(x)-f(x+y)\right\|_X^p\\
&\lesssim_p \frac{1}{2^n}\sum_{\eta\in \{-1,1\}^n}\sum_{x\in Z_{4m}^n} \left\|f(x)-f(x+\eta_S)\right\|_X^p\\ &\qquad+\frac{R^{p-1}}{2^n|U_S|}\sum_{\eta\in \{-1,1\}^n} \sum_{t=1}^{\|y-\eta_S\|_\infty}\sum_{x\in Z_{4m}^n}\sum_{y\in U_S}\left\|f\left(\gamma_{x,y}^\eta(t-1)\right)-f\left(\gamma_{x,y}^\eta(t)\right)\right\|_X^p\\
&=\frac{1}{2^n}\sum_{\eta\in \{-1,1\}^n}\sum_{x\in Z_{4m}^n} \left\|f(x)-f(x+\eta_S)\right\|_X^p+\frac{R^{p-1}N}{2^n|U_S|}\sum_{\e\in \{-1,1\}^n}\sum_{z\in \Z_{4m}^n}\|f(z+\e)-f(z)\|_X^p.
\end{align*}
Recalling the upper bound on $N$ appearing in~\eqref{eq:N estimate}, this implies the desired estimate~\eqref{eq:average displacement ineq}.
\end{proof}

We record for future use the following very simple lemma.

\begin{lemma}\label{lem:displacement on S gradient} Suppose that $(X,d_X)$ is a metric space and $p\in [1,\infty)$. Then for every $f:\Z_{4m}^n\to X$, $\e\in \{-1,1\}^n$ and $S\subset \n$ we have
\begin{equation}\label{eq:average displacement ineq-radient}
\sum_{x\in \Z_{4m}^n} d_X\left(f\left(x+\e_S\right),f(x)\right)^p\le |S|^{p-1} \sum_{j\in S}\sum_{x\in \Z_{4m}^n} d_X\left(f\left(x+e_{j}\right),f\left(x\right)\right)^p.
\end{equation}
\end{lemma}

\begin{proof}
Write $S=\{j(1),\ldots,j(|S|)\}$ and for every $\ell\in \{0,\ldots,|S|\}$ denote $S(\ell)=\{j(1),\ldots,j(\ell)\}$ (with the convention $S(0)=\emptyset$).  Then by the triangle inequality and H\"older's inequality, for every $\e \in \{-1,1\}^n$ we have
$$
d_X\left(f\left(x+\e_S\right),f(x)\right)^p\le |S|^{p-1}\sum_{\ell=1}^{|S|} d_X\left(f\left(x+\e_{S(\ell-1)}+\e_{j(\ell)}e_{j(\ell)}\right),f\left(x+\e_{S(\ell-1)}\right)\right)^p.
$$
Hence,
\begin{multline*}
\sum_{x\in \Z_{4m}^n} d_X\left(f\left(x+\e_S\right),f(x)\right)^p\le |S|^{p-1} \sum_{\ell=1}^{|S|} \sum_{y\in \Z_{4m}^n} d_X\left(f\left(y+\e_{j(\ell)}e_{j(\ell)}\right),f\left(y\right)\right)^p\\
= |S|^{p-1} \sum_{\ell=1}^{|S|} \sum_{z\in \Z_{4m}^n} d_X\left(f\left(z+e_{j(\ell)}\right),f\left(z\right)\right)^p=|S|^{p-1} \sum_{j\in S}\sum_{z\in \Z_{4m}^n} d_X\left(f\left(z+e_{j}\right),f\left(z\right)\right)^p.\tag*\qedhere
\end{multline*}
\end{proof}

\begin{lemma} Suppose that $m,n\in \N$, and that $R\in \{1,\ldots,2m-1\}$ is odd and $k\in \n$. Let $(X,\|\cdot\|_X)$ be a Banach space and $p\in [1,\infty)$. Then for every $f:\Z_{4m}^n\to X$ and $\d\in \{-1,1\}^n$,
\begin{align}\label{eq:desired averaged on S}
\nonumber\frac{1}{\binom{n}{k}}&\sum_{\substack{S\subset \n\\|S|=k}}\sum_{x\in \Z_{4m}^n}\left\|f(x+2m\d_S)-f(x)\right\|_X^p\\ \nonumber&\lesssim_p
\frac{m^p}{\binom{n}{k}}\sum_{\substack{S\subset \n\\|S|=k}}\sum_{x\in \Z_{4m}^n}\left\|D_Sf(x+2\d_S)-D_Sf(x)\right\|_X^p\\& \qquad +\frac{R^p}{2^n}\sum_{\e\in \{-1,1\}^n}\sum_{x\in \Z_{4m}^n} \left\|f(x+\e)-f(x)\right\|_X^p+\frac{k^p}{n}\sum_{j=1}^n \sum_{x\in \Z_{4m}^n} \left\|f(x+e_j)-f(x)\right\|_X^p.
\end{align}
\end{lemma}

\begin{proof}
For every $S\subset \n$ with $|S|=k$ we have
\begin{align}\label{eq:pass to average on S}
\nonumber \sum_{x\in \Z_{4m}^n}&\left\|f(x+2m\d_S)-f(x)\right\|_X^p\\&\lesssim_p \sum_{x\in \Z_{4m}^n}\left\|D_Sf(x+2m\d_S)-D_Sf(x)\right\|_X^p\nonumber \\&\qquad+\sum_{x\in \Z_{4m}^n}\left\|D_Sf(x+2m\d_S)-f(x+2m\d_S)\right\|_X^p+ \sum_{x\in \Z_{4m}^n}\left\|D_Sf(x)-f(x)\right\|_X^p\nonumber\\
&= \sum_{x\in \Z_{4m}^n}\left\|D_Sf(x+2m\d_S)-D_Sf(x)\right\|_X^p+2\sum_{x\in \Z_{4m}^n}\left\|D_Sf(x)-f(x)\right\|_X^p.
\end{align}
The first term in~\eqref{eq:pass to average on S} can be bounded as follows.
\begin{align}\label{eq:first term upper}
\sum_{x\in \Z_{4m}^n}\left\|D_Sf(x+2m\d_S)-D_Sf(x)\right\|_X^p&\le m^{p-1}\sum_{t=1}^m \sum_{x\in \Z_{4m}^n} \left\|D_Sf(x+2t\d_S)-D_Sf(x+(2t-2)\d_S)\right\|_X^p\nonumber\\
&=m^p \sum_{x\in \Z_{4m}^n}\left\|D_Sf(x+2\d_S)-D_Sf(x)\right\|_X^p.
\end{align}
The second term in~\eqref{eq:pass to average on S}  is bounded using Lemma~\ref{lem:displacement first version} and Lemma~\ref{lem:displacement on S gradient} as follows.
\begin{multline}\label{eq:second term upper use lemmas}
\sum_{x\in \Z_{4m}^n}\left\|D_Sf(x)-f(x)\right\|_X^p\\ \lesssim_p \frac{R^p}{2^n}\sum_{\e\in \{-1,1\}^n}\sum_{x\in \Z_{4m}^n} \left\|f(x+\e)-f(x)\right\|_X^p+|S|^{p-1}\sum_{j\in S}\sum_{x\in \Z_{4m}^n} \left\|f\left(x+e_{j}\right)-f\left(x\right)\right\|_X^p.
\end{multline}

Note that for every $x\in \Z_{4m}^n$,
\begin{align*}\label{eq:binomial ratios}
\frac{1}{\binom{n}{k}}\sum_{\substack{S\subset \n\\|S|=k}}\sum_{j\in S} \left\|f\left(x+e_{j}\right)-f\left(x\right)\right\|_X^p
=\frac{k}{n}\sum_{j=1}^n\left\|f\left(x+e_{j}\right)-f\left(x\right)\right\|_X^p.
\end{align*}
Hence, the desired inequality~\eqref{eq:desired averaged on S} follows by substituting~\eqref{eq:first term upper} and~\eqref{eq:second term upper use lemmas} into~\eqref{eq:pass to average on S} and averaging the resulting inequality over all $S\subset \n$ with $|S|=k$.
\end{proof}

Our next goal is to bound the first term in the right-hand side of~\eqref{eq:desired averaged on S}. To this end we first recall some results from~\cite{GMN11}.

Fixing a Banach space $(X,\|\cdot\|_X)$, consider the averaging operator $A:L_2(\Z_{4m}^n,X)\to L_2(\Z_{4m}^n,X)$ given, for every $f:\Z_{4m}^n\to X$ and $x\in \Z_{4m}^n$, by
\begin{equation}\label{eq:def A}
Af(x)\eqdef \frac{1}{R^n}\sum_{y\in (-R,R)^n\cap (2\Z)^n} f(x+y).
\end{equation}
For $j\in \n$ denote $B_j=D_{\{j\}}$, i.e., $B_j$ is the averaging operator corresponding to the set $U_{\{j\}}$, which consists of those $y\in [-R,R]^n$ such that $y_j$ is even and $y_\ell$ is odd for every $\ell\in \n\setminus\{j\}$. (In~\cite{GMN11} the set  $U_{\{j\}}$ was denoted $S(j,R)$ and the operator $B_j$ was denoted $\mathcal{E}_j$.)

It follows from~\cite{GMN11} that for every $f:\Z_{4m}^n\to X$, every $p\in [1,\infty)$ and every $\e\in \{-1,1\}^n$ we have
\begin{multline}\label{eq:pass to rademacher quote}
\sum_{x\in \Z_{4m}^n}\left\|\left(\frac{R}{R+1}\right)^{n-1}\left(Af(x+\e)-Af(x-\e)\right)-\sum_{j=1}^n\e_j\left[B_jf(x+e_j)-B_jf(x-e_j)\right]\right\|_X^p\\
\lesssim_p p^p\sum_{s=0}^{n-1} \frac{(n/R)^{(n-s)p}}{\binom{n}{s}}\sum_{\substack{S\subset\n\\ |S|=s}}\sum_{x\in \Z_{4m}^n}\left\|f(x+2\e_S)-f(x)\right\|_X^p.
\end{multline}
Since~\eqref{eq:pass to rademacher quote} is only implicit in~\cite{GMN11} (it follows from proofs in~\cite{GMN11} rather than from explicit statements in~\cite{GMN11}), we shall now  explain how to establish~\eqref{eq:pass to rademacher quote}.

\begin{proof}[Proof of~\eqref{eq:pass to rademacher quote}] For every $T\subset \n$ define $L_T\subseteq \Z_{4m}^n$ by
$$
L_T\eqdef \left\{y\in (-R,R)^n:\ \forall (i,j)\in T\times (\n\setminus T),\ (y_i,y_j)\in 2\Z\times \{0\}\right\}.
$$
Thus $L_T$ consists of those $y\in (-R,R)^n$ all of whose coordinates are even, and all of whose coordinates that lie outside $T$ vanish. As in~\cite[Def.~3.2]{GMN11}, we let $\Delta_T:L_2(\Z_{4m}^n,X)\to L_2(\Z_{4m}^n,X)$ denote the averaging operator corresponding to $L_T$, i.e., for every $f:\Z_{4m}^n\to X$ and $x\in \Z_{4m}^n$,
$$
\Delta_Tf(x)\eqdef \frac{1}{|L_T|}\sum_{y\in L_T} f(x+y).
$$
We note in passing that the operator $A$ given in~\eqref{eq:def A} coincides with $\Delta_{\n}$.

For $\e\in \{-1,1\}^n$, $\alpha\in \{0,\ldots,n\}$ and $\beta\in \{0,\ldots,\alpha\}$ define $V_{\alpha,\beta}^\e: L_2(\Z_{4m}^n,X)\to L_2(\Z_{4m}^n,X)$ by setting for every $f:\Z_{4m}^n\to X$ and $x\in \Z_{4m}^n$,
\begin{equation}\label{eq:def V alpha beta}
V_{\alpha,\beta}^\e f(x)\eqdef \sum_{\substack{T\subset \n\\|T|=n-\alpha}}\sum_{\substack{\d\in \{-1,1\}^{\n\setminus T}\\ \langle \d,\e_{\n\setminus T}\rangle=\alpha-2\beta}}\left[\Delta_{T}f\left(x+R\d+\e_{T}\right)-\Delta_{ T}f\left(x+R\d-\e_{ T}\right)\right].
\end{equation}
Here $\langle \cdot,\cdot\rangle$ denotes the standard scalar product on $\R^n$. It is worthwhile to compare the right-hand side of~\eqref{eq:def V alpha beta} to the right-hand side of equation (44) in~\cite{GMN11} (however, note that there is a difference of a normalization factor. Our $R$ is the same as the parameter $k$ of~\cite{GMN11}).  By combining Lemma~3.8 of~\cite{GMN11} with Lemma~3.5 of~\cite{GMN11} and identity (44) of~\cite{GMN11} we see that for every $\alpha\in \{0,\ldots,n\}$ and $\beta\in \{0,\ldots,\alpha\}$ there exists $h_{\alpha,\beta}\in \R$ (related to the bivariate Bernoulli numbers; see~\cite[Sec.~3.1]{GMN11}) such that $h_{0,0}=1$,
\begin{equation}\label{eq:bernoulli number bound}
\forall\,\alpha\in \{0,\ldots,n\},\ \forall\, \beta\in \{0,\ldots,\alpha\},\qquad |h_{\alpha,\beta}|\lesssim \frac{(\alpha-\beta)!\beta!}{2^\alpha},
\end{equation}
and for every $f:\Z_{4m}^n\to X$ and $x\in \Z_{4m}^n$,
\begin{equation}\label{eq:rademacher identity}
\sum_{j=1}^n\e_j\left[B_jf(x+e_j)-B_jf(x-e_j)\right]=\left(\frac{R}{R+1}\right)^{n-1}\sum_{\alpha=0}^n\sum_{\beta=0}^\alpha \frac{h_{\alpha,\beta}}{R^\alpha}V_{\alpha,\beta}^\e f(x).
\end{equation}

Observe that $V^\e_{0,0}f(x)=Af(x+\e)-Af(x-\e)$, so it follows from~\eqref{eq:rademacher identity} that
\begin{multline}\label{eq:substract zero term}
\left\|\sum_{j=1}^n\e_j\left[B_jf(x+e_j)-B_jf(x-e_j)\right]-\left(\frac{R}{R+1}\right)^{n-1}\left(Af(x+\e)-Af(x-\e)\right)\right\|_X\\
\stackrel{\eqref{eq:rademacher identity}}{\le} \sum_{\alpha=1}^n\sum_{\beta=0}^\alpha \frac{|h_{\alpha,\beta}|}{R^\alpha}\left\|V_{\alpha,\beta}^\e f(x)\right\|_X\stackrel{\eqref{eq:bernoulli number bound}}{\lesssim} \sum_{\alpha=1}^n\frac{1}{2^\alpha}\sum_{\beta=0}^\alpha \frac{(\alpha-\beta)!\beta!}{R^\alpha}\left\|V_{\alpha,\beta}^\e f(x)\right\|_X.
\end{multline}
By convexity, it follows from~\eqref{eq:substract zero term} that
\begin{multline*}
\left\|\sum_{j=1}^n\e_j\left[B_jf(x+e_j)-B_jf(x-e_j)\right]-\left(\frac{R}{R+1}\right)^{n-1}\left(Af(x+\e)-Af(x-\e)\right)\right\|_X^p\\
\lesssim_p \sum_{\alpha=1}^n\frac{1}{2^\alpha}\left(\sum_{\beta=0}^\alpha \frac{(\alpha-\beta)!\beta!}{R^\alpha}\left\|V_{\alpha,\beta}^\e f(x)\right\|_X\right)^p\le
\sum_{\alpha=1}^n\sum_{\beta=0}^\alpha \frac{(\alpha+1)^{p-1}\left((\alpha-\beta)!\beta!\right)^p}{2^\alpha R^{\alpha p}}\left\|V_{\alpha,\beta}^\e f(x)\right\|_X^p.
\end{multline*}
We can therefore bound the left-hand side of~\eqref{eq:pass to rademacher quote} as follows.
\begin{multline}\label{eq:reduce to bounding the V}
\sum_{x\in \Z_{4m}^n}\left\|\left(\frac{R}{R+1}\right)^{n-1}\left(Af(x+\e)-Af(x-\e)\right)-\sum_{j=1}^n\e_j\left[B_jf(x+e_j)-B_jf(x-e_j)\right]\right\|_X^p\\
\lesssim_p  \sum_{\alpha=1}^n\sum_{\beta=0}^\alpha \frac{(\alpha+1)^{p-1}\left((\alpha-\beta)!\beta!\right)^p}{2^\alpha R^{\alpha p}}
\sum_{x\in \Z_{4m}^n}\left\|V_{\alpha,\beta}^\e f(x)\right\|_X^p.
\end{multline}

Since the number of terms in the sums that appear in the definition~\eqref{eq:def V alpha beta} of $V_{\alpha,\beta}^\e$ is $\binom{n}{\alpha}\binom{\alpha}{\beta}$,
\begin{align}
\nonumber&\sum_{x\in \Z_{4m}^n}\left\|V_{\alpha,\beta}^\e f(x)\right\|_X^p\\ \nonumber&\le \binom{n}{\alpha}^{p-1}\binom{\alpha}{\beta}^{p-1}\sum_{\substack{T\subset \n\\|T|=n-\alpha}}\sum_{\substack{\d\in \{-1,1\}^{\n\setminus T}\\ \langle \d,\e_{\n\setminus T}\rangle=\alpha-2\beta}} \sum_{x\in \Z_{4m}^n}\left\|\Delta_{T}f\left(x+R\d+\e_{T}\right)-\Delta_{ T}f\left(x+R\d-\e_{ T}\right)\right\|_X^p\\ \nonumber
&=\binom{n}{\alpha}^{p-1}\binom{\alpha}{\beta}^p\sum_{\substack{T\subset \n\\|T|=n-\alpha}}\sum_{x\in \Z_{4m}^n}\left\|\Delta_{T}f\left(x+2\e_{T}\right)-\Delta_{ T}f\left(x\right)\right\|_X^p\\
&\le \binom{n}{\alpha}^{p-1}\binom{\alpha}{\beta}^p\sum_{\substack{T\subset \n\\|T|=n-\alpha}}\sum_{x\in \Z_{4m}^n}\left\|f\left(x+2\e_{T}\right)-f\left(x\right)\right\|_X^p,
\label{eq:double sum holder}
\end{align}
where~\eqref{eq:double sum holder} is valid since $\Delta_T$ is an averaging operator.

By combining~\eqref{eq:reduce to bounding the V} with~\eqref{eq:double sum holder} we see that
\begin{align}\label{eq:before alpha estimates}
\nonumber \sum_{x\in \Z_{4m}^n}&\left\|\left(\frac{R}{R+1}\right)^{n-1}\left(Af(x+\e)-Af(x-\e)\right)-\sum_{j=1}^n\e_j\left[B_jf(x+e_j)-B_jf(x-e_j)\right]\right\|_X^p\\ \nonumber
&\lesssim_p  \sum_{\alpha=1}^n\sum_{\beta=0}^\alpha \frac{(\alpha+1)^{p-1}\left((\alpha-\beta)!\beta!\right)^p}{2^\alpha R^{\alpha p}}\binom{n}{\alpha}^{p-1}\binom{\alpha}{\beta}^p\sum_{\substack{T\subset \n\\|T|=n-\alpha}}\sum_{x\in \Z_{4m}^n}\left\|f\left(x+2\e_{T}\right)-f\left(x\right)\right\|_X^p
\\& =\sum_{\alpha=1}^n \frac{(\alpha+1)^{p}}{2^\alpha R^{\alpha p}\binom{n}{\alpha}}\left(\frac{n!}{(n-\alpha)!}\right)^p\sum_{\substack{T\subset \n\\|T|=n-\alpha}} \sum_{x\in \Z_{4m}^n}
\left\|f\left(x+2\e_{T}\right)-f\left(x\right)\right\|_X^p.
\end{align}
The desired estimate~\eqref{eq:pass to rademacher quote} is a consequence of~\eqref{eq:before alpha estimates} via the change of variable $s=n-\alpha$ and by using the bounds $n!/(n-\alpha)!\le n^\alpha$ and $(\alpha+1)^p/2^\alpha\le (2p)^p$. \end{proof}

In what follows, we will use the following simple lemma several times.

\begin{lemma}\label{lem:2eps} Suppose that $(X,d_X)$ is a metric space. Fix $S\subset \n$ and $p\in [1,\infty)$. Then for every $f:\Z_{4m}^n\to X$ we have
\begin{equation}\label{eq:2eps}
\sum_{\e\in \{-1,1\}^n} \sum_{x\in \Z_{4m}^n}d_X\left(f(x),f(x+2\e_S)\right)^p\le 2^p \sum_{\e\in \{-1,1\}^n} \sum_{x\in \Z_{4m}^n}d_X\left(f(x),f(x+\e)\right)^p.
\end{equation}
\end{lemma}

\begin{proof}
For every $\e,\d\in \{-1,1\}^n$ we have
\begin{multline*}
d_X\left(f(x),f(x+2\e_S)\right)^p\\\le 2^{p-1}d_X\left(f(x),f(x+\e_S+\d_{\n\setminus S})\right)^p+2^{p-1}d_X\left(f(x+\e_S+\d_{\n\setminus S}),f(x+2\e_S)\right)^p.
\end{multline*}
Hence,
\begin{align}\label{eq:to average over delta}
\nonumber &\sum_{x\in \Z_{4m}^n}d_X\left(f(x),f(x+2\e_S)\right)^p\\ \nonumber& \le 2^{p-1}\sum_{x\in \Z_{4m}^n}\left(d_X\left(f(x),f(x+\e_S+\d_{\n\setminus S})\right)^p+d_X\left(f(x+\e_S+\d_{\n\setminus S}),f(x+2\e_S)\right)^p\right)\\&= 2^{p-1}\sum_{x\in \Z_{4m}^n}\left(d_X\left(f(x),f(x+\e_S+\d_{\n\setminus S})\right)^p+d_X\left(f(x),f(x+\e_S-\d_{\n\setminus S})\right)^p\right).
\end{align}
By averaging~\eqref{eq:to average over delta} over $\d\in \{-1,1\}^n$ while using the fact that $\d_{\n\setminus S}$ and $-\d_{\n\setminus S}$ are identically distributed, we deduce that
\begin{equation}\label{eq:delta average to average over eps}
\sum_{x\in \Z_{4m}^n}d_X\left(f(x),f(x+2\e_S)\right)^p\le \frac{2^p}{2^n}\sum_{\d\in \{-1,1\}^n} \sum_{x\in \Z_{4m}^n}d_X\left(f(x),f(x+\e_S+\d_{\n\setminus S})\right)^p.
\end{equation}

If $\e$ and $\d$ are i.i.d. and uniformly distributed over $\{-1,1\}^n$ then the vector $\e_S+\d_{\n\setminus S}$ is also uniformly distributed over $\{-1,1\}^n$. Consequently, the desired estimate~\eqref{eq:2eps} follows by averaging~\eqref{eq:delta average to average over eps} over $\e\in \{-1,1\}^n$.
\end{proof}

The following two lemmas contain estimates that will be used crucially in the ensuing discussion.

\begin{lemma}\label{lem:dist to rad on whole grid} Let $(X,\|\cdot\|_X)$ be a Banach space.
Suppose that $R\ge 2n-1$ (in addition to the previous assumptions on $R$, i.e., that it is an odd integer with $R\le 2m$). Then for every $p\in [1,\infty)$ and $f:\Z_{4m}^n\to X$ we have
\begin{multline}\label{eq:displacement rademacher p}
\sum_{\e\in \{-1,1\}^n}\sum_{x\in \Z_{4m}^n}\left\|\left(\frac{R}{R+1}\right)^{n-1}\left(Af(x+\e)-Af(x-\e)\right)-\sum_{j=1}^n\e_j\left[B_jf(x+e_j)-B_jf(x-e_j)\right]\right\|_X^p\\
\lesssim_p \left(\frac{pn}{R}\right)^p\sum_{\e\in \{-1,1\}^n} \sum_{x\in \Z_{4m}^n}\left\|f(x)-f(x+\e)\right\|_X^p.
\end{multline}
\end{lemma}

\begin{proof}
By summing~\eqref{eq:pass to rademacher quote} over $\e\in \{-1,1\}^n$ and using Lemma~\ref{lem:2eps} we see that the left-hand side of~\eqref{eq:displacement rademacher p} is at most $(O(1)p)^p$ times the following quantity
$$
\left(\sum_{s=0}^{n-1} \left(\frac{n}{R}\right)^{(n-s)p}\right)\sum_{\e\in \{-1,1\}^n} \sum_{x\in \Z_{4m}^n}\left\|f(x)-f(x+\e)\right\|_X^p\lesssim \left(\frac{n}{R}\right)^p\sum_{\e\in \{-1,1\}^n} \sum_{x\in \Z_{4m}^n}\left\|f(x)-f(x+\e)\right\|_X^p,
$$
where in the last step we used the fact that $R\ge 2n-1$.
\end{proof}

The following lemma contains an estimate that will be used to control the average over all $\d\in \{-1,1\}^n$ of the first term in the right-hand side of~\eqref{eq:desired averaged on S}.

\begin{lemma}\label{lem:DS bound} Let $(X,\|\cdot\|_X)$ be a Banach space and fix $S\subset \n$. Suppose that $R$ is an odd integer satisfying  $2|S|-1\le R\le 2m$. Then for every $p\in [1,\infty)$ and $f:\Z_{4m}^n\to X$ we have
\begin{align}\label{eq:in lemma to average over S}
\nonumber\sum_{\d\in \{-1,1\}^n}&\sum_{x\in \Z_{4m}^n}\left\|D_Sf(x+2\d_S)-D_Sf(x)\right\|_X^p\\ \nonumber&\lesssim_p\sum_{\e\in \{-1,1\}^n}\sum_{x\in \Z_{4m}^n}\left\|\sum_{j\in S} \e_j\left[B_jf(x+e_j)-B_jf(x-e_j)\right]\right\|_X^p\\&\qquad +\left(\frac{p|S|}{R}\right)^p\sum_{\e\in \{-1,1\}^n} \sum_{x\in \Z_{4m}^n}\left\|f(x)-f(x+\e)\right\|_X^p.
\end{align}
\end{lemma}

\begin{proof} Denote $k\eqdef |S|$,  $T\eqdef \n\setminus S$ and consider $\Z_{4m}^n$ as being equal to $\Z_{4m}^S\times \Z_{4m}^T$. For every $y\in \Z_{4m}^T$ define $f_y:\Z_{4m}^S\to X$ by setting for every $x\in \Z_{4m}^S$,
$$
f_y(x)\eqdef \frac{1}{(R+1)^{n-k}} \sum_{z\in (1+2\Z)^T\cap [-R,R]^T} f(x,y+z).
$$
Let $A^{(S)}$ be the averaging operator corresponding to~\eqref{eq:def A} with $\Z_{4m}^n$ replaced by $\Z_{4m}^S$, i.e., for every $h:\Z_{4m}^S\to X$ and $x\in \Z_{4m}^S$,
$$
A^{(S)}h(x) \eqdef \frac{1}{R^k}\sum_{w\in (-R,R)^S\cap (2\Z)^S} h(x+w).
$$
Similarly, for every $j\in S$ let $B_j^{(S)}$ be the averaging operator analogous to $B_j$ but with $\Z_{4m}^n$ replaced by $\Z_{4m}^S$, i.e., for every $h:\Z_{4m}^S\to X$ and $x\in \Z_{4m}^S$,
$$
B^{(S)}_jh(x)\eqdef \frac{1}{R(R+1)^{k-1}}\sum_{\substack{a\in [-R,R]\cap (2\Z)\\b\in ([-R,R]\cap (1+2\Z))^{S\setminus \{j\}}}} h\left(x+a e_j+\sum_{s\in S\setminus\{j\}} b_s e_s\right)
$$
With these definitions, for every $(x,y)\in \Z_{4m}^S\times \Z_{4m}^T$ and $j\in S$ we have
\begin{equation}\label{eq:product identities}
D_Sf(x,y)=A^{(S)}f_y(x)\qquad\mathrm{and}\qquad B_jf(x,y)= B_j^{(S)}f_y(x).
\end{equation}

Since $R\ge 2k-1$, an application of~\eqref{eq:pass to rademacher quote} to $f_y$ yields the following estimate, which holds true for every fixed $\d\in \{-1,1\}^n$ and $y\in \Z_{4m}^T$.
\begin{align}\label{eq:to average restriction on S}
\sum_{x\in \Z_{4m}^S}&\left\|A^{(S)}f_y(x+2\d_S)-A^{(S)}f_y(x)\right\|_X^p=\sum_{x\in \Z_{4m}^S}\left\|A^{(S)}f_y(x+\d_S)-A^{(S)}f_y(x-\d_S)\right\|_X^p\nonumber \\
&\qquad\qquad\lesssim_p \sum_{x\in \Z_{4m}^S}\left\|\sum_{j\in S}\d_j\left[B_j^{(S)}f_y(x+e_j)-B_j^{(S)}f_y(x-e_j)\right]\right\|_X^p\nonumber\\&\qquad \qquad\qquad +p^p\sum_{s=0}^{k-1}\frac{(k/R)^{(k-s)p}}{\binom{k}{s}}
\sum_{\substack{W\subset S\\ |W|=s}}\sum_{x\in \Z_{4m}^S}\left\|f_y(x+2\d_W)-f_y(x)\right\|_X^p.
\end{align}
By summing~\eqref{eq:to average restriction on S} over $\d\in \{-1,1\}^n$ and $y\in \Z_{4m}^T$, while using the identities~\eqref{eq:product identities}, we see that
 \begin{align}\label{eq:now averaged restriction on S}
\sum_{\d\in \{-1,1\}^n}&\sum_{z\in \Z_{4m}^n}\left\|D_Sf(z+2\d_S)-D_Sf(z)\right\|_X^p\nonumber \\
&\lesssim_p \sum_{\d\in \{-1,1\}^n}\sum_{z\in \Z_{4m}^n}\left\|\sum_{j\in S}\d_j\left[B_jf(z+e_j)-B_jf(x-e_j)\right]\right\|_X^p\nonumber\\&\qquad +p^p\sum_{s=0}^{k-1}\frac{(k/R)^{(k-s)p}}{\binom{k}{s}}
\sum_{\substack{W\subset S\\ |W|=s}}\sum_{\d\in \{-1,1\}^n}\sum_{x\in \Z_{4m}^S}\sum_{y\in \Z_{4m}^T}\left\|f_y(x+2\d_W)-f_y(x)\right\|_X^p.
\end{align}
Recalling that $f_y$ is obtained from $f$ by averaging, it follows by convexity that for every $W\subseteq S$ and $\d\in \{-1,1\}^n$ we have
$$
\sum_{x\in \Z_{4m}^S}\sum_{y\in \Z_{4m}^T}\left\|f_y(x+2\d_W)-f_y(x)\right\|_X^p\le \sum_{z\in \Z_{4m}^n}\left\|f(z+2\d_W)-f(z)\right\|_X^p.
$$
Consequently, using Lemma~\ref{lem:2eps} and the assumption $R\ge 2k-1$, the final term in~\eqref{eq:now averaged restriction on S} is at most $(O(1)p)^p$ times the following quantity
\begin{multline*}
\left(\sum_{s=0}^{k-1}\left(\frac{k}{R}\right)^{(k-s)p}\right) \sum_{\e\in \{-1,1\}^n} \sum_{z\in \Z_{4m}^n}\left\|f(z+2\e)-f(z)\right\|_X^p\\\lesssim_p \left(\frac{k}{R}\right)^p \sum_{\e\in \{-1,1\}^n} \sum_{z\in \Z_{4m}^n}\left\|f(z+\e)-f(z)\right\|_X^p.
\end{multline*}
Hence~\eqref{eq:now averaged restriction on S}  implies the desired inequality~\eqref{eq:in lemma to average over S}.
\end{proof}


\begin{proof}[Proof of Theorem~\ref{thm:main}] From now on choose $R$ to be the smallest odd integer that is greater than $pn$,  and suppose that $$m\ge \frac{n^{3/2}\log p}{\sqrt{k}}+pn.$$ In particular we have $2n\le R\le 2m$. Fix $x\in \Z_{4m}^n$ and apply inequality~\eqref{eq:k set quote} to the scalars  $a_j=B_jf(x+e_j)-B_jf(x-e_j)$. The resulting estimate is
\begin{align}\label{eq:finalay use JMST}
\frac{(p/\log p)^{-p}}{2^n\binom{n}{k}}\sum_{\substack{S\subset \n\\|S|=k}}& \sum_{\e\in \{-1,1\}^n} \left|\sum_{j\in S} \e_j\left[B_jf(x+e_j)-B_jf(x-e_j)\right]\right|^p\nonumber \\&\lesssim_p
\frac{k}{n}\sum_{j=1}^n \left|B_jf(x+e_j)-B_jf(x-e_j)\right|^p\nonumber \\&\qquad +\frac{(k/n)^{\frac{p}{2}}}{2^n}\sum_{\e\in \{-1,1\}^n}\left|\sum_{j=1}^n \e_j\left[B_jf(x+e_j)-B_jf(x-e_j)\right]\right|^p.
\end{align}
 By summing~\eqref{eq:finalay use JMST} over $x\in \Z_{4m}^n$ we deduce that
\begin{align}\label{eq:use linear Xp}
\nonumber\frac{(p/\log p)^{-p}}{2^n\binom{n}{k}}\sum_{\substack{S\subset \n\\|S|=k}} &\sum_{\e\in \{-1,1\}^n} \sum_{x\in \Z_{4m}^n}\left|\sum_{j\in S} \e_j\left[B_jf(x+e_j)-B_jf(x-e_j)\right]\right|^p\\\nonumber&\lesssim_p
\frac{k}{n}\sum_{x\in \Z_{4m}^n}\sum_{j=1}^n \left|B_jf(x+2e_j)-B_jf(x)\right|^p\\
&\qquad +\frac{(k/n)^{\frac{p}{2}}}{2^n}\sum_{\e\in \{-1,1\}^n}\sum_{x\in \Z_{4m}^n}\left|\sum_{j=1}^n \e_j\left[B_jf(x+e_j)-B_jf(x-e_j)\right]\right|^p.
\end{align}

For every $j\in \n$, since $B_j$ is an averaging operator we have
\begin{multline}\label{eq:use Bj average}
 \sum_{x\in \Z_{4m}^n} \left|B_jf(x+2e_j)-B_jf(x)\right|^p\le \sum_{x\in \Z_{4m}^n} \left|f(x+2e_j)-f(x)\right|^p
 \lesssim_p \sum_{x\in \Z_{4m}^n} \left|f(x+e_j)-f(x)\right|^p.
\end{multline}
Recalling that $R\ge pn$, by Lemma~\ref{lem:dist to rad on whole grid} we have
\begin{align}
\nonumber \sum_{\e\in \{-1,1\}^n}&\sum_{x\in \Z_{4m}^n}\left|\sum_{j=1}^n \e_j\left[B_jf(x+e_j)-B_jf(x-e_j)\right]\right|^p\\&\lesssim_p  \sum_{\e\in \{-1,1\}^n}\sum_{x\in \Z_{4m}^n}\left|Af(x+2\e)-Af(x)\right|^p+\sum_{\e\in \{-1,1\}^n}\sum_{x\in \Z_{4m}^n}|f(x+\e)-f(x)|^p\nonumber \\
&\le \sum_{\e\in \{-1,1\}^n}\sum_{x\in \Z_{4m}^n}\left|f(x+2\e)-f(x)\right|^p+\sum_{\e\in \{-1,1\}^n}\sum_{x\in \Z_{4m}^n}|f(x+\e)-f(x)|^p \label{eq:A averaging}\\
&\lesssim_p \sum_{\e\in \{-1,1\}^n}\sum_{x\in \Z_{4m}^n}|f(x+\e)-f(x)|^p, \label{eq:implicit power p 2}
\end{align}
where in~\eqref{eq:A averaging} we used the fact that $A$ is an averaging operator.

By substituting~\eqref{eq:use Bj average} and~\eqref{eq:implicit power p 2} into~\eqref{eq:use linear Xp} we see that
\begin{multline}\label{eq:rademacher B_j bound}
\frac{(p/\log p)^{-p}}{2^n\binom{n}{k}}\sum_{\substack{S\subset \n\\|S|=k}} \sum_{\e\in \{-1,1\}^n} \sum_{x\in \Z_{4m}^n} \left|\sum_{j\in S} \e_j\left[B_jf(x+e_j)-B_jf(x-e_j)\right]\right|^p\\
\lesssim_p \frac{k}{n}\sum_{j=1}^n\sum_{x\in \Z_{4m}^n} \left|f(x+e_j)-f(x)\right|^p+ \frac{1}{2^n}\left(\frac{k}{n}\right)^{\frac{p}{2}}\sum_{\e\in \{-1,1\}^n}\sum_{x\in \Z_{4m}^n}|f(x+\e)-f(x)|^p.
\end{multline}

By averaging~\eqref{eq:in lemma to average over S} over all $S\subset \n$ with $|S|=k$ and substituting~\eqref{eq:rademacher B_j bound} into the resulting inequality, we obtain the following estimate.
\begin{multline}\label{eq:DS upper bound two terms}
\frac{(p/\log p)^{-p}}{2^n\binom{n}{k}}\sum_{\substack{S\subset \n\\|S|=k}}\sum_{\d\in \{-1,1\}^n}\sum_{x\in \Z_{4m}^n}\left|D_Sf(x+\d_S)-D_Sf(x-\d_S)\right|^p\\
\lesssim_p \frac{k}{n}\sum_{j=1}^n\sum_{x\in \Z_{4m}^n} \left|f(x+e_j)-f(x)\right|^p+\frac{1}{2^n}\left(\frac{k}{n}\right)^{\frac{p}{2}}\sum_{\e\in \{-1,1\}^n}\sum_{x\in \Z_{4m}^n}|f(x+\e)-f(x)|^p.
\end{multline}
Next, average~\eqref{eq:desired averaged on S} over $\d\in \{-1,1\}^n$ and substitute~\eqref{eq:DS upper bound two terms} into the resulting inequality, thus obtaining the following estimate (recall that in the present setting $R\lesssim pn$).
\begin{align}\label{eq;almost done main}
\nonumber \frac{1}{2^n\binom{n}{k}}&\sum_{\substack{S\subset \n\\|S|=k}}\sum_{\e\in \{-1,1\}^n}\sum_{x\in \Z_{4m}^n}\frac{\left|f(x+2m\e_S)-f(x)\right|^p}{m^p}\\ &
\nonumber \lesssim_p \left(\frac{p^p}{(\log p)^p}+\frac{k^{p-1}}{m^p}\right)\frac{k}{n}\sum_{j=1}^n\sum_{x\in \Z_{4m}^n} \left|f(x+e_j)-f(x)\right|^p\\&\qquad+\left(\frac{p^p}{(\log p)^p}+\frac{(pn)^p}{m^p}\left(\frac{n}{k}\right)^{\frac{p}{2}}\right)\frac{1}{2^n}\left(\frac{k}{n}\right)^{\frac{p}{2}}\sum_{\e\in \{-1,1\}^n}\sum_{x\in \Z_{4m}^n}|f(x+\e)-f(x)|^p.
\end{align}
Since $m\ge \frac{n^{3/2}\log p}{\sqrt{k}}$, the desired inequality~\eqref{eq:Xp in theorem1} is a consequence of~\eqref{eq;almost done main}.
\end{proof}

\section{Proof of Theorem~\ref{thm:reverse}}\label{sec:reverse}

The desired inequality~\eqref{eq:reverse in theorem1} is equivalent to the conjunction of the following two inequalities.

\begin{equation}\label{eq:reverse goal 1}
\sum_{\e\in \{-1,1\}^n}\sum_{x\in \Z_{8m}^n}|f(x+2\e)-f(x)|^p\lesssim_p \frac{(pn/k)^{\frac{p}{2}}}{\binom{n}{k}}\sum_{\substack{S\subset \n\\|S|= k}}
\sum_{\e\in \{-1,1\}^n}\sum_{x\in \Z_{8m}^n}|f(x+\e_S)-f(x)|^p,
\end{equation}
and
\begin{equation}\label{eq:reverse goal 2}
\sum_{j=1}^n\sum_{x\in \Z_{8m}^n}\frac{|f(x+4me_j)-f(x)|^p}{m^p}\lesssim_p \frac{p^{\frac{p}{2}}}{2^n\binom{n-1}{k-1}}\sum_{\substack{S\subset \n\\|S|= k}}
\sum_{\e\in \{-1,1\}^n}\sum_{x\in \Z_{8m}^n}|f(x+\e_S)-f(x)|^p.
\end{equation}
The proofs of~\eqref{eq:reverse goal 1} and~\eqref{eq:reverse goal 2} are of a different nature: \eqref{eq:reverse goal 1} is related to metric type and~\eqref{eq:reverse goal 2} is related to metric cotype. We therefore treat~\eqref{eq:reverse goal 1} and~\eqref{eq:reverse goal 2} in separate subsections.

\subsection{Metric type and proof of~\eqref{eq:reverse goal 1}}  For every $n\in \N$ and $j\in \n$ let $\sigma^j\in \{-1,1\}^n$ be given by
$$
\sigma^j\eqdef -e_j+\sum_{s\in \n\setminus \{j\}} e_s.
$$
Thus, for every $\e\in \{-1,1\}^n$, coordinate-wise multiplication by $\sigma^j$ yields
$$
\sigma^j\e=(\e_1,\ldots,\e_{j-1},-\e_j,\e_{j+1},\ldots,\e_n).
$$

Suppose that $(X,\|\cdot\|_X)$ is a Banach space and that $p\in [1,\infty]$. Slightly abusing notation that was introduced in~\cite{HN13}, let $\P_p^n(X)$ be the infimum over those $\P\in (0,\infty)$ such that for every $h:\{-1,1\}^n\to X$ we have
\begin{equation}\label{eq:pisier}
\bigg(\frac{1}{2^n}\sum_{\e\in \{-1,1\}^n}\|h(\e)-h(-\e)\|_X^p\bigg)^{\frac{1}{p}}\le \P\bigg(\frac{1}{4^n}\sum_{\e,\d\in \{-1,1\}^n} \Big\|\sum_{j=1}^n \d_j\left[h\left(\sigma^j\e\right)-h(\e)\right]\Big\|_X^p\bigg)^{\frac{1}{p}}.
\end{equation}
Note that in~\cite{HN13} the quantity $\P_p^n(X)$ denotes the best constant in an inequality that is stronger than but closely related to~\eqref{eq:pisier}. However, this distinction is not important for us here and we prefer to use the notation $\P_p^n(X)$ rather than introducing ad hoc terminology.

The quantity $\P_p^n(X)$ is called the Pisier constant of $(X,\|\cdot\|_X)$ (corresponding to dimension $n$ and exponent $p$). In the context of his work on metric type, Pisier proved in~\cite{Pis86} that $\P_p^n(X)\lesssim \log n$ for every Banach space $(X,\|\cdot\|_X)$. In order to prove~\eqref{eq:reverse goal 1} we will deal with $X=\R$, in which case it will be important that $\sup_{n\in \N} \P_p^n(\R)<\infty$. This strengthening of Pisier's inequality for real-valued functions is due to Talagrand~\cite{Tal93}, who proved that $\sup_{n\in \N} \P_p^n(\R)\le K^p$ for some universal constant $K\in (1,\infty)$, an estimate that was later improved in~\cite{NS02} to $\sup_{n\in \N} \P_p^n(\R)\lesssim p$. The rate of growth of $\sup_{n\in \N} \P_p^n(\R)$ as $p\to \infty$ remains unknown, the best available lower bound, due to Talagrand~\cite{Tal93}, being that $\sup_{n\in \N} \P_p^n(\R)$ is at least a constant multiple of $\log p$.  We refer to~\cite{Wag00,NS02,HN13} for additional classes of Banach space $(X,\|\cdot\|_X)$ for which $\sup_{n\in \N} \P_p^n(X)<\infty$.

Given a metric space $(X,\|\cdot\|_X)$, for every $n\in \N$ and $q\in [1,\infty)$ define $\BMW_q^n(X;p)$ to be the infimum over those $B\in [1,\infty)$ such that for every $h:\{-1,1\}^n\to X$ we have
\begin{equation}\label{eq:def bmw}
\sum_{\e\in \{-1,1\}^n} d_X(h(\e),h(-\e))^p\le B^pn^{\frac{p}{q}-1} \sum_{j=1}^n\sum_{\e\in \{-1,1\}^n}d_X\left(h\left(\sigma^j\e\right),h(\e)\right)^p.
\end{equation}
The quantity $\BMW_q^n(X;p)$ is called the Bourgain-Milman-Wolfson type $q$ constant of $(X,\|\cdot\|_X)$ (corresponding to dimension $n$ and exponent $p$). It was introduced and studied by Bourgain, Milman and Wolfson in~\cite{BMW86}, though, as we explained in the Introduction, the case $p=q$ was previously introduced by Enflo~\cite{Enf76}  and Gromov~\cite{Gro83} (Gromov only dealt with the case $p=q=2$). It follows from~\eqref{eq:pisier} and H\"older's inequality that if $(X,\|\cdot\|_X)$ is a Banach space then
\begin{equation}\label{eq:use pisier}
\BMW_q^n(X;p)\le \P_p^n(X)\cdot T_q^n(X;p)\lesssim (\log n) \cdot T_q^n(X;p),
\end{equation}
where the (Rademacher type $q$) constant $T_q^n(X;p)$ is defined to be the infimum over those $T\in (0,\infty)$ such that for every $x_1,\ldots,x_n\in X$ we have
$$
\left(\frac{1}{2^n}\sum_{\d\in \{-1,1\}^n} \left\|\sum_{i=1}^n \d_i x_i\right\|_X^p\right)^{\frac{1}{p}}\le T\left(\sum_{j=1}^n \|x_j\|_X^q\right)^{\frac{1}{q}}
$$

Since for many Banach spaces $(X,\|\cdot\|_X)$ good estimates on $T_q^n(X;p)$ are known, in conjunction with the available bounds on $\P_p^n(X)$, inequality~\eqref{eq:use pisier} often yields a satisfactory estimate on $\BMW_q^n(X;p)$. Such an estimate will be relevant to the ensuing proof of a metric-space-valued extension of~\eqref{eq:reverse goal 1}. There are also several important classes of (non-Banach) metric spaces $(X,d_X)$ for which good bounds on $\BMW_q^n(X;p)$ have been obtained; see for example~\cite{NS02,NPSS06,Oht09,NS11,Nao14}.  When $X=\R$, a bound that is even better than what follows from~\eqref{eq:use pisier} is known: see inequality~(6.32) in~\cite{Nao14}, which yields the estimate
\begin{equation}\label{eq: real bmw bound}
\sup_{n\in \N} \BMW_2^n(\R;p)\lesssim \sqrt{p}.
\end{equation}

The following lemma, in conjunction with~\eqref{eq: real bmw bound}, implies~\eqref{eq:reverse goal 1}. Note that there is no requirement that $m$ is sufficiently large here: the lower bound on $m$ that is assumed in Theorem~\ref{thm:reverse} will be needed only for the proof of~\eqref{eq:reverse goal 2}.

\begin{lemma}\label{lem:metric type} Suppose that $(X,d_X)$ is a metric space and $p,q\in [1,\infty)$. Then for every $n\in \N$, $k\in \n$ and $f:\Z_{8m}^n\to X$ we have
\begin{multline}\label{eq:reverse goal 1 in lemma}
\sum_{\e\in \{-1,1\}^n}\sum_{x\in \Z_{8m}^n}d_X(f(x+2\e),f(x))^p\\ \lesssim_p \left(\BMW_q^{\lfloor n/k\rfloor+1}(X;p)\right)^p\frac{(n/k)^{\frac{p}{q}}}{\binom{n}{k}}\sum_{\substack{S\subset \n\\|S|= k}}
\sum_{\e\in \{-1,1\}^n}\sum_{x\in \Z_{8m}^n}d_X(f(x+\e_S),f(x))^p.
\end{multline}
\end{lemma}

\begin{proof}
Write $n=ak+b$ where $a= \lfloor n/k\rfloor$ and $b\in \{0,\ldots,k-1\}$. For every $j\in \{1,\ldots, a\}$ define $I_j=\{(j-1)k+1,\ldots, jk\}$, and also define $I_{a+1}=\{ak+1,\ldots,ak+b\}$. Fix $x\in \Z_{8m}^n$ and $\e\in \{-1,1\}^n$. For every permutation $\pi \in S_n$ define $h^\pi_{x,\e}:\{-1,1\}^{a+1}\to X$ by
$$
\forall\,\d\in \{-1,1\}^{a+1},\qquad h^\pi_{x,\e}(\d)\eqdef f\left(x+\sum_{j=1}^{a+1} \d_j\e_{\pi(I_j)}\right).
$$
Note that for every $\pi\in S_n$, every $x\in \Z_{8m}^n$ and every $\d\in \{-1,1\}^{a+1}$ we have
\begin{align}\label{eq:invariance identity 1}
\sum_{\e\in \{-1,1\}^n}d_X\left(h^\pi_{x,\e}(\d),h^\pi_{x,\e}(-\d)\right)^p&=\sum_{\e\in \{-1,1\}^n}d_X(f(x+\e),f(x-\e))^p\nonumber\\&=\sum_{\e\in \{-1,1\}^n}d_X(f(x+2\e),f(x))^p.
\end{align}
Also, for every $\pi\in S_n$ and $j\in \n$ we have
\begin{align}\label{eq:invariance identity 2}
\nonumber\frac{1}{2^{a+1}}\sum_{\e\in \{-1,1\}^n}&\sum_{\d\in \{-1,1\}^{a+1}}\sum_{x\in \Z_{8m}^n}d_X\left(h^\pi_{x,\e}\left(\sigma^j\d\right),h^\pi_{x,\e}(\d)\right)^p\\ \nonumber&=\sum_{\e\in \{-1,1\}^n}\sum_{x\in \Z_{8m}^n}d_X\left(f\left(x+\e_{\pi(I_j)}\right),f\left(x-\e_{\pi(I_j)}\right)\right)^p\\&=\sum_{\e\in \{-1,1\}^n}\sum_{x\in \Z_{8m}^n}d_X\left(f\left(x+2\e_{\pi(I_j)}\right),f\left(x\right)\right)^p.
\end{align}

Fix $B>\BMW_q^{a+1}(X;p)$, apply~\eqref{eq:def bmw} to $h^\pi_{x,\e}$, and sum the resulting inequality over $\e\in \{-1,1\}^n$ and $x\in \Z_{8m}^n$, while using the identities~\eqref{eq:invariance identity 1} and~\eqref{eq:invariance identity 2}. The resulting inequality is
\begin{multline}\label{eq:to average over pi}
\sum_{\e\in \{-1,1\}^n}\sum_{x\in \Z_{8m}^n}d_X(f(x+2\e),f(x))^p\\\le B^p(a+1)^{\frac{p}{q}-1}\sum_{j=1}^{a+1}\sum_{\e\in \{-1,1\}^n}\sum_{x\in \Z_{8m}^n}d_X\left(f\left(x+2\e_{\pi(I_j)}\right),f\left(x\right)\right)^p.
\end{multline}
By averaging~\eqref{eq:to average over pi} over $\pi\in S_n$ we see that
\begin{align}\label{eq:use BMW}
\nonumber B^{-p}&(a+1)^{1-\frac{p}{q}}\sum_{\e\in \{-1,1\}^n}\sum_{x\in \Z_{8m}^n}d_X(f(x+2\e),f(x))^p\\ \nonumber
 &\le \sum_{j=1}^{a} \sum_{\substack{S\subset \n\\|S|= k}} \frac{|\{\pi \in S_n:\ \pi(I_j)=S\}|}{n!}\sum_{\e\in \{-1,1\}^n}\sum_{x\in \Z_{8m}^n}d_X(f(x+2\e_S),f(x))^p\\ \nonumber
 &\qquad+ \sum_{\substack{T\subset \n\\|T|= b}}\frac{|\{\pi \in S_n:\ \pi(I_{a+1})=T\}|}{n!}\sum_{\e\in \{-1,1\}^n}\sum_{x\in \Z_{8m}^n}d_X\left(f(x+2\e_T),f(x)\right)^p\\ \nonumber
 &= \frac{a}{\binom{n}{k}}\sum_{\substack{S\subset \n\\|S|= k}} \sum_{\e\in \{-1,1\}^n}\sum_{x\in \Z_{8m}^n}d_X\left(f(x+2\e_S),f(x)\right)^p\\&\qquad+\frac{1}{\binom{n}{b}}\sum_{\substack{T\subset \n\\|T|= b}} \sum_{\e\in \{-1,1\}^n}\sum_{x\in \Z_{8m}^n}d_X\left(f(x+2\e_T),f(x)\right)^p.
\end{align}

By Lemma~\ref{lem:2eps}, if $T\subseteq S\subset \n$  then
\begin{equation}\label{eq:ST bk}
 \sum_{\e\in \{-1,1\}^n}\sum_{x\in \Z_{8m}^n}d_X\left(f(x+2\e_T),f(x)\right)^p\le 2^p\sum_{\e\in \{-1,1\}^n}\sum_{x\in \Z_{8m}^n}d_X\left(f(x+\e_S),f(x)\right)^p.
\end{equation}
Fixing $T\subset\n$ with $|T|=b$, by averaging~\eqref{eq:ST bk} over all $k$-point subsets $S\subset \n$ with $S\supseteq T$ we see that
$$
 \sum_{\e\in \{-1,1\}^n}\sum_{x\in \Z_{8m}^n}d_X\left(f(x+2\e_T),f(x)\right)^p\le \frac{2^p}{\binom{n-b}{k-b}}\sum_{\substack{T\subset S\subset \n\\|S|=k}}\sum_{\e\in \{-1,1\}^n}\sum_{x\in \Z_{8m}^n}d_X\left(f(x+\e_S),f(x)\right)^p.
$$
Consequently,
\begin{align}\label{eq:pass from T to S}
\nonumber \frac{1}{\binom{n}{b}}&\sum_{\substack{T\subset \n\\|T|= b}} \sum_{\e\in \{-1,1\}^n}\sum_{x\in \Z_{8m}^n}d_X(f(x+2\e_T),f(x))^p\\
&\le \frac{2^p}{\binom{n}{b}\binom{n-b}{k-b}}\sum_{\substack{S\subset \n\\|S|=k}}|\{T\subset S:\ |T|=b\}|\sum_{\e\in \{-1,1\}^n}\sum_{x\in \Z_{8m}^n}d_X\left(f(x+\e_S),f(x)\right)^p\nonumber \\
&= \frac{2^p\binom{k}{b}}{\binom{n}{b}\binom{n-b}{k-b}}\sum_{\substack{S\subset \n\\|S|=k}}\sum_{\e\in \{-1,1\}^n}\sum_{x\in \Z_{8m}^n}d_X\left(f(x+\e_S),f(x)\right)^p\nonumber\\
&= \frac{2^p}{\binom{n}{k}}\sum_{\substack{S\subset \n\\|S|= k}} \sum_{\e\in \{-1,1\}^n}\sum_{x\in \Z_{8m}^n}d_X(f(x+\e_S),f(x))^p.
\end{align}
Since by the triangle inequality we also have
$$
\sum_{\substack{S\subset \n\\|S|= k}} \sum_{x\in \Z_{8m}^n}d_X(f(x+2\e_S),f(x))^p\le 2^p\sum_{\substack{S\subset \n\\|S|= k}} \sum_{x\in \Z_{8m}^n}d_X(f(x+\e_S),f(x))^p,
$$
it follows from~\eqref{eq:use BMW} and~\eqref{eq:pass from T to S} that
\begin{align*}
\sum_{\e\in \{-1,1\}^n}&\sum_{x\in \Z_{8m}^n}d_X(f(x+2\e),f(x))^p\\&\le
\frac{(2B)^p(a+1)^{\frac{p}{q}}}{\binom{n}{k}}\sum_{\substack{S\subset \n\\|S|= k}} \sum_{\e\in \{-1,1\}^n}\sum_{x\in \Z_{8m}^n}d_X(f(x+\e_S),f(x))^p\\
&\le  \frac{(2B)^p(2n/k)^{\frac{p}{q}}}{\binom{n}{k}}\sum_{\substack{S\subset \n\\|S|= k}} \sum_{\e\in \{-1,1\}^n}\sum_{x\in \Z_{8m}^n}d_X(f(x+\e_S),f(x))^p. \tag*\qedhere
\end{align*}
\end{proof}

\subsection{Metric cotype and proof of~\eqref{eq:reverse goal 2}} Given a metric space $(X,d_X)$
and $m,n\in \N$, for $p\in (1,\infty)$ define $\Gamma_p(X;m,n)$ to be the infimum over those $\Gamma\in (0,\infty)$ such that for every $f:\Z_{2m}^n\to X$,
\begin{equation}\label{eq:def metric cotype in section}
\sum_{j=1}^n\sum_{x\in \Z_{2m}^n} \frac{d_X\left(f(x+me_j),f(x)\right)^p}{m^p}\le \frac{\Gamma^p}{3^n}
\sum_{\e\in \{-1,0,1\}^n} \sum_{x\in \Z_{2m}^n} d_X(f(x+\e),f(x))^p.
\end{equation}
As discussed in the Introduction, following~\cite{MN08}, we say that $(X,d_X)$ has metric cotype $p$ if
$$
\Gamma_p(X)\eqdef \sup_{n\in \N}\inf_{m\in \N} \Gamma_p(X;m,n)<\infty.
$$

We need to briefly recall some facts related to $K$-convexity of Banach spaces; see the survey~\cite{Mau03} for much more on this topic. Given a Banach space $(X,\|\cdot\|_X)$, $p\in (1,\infty)$ and $n\in \N$, for every $f:\{-1,1\}^n\to X$ define its Rademacher projection $\Rad(f):\{-1,1\}^n\to X$ by
$$
\forall\, \e\in \{-1,1\}^n,\qquad \Rad(f)(\e)\eqdef \sum_{j=1}^n \frac{\sum_{\d\in \{-1,1\}^n}f(\d)\d_j}{2^n}\e_j.
$$
For $p\in (1,\infty)$ let $K_p(X)\in [1,\infty]$ be the infimum over those $K\in [1,\infty]$ such that for every $n\in \N$ and every $f:\{-1,1\}^n\to X$ we have
$$
\sum_{\e\in \{-1,1\}^n}\|\Rad(f)(\e)\|_X^p\le K^p \sum_{\e\in \{-1,1\}^n}\|f(\e)\|_X^p.
$$
A simple application of Khinchine's inequality (with asymptotically sharp constant, see~\cite[Lem.~2]{PZ30}) shows that $K_p(\R)\lesssim \sqrt{p}$ for $p\in [2,\infty)$. A Banach space $(X,\|\cdot\|_X)$ is said to be $K$-convex if $K_p(X)<\infty$ for some (equivalently for all) $p\in (1,\infty)$; see~\cite{Mau03} and the references therein.

Theorem~\ref{thm:cotype Rademacher version} below establishes a sharp metric cotype inequality for $K$-convex Banach spaces, with one difference: the averaging on the right-hand side is over $\e\in \{-1,1\}^n$ rather than $\e\in \{-1,0,1\}^n$. The same result with averages over $\e\in \{-1,0,1\}^n$ (and $x\in \Z_{4m}^n$ rather than $x\in \Z_{8m}^n$) is the content of Theorem 4.1 in~\cite{MN08}. The proof here follows the argument in~\cite{MN08} with some technical modifications. It seems likely that a similar statement could be proved for the metric cotype $p$ inequalities for Banach spaces of Rademacher cotype $p$ (with no assumption of $K$-convexity) in~\cite{MN08,GMN11}, though this may require changes to the arguments of~\cite{MN08,GMN11} that are more substantial  than what we do here.

\begin{theorem}\label{thm:cotype Rademacher version} Fix $p\in
[2,\infty)$ and $\alpha\in [1,\infty)$. Let $(X,\|\cdot\|_X)$ be a
$K$-convex Banach space of cotype $p$. Suppose that $m,n\in \N$
satisfy
\begin{equation}\label{eq:m K convexity cotype}
m\ge \frac{n^{1/p}}{\alpha K_p(X)C_p(X)},
\end{equation}
where, recalling~\eqref{eq:type cotype def}, $C_p(X)$ is the cotype $p$ constant of $X$. Then for every $f:\Z_{8m}^n\to X$ we have
\begin{equation}\label{eq:cotype with rademachers}
\sum_{j=1}^n\sum_{x\in \Z_{8m}^n} \frac{\|f(x+4me_j)-f(x)\|_X^p}{m^p}\lesssim_p \frac{(\alpha K_p(X)C_p(X))^p}{2^n}
\sum_{\e\in \{-1,1\}^n}\sum_{x\in \Z_{8m}^n} \|f(x+\e)-f(x)\|_X^p.
\end{equation}
\end{theorem}

Before proving Theorem~\ref{thm:cotype Rademacher version} we deduce
the following simple corollary, which implies~\eqref{eq:reverse goal
2} because $C_p(\R)=1$ and $K_p(\R)\lesssim\sqrt{p}$.

\begin{corollary}
Fix $p\in [2,\infty)$ and $\alpha\in [1,\infty)$. Let
$(X,\|\cdot\|_X)$ be a $K$-convex Banach space of cotype $p$.
Suppose that $m,n\in \N$ and $k\in \n$ satisfy
\begin{equation*}
m\ge \frac{k^{1/p}}{\alpha K_p(X)C_p(X)}.
\end{equation*}
Then for every $f:\Z_{8m}^n\to X$ we have
\begin{multline}\label{eq:cotype with rademachers on S sets}
\sum_{j=1}^n\sum_{x\in \Z_{8m}^n}
\frac{\|f(x+4me_j)-f(x)\|_X^p}{m^p}\\\lesssim_p \frac{(\alpha
K_p(X)C_p(X))^p}{2^n\binom{n-1}{k-1}}\sum_{\substack{S\subset
\n\\|S|= k}} \sum_{\e\in \{-1,1\}^n}\sum_{x\in
\Z_{8m}^n}\|f(x+\e_S)-f(x)\|_X^p.
\end{multline}
\end{corollary}

\begin{proof}
By Theorem~\ref{thm:cotype Rademacher version}, for every
$S\subset\n$ with $|S|=k$ we have
$$
\sum_{j\in S}\sum_{x\in \Z_{8m}^n} \frac{\|f(x+4me_j)-f(x)\|_X^p}{m^p}\lesssim_p
\frac{(\alpha K_p(X)C_p(X))^p}{2^n}
\sum_{\e\in \{-1,1\}^n}\sum_{x\in \Z_{8m}^n} \|f(x+\e_S)-f(x)\|_X^p.
$$
By averaging this inequality over all $S\subset\n$ with $|S|=k$ we
obtain~\eqref{eq:cotype with rademachers on S sets}.
\end{proof}

In order to prove Theorem~\ref{thm:cotype Rademacher version} we
first introduce a small amount of notation and prove an auxiliary
lemma. For every $j\in \n$ define a linear operator
$T_j:L_2(\Z_{8m}^n,X)\to L_2(\Z_{8m}^n,X)$ by setting for every
$f:\Z_{8m}^n\to X$ and $x\in \Z_{8m}^n$,
\begin{equation}\label{eq:def Tj}
T_jf(x)\eqdef \frac{1}{2^n}\sum_{\e\in \{-1,1\}^n} f\left(x+2\e_{\n\setminus \{j\}}\right).
\end{equation}

\begin{lemma}
Let $(X,\|\cdot\|_X)$ be a Banach space and $p\in [1,\infty)$. Fix
also $m,n\in \N$ and $j\in \n$. Then for every $f:\Z_{8m}^n\to X$ we
have
\begin{equation}\label{eq:Tj close to identity}
\sum_{x\in \Z_{8m}^n}\left\|f(x)-T_jf(x)\right\|_X^p\le \frac{2^p}{2^n}\sum_{\e\in \{-1,1\}^n}
\sum_{x\in \Z_{8m}^n} \|f(x+\e)-f(x)\|_X^p.
\end{equation}
Moreover, for every $x\in \Z_{8m}^n$ we have
\begin{multline}\label{eq:rademacher sum Tj}
\sum_{\e\in \{-1,1\}^n} \left\|\sum_{j=1}^n
\e_j\left[T_jf(x+2e_j)-T_jf(x-2e_j) \right]\right\|_X^p\\\le
(2K_p(X))^p\sum_{\e\in
\{-1,1\}^n}\left\|f(x+2\e)-f(x)\right\|_X^p\le(4K_p(X))^p\sum_{\e\in
\{-1,1\}^n}\left\|f(x+\e)-f(x)\right\|_X^p.
\end{multline}
\end{lemma}

\begin{proof} By the definition~\eqref{eq:def Tj} and convexity,
\begin{align*}
\sum_{x\in \Z_{8m}^n}\left\|f(x)-T_jf(x)\right\|_X^p&\le \frac{1}{2^n}\sum_{\e\in \{-1,1\}^n}
\sum_{x\in \Z_{8m}^n} \left\|f(x)-f\left(x+2\e_{\n\setminus \{j\}}\right)\right\|_X^p\\
&\le \frac{2^p}{2^n}\sum_{\e\in \{-1,1\}^n} \sum_{x\in \Z_{8m}^n} \|f(x+\e)-f(x)\|_X^p,
\end{align*}
where in the last step we used Lemma~\ref{lem:2eps}. This
proves~\eqref{eq:Tj close to identity}.

To prove~\eqref{eq:rademacher sum Tj}, for every fixed $x\in
\Z_{8m}^n$ define $h^x:\{-1,1\}^n\to X$ by
$$
\forall\, \e\in \{-1,1\}^n,\qquad h^x(\e)\eqdef f(x+2\e)-f(x).
$$
We claim that the following identity holds true.
\begin{equation}\label{eq:rademacher projection identity}
\Rad(h^x)(\e)=\frac{i}{2}\sum_{j=1}^n\e_j\left[T_jf(x+2e_j)-T_jf(x-2e_j) \right].
\end{equation}
Once~\eqref{eq:rademacher projection identity} is proved, the
desired inequality~\eqref{eq:rademacher sum Tj} would follow from
the definition of $K_p(X)$.

By composing with linear functionals, it suffices to verify the
validity of~\eqref{eq:rademacher projection identity} when $X=\C$.
Moreover, for every $y\in \Z_{8m}^n$ define $W_y:\Z_{8m}^n\to \C$ by
$$
\forall\, x\in \Z_{8m}^n,\qquad W_y(x)\eqdef \exp\left({\frac{\pi i}{4m}\sum_{j=1}^n x_jy_j}\right).
$$
Then $\{W_y\}_{y\in \Z_{8m}^n}$ forms an orthonormal basis of
$L_2(\Z_{8m}^n)$, and therefore it suffices to verify the validity
of~\eqref{eq:rademacher projection identity} when $f=W_y$ for some
$y\in \Z_{8m}^n$. Now,
\begin{align*}
W_y^x(\e)&=-\left(1-\prod_{j=1}^n\left(\cos\left(\frac{\pi
\e_jy_j}{2m}\right) +i\sin\left(\frac{\pi
\e_jy_j}{2m}\right)\right)\right)W_y(x)\\
&= -\left(1-\prod_{j=1}^n\left(\cos\left(\frac{\pi y_j}{2m}\right)
+i\e_j\sin\left(\frac{\pi y_j}{2m}\right)\right)\right)W_y(x)
\end{align*}
Consequently,
\begin{equation}\label{eq:specific rad formula}
\Rad(W^x_y)(\e)=i\left(\sum_{j=1}^n\e_j \sin\left(\frac{\pi y_j}{2m}\right)\prod_{s\in \n\setminus\{j\}}
\cos\left(\frac{\pi y_s}{2m}\right)\right)W_y(x).
\end{equation}
At the same time, for every $j\in \n$ we have
\begin{align*}
T_jW_y(x) &=\left(\frac{1}{2^n}\sum_{\e\in \{-1,1\}^n}\prod_{s\in \n\setminus\{j\}}\exp\left(\frac{\pi i\e_sy_s}{2m}
\right)\right)W_y(x)\\&=\left(\prod_{s\in \n\setminus\{j\}}
\cos\left(\frac{\pi y_s}{2m}\right)\right)W_y(x).
\end{align*}
Therefore
\begin{align*}
\sum_{j=1}^n\e_j&\left[T_jW_y(x+2e_j)-T_jW_y(x-2e_j)\right]\\
&=
\left(\sum_{j=1}^n\e_j\left(W_y(2e_j)-W_y(-2e_j)\right)\prod_{s\in \n\setminus\{j\}}
\cos\left(\frac{\pi y_s}{2m}\right)\right)W_y(x)\\
&=2\left(\sum_{j=1}^n\e_j \sin\left(\frac{\pi y_j}{2m}\right)\prod_{s\in \n\setminus\{j\}}
\cos\left(\frac{\pi y_s}{2m}\right)\right)W_y(x)=\frac{2}{i}\Rad(W^x_y)(\e),
\end{align*}
where in the last step we used~\eqref{eq:specific rad formula}.
\end{proof}

\begin{proof}[Proof of Theorem~\ref{thm:cotype Rademacher version}]
By the triangle inequality, for every $j\in \n$ we have
\begin{multline*}
\left\|f(x+4me_j)-f(x)\right\|_X^p\\\lesssim_p
\left\|T_jf(x+4me_j)-T_jf(x)\right\|_X^p
+\left\|f(x+4me_j)-T_jf(x+4me_j)\right\|_X^p+\left\|f(x)-T_jf(x)\right\|_X^p.
\end{multline*}
Hence, using~\eqref{eq:Tj close to identity} we see that
\begin{multline}\label{eq:three triangle cotype}
\sum_{x\in \Z_{8m}^n}\left\|f(x+4me_j)-f(x)\right\|_X^p\\\lesssim_p
\sum_{x\in \Z_{8m}^n}
\left\|T_jf(x+4me_j)-T_jf(x)\right\|_X^p+\frac{1}{2^n}\sum_{\e\in
\{-1,1\}^n} \sum_{x\in \Z_{8m}^n} \|f(x+\e)-f(x)\|_X^p.
\end{multline}
By the triangle inequality combined with H\"older's inequality we
have
\begin{align*}
\sum_{x\in \Z_{8m}^n}
\left\|T_jf(x+4me_j)-T_jf(x)\right\|_X^p&\le m^{p-1}\sum_{s=1}^m \sum_{x\in \Z_{8m}^n}\left\|T_jf(x+4se_j)-T_jf(x+4(s-1)e_j)\right\|_X^p\\
&= m^p\sum_{x\in \Z_{8m}^n}\left\|T_jf(x+2e_j)-T_jf(x-2e_j)\right\|_X^p.
\end{align*}
In combination with~\eqref{eq:three triangle cotype}, this implies
that
\begin{multline}\label{eq:set up to use cotype}
\sum_{j=1}^n\sum_{x\in \Z_{8m}^n}\frac{\left\|f(x+4me_j)-f(x)\right\|_X^p}{m^p}\\
\lesssim_p \sum_{j=1}^n\sum_{x\in
\Z_{8m}^n}\left\|T_jf(x+2e_j)-T_jf(x-2e_j)\right\|_X^p+\frac{n}{m^p2^n}\sum_{\e\in
\{-1,1\}^n} \sum_{x\in \Z_{8m}^n} \|f(x+\e)-f(x)\|_X^p.
\end{multline}
By the definition of the cotype $p$ constant $C_p(X)$, for every
$x\in \Z_{8m}^n$ we have
\begin{align}\label{eq:use rademacher for cotype}
\nonumber\sum_{j=1}^n\left\|T_jf(x+2e_j)-T_jf(x-2e_j)\right\|_X^p
&\le \nonumber \frac{C_p(X)^p}{2^n}\sum_{\e\in \{-1,1\}^n}\left\|\sum_{j=1}^n
\e_j\left[T_jf(x+2e_j)-T_jf(x-2e_j) \right]\right\|_X^p\\
& \lesssim_p \frac{(K_p(X)C_p(X))^p}{2^n}\sum_{\e\in
\{-1,1\}^n} \|f(x+\e)-f(x)\|_X^p,
\end{align}
where in the last step of~\eqref{eq:use rademacher for cotype} we
used~\eqref{eq:rademacher sum Tj}. By substituting~\eqref{eq:use
rademacher for cotype} into~\eqref{eq:set up to use cotype} we
conclude that
\begin{multline}\label{eq:before using m large cotype}
\sum_{j=1}^n\sum_{x\in
\Z_{8m}^n}\frac{\left\|f(x+4me_j)-f(x)\right\|_X^p}{m^p}\\\lesssim_p
\frac{(K_p(X)C_p(X))^p+n/m^p}{2^n}\sum_{\e\in \{-1,1\}^n} \sum_{x\in
\Z_{8m}^n} \|f(x+\e)-f(x)\|_X^p.
\end{multline}
Due to~\eqref{eq:m K convexity cotype}, \eqref{eq:before using m
large cotype} implies the desired inequality~\eqref{eq:cotype with
rademachers}.
\end{proof}

\section{A conjectural convolution inequality as a way to prove Conjecture~\ref{conj:best
m}}\label{sec:best m conv}

For every $j\in \n$ define an averaging operator $E_j:L_2(\Z_{m}^n)\to L_2(\Z_{m}^n)$ by setting for every $f:\Z_{m}^n\to \R$ and $x\in \Z_{m}^n$,
$$
E_jf(x)\eqdef \frac{f(x+e_j)+f(x-e_j)}{2}.
$$
We also set $\mathcal{E}_j\eqdef \prod_{s\in \n\setminus \{j\}} E_s$ and $\mathcal{E}\eqdef \prod_{s=1}^n E_j$. Thus, for every $f:\Z_{m}^n\to \R$ and $x\in \Z_{m}^n$,
\begin{equation}\label{eq:cal E operators}
\mathcal{E}_jf(x)= \frac{1}{2^n}\sum_{\e\in \{-1,1\}^n} f\left(x+\e_{\n\setminus \{j\}}\right)\qquad \mathrm{and}\qquad \mathcal{E}f(x)=\frac{1}{2^n}\sum_{\e\in \{-1,1\}^n} f(x+\e).
\end{equation}

\begin{question}\label{Q:convolution} Is it true that for every $p\in (2,\infty)$ there exists $\beta_p\in (0,1]$ such that for every $m,n\in \N$, every $f:\Z_{m}^n\to \R$ satisfies
\begin{multline}\label{eq:our convolution conjecture}
\frac{\beta_p}{2^n}\sum_{\e\in \{-1,1\}^n}\sum_{x\in \Z_{m}^n}\left|\mathcal{E}f(x+\e)-\mathcal{E}f(x-\e)\right|^p\\\le \frac{1}{2^n}\sum_{\e\in \{-1,1\}^n}\sum_{x\in \Z_{m}^n}\left|\sum_{j=1}^n\e_j\left[\mathcal{E}_jf(x+e_j)-\mathcal{E}_jf(x-e_j)\right]\right|^p+\sum_{j=1}^n\sum_{x\in \Z_{m}^n}\left|f(x+e_j)-f(x)\right|^p.
\end{multline}
\end{question}
It may very well be the case that~\eqref{eq:our convolution conjecture} holds true without the second term that appears in the right-hand side, i.e., that
$$
\beta_p\sum_{\e\in \{-1,1\}^n}\sum_{x\in \Z_{m}^n}\left|\mathcal{E}f(x+\e)-\mathcal{E}f(x-\e)\right|^p\le \sum_{\e\in \{-1,1\}^n}\sum_{x\in \Z_{m}^n}\left|\sum_{j=1}^n\e_j\left[\mathcal{E}_jf(x+e_j)-\mathcal{E}_jf(x-e_j)\right]\right|^p.
$$
We formulated Question~\ref{Q:convolution} in the above weaker form since it suffices for the following  proposition.

\begin{proposition}\label{prop:convolution implies best m}
A positive answer to Question~\ref{Q:convolution} implies that Conjecture~\ref{conj:best m} holds true, and hence also that all the conclusions of Theorem~\ref{thm:best m implies sharp nonembeddability} hold true. Specifically, for every $\d\in (0,\infty)$, if $m,n\in \N$ and $k\in \n$ satisfy $m\ge \d\sqrt{n/k}$ then~\eqref{eq:Xp alpha(p) version} holds true with
$$
\alpha_p\gtrsim_p \min\left\{\beta_p\left(\frac{\log p}{p^{\frac32}}\right)^p,\delta^p\right\}.
$$
where $\beta_p$ is as in~\eqref{eq:our convolution conjecture}.
\end{proposition}

\begin{proof} Fix $f:\Z_{4m}^n\to \R$.  By convexity, it follows from~\eqref{eq:cal E operators} that
$$
\sum_{x\in \Z_{4m}^n}|f(x)-\mathcal{E}f(x)|^p\le \frac{1}{2^n}\sum_{\e\in \{-1,1\}^n}\sum_{x\in \Z_{4m}^n} |f(x+\e)-f(x)|^p.
$$
Hence, for every $S\subset \n$ we have
\begin{multline}\label{eq:beta approx cal E on S}
\sum_{\e\in \{-1,1\}^n}\sum_{x\in \Z_{4m}^n} |f(x+2m \e_S)-f(x)|^p\\\lesssim_p \sum_{\e\in \{-1,1\}^n}\sum_{x\in \Z_{4m}^n} |\mathcal{E}f(x+2m \e_S)-\mathcal{E}f(x)|^p+\sum_{\e\in \{-1,1\}^n}\sum_{x\in \Z_{4m}^n} |f(x+\e)-f(x)|^p.
\end{multline}
Arguing as in~\eqref{eq:first term upper}, it follows from the triangle inequality that
\begin{equation}\label{eq:get m triangle beta}
\sum_{\e\in \{-1,1\}^n}\sum_{x\in \Z_{4m}^n} |\mathcal{E}f(x+2m \e_S)-\mathcal{E}f(x)|^p\le m^p\sum_{\e\in \{-1,1\}^n}\sum_{x\in \Z_{4m}^n}\left|\mathcal{E}f(x+\e_S)-\mathcal{E}f(x-\e_S)\right|^p.
\end{equation}

For every $z\in \Z_{4m}^{\n\setminus S}$ apply~\eqref{eq:our convolution conjecture} to the mapping $(y\in \Z_{4m}^S)\mapsto \prod_{j\in \n\setminus S} E_jf(y,z)$, and then average the resulting inequality over $z$. The estimate thus obtained is
\begin{multline}\label{eq:use beta inequality on S}
\frac{\beta_p}{2^n}\sum_{\e\in \{-1,1\}^n}\sum_{x\in \Z_{4m}^n}\left|\mathcal{E}f(x+\e_S)-\mathcal{E}f(x-\e_S)\right|^p\\\le \frac{1}{2^n}\sum_{\e\in \{-1,1\}^n}\sum_{x\in \Z_{4m}^n}\left|\sum_{j\in S}\e_j\left[\mathcal{E}_jf(x+e_j)-\mathcal{E}_jf(x-e_j)\right]\right|^p+\sum_{j\in S}\sum_{x\in \Z_{4m}^n}\left|f(x+e_j)-f(x)\right|^p.
\end{multline}
By averaging~\eqref{eq:use beta inequality on S} over those $S\subset\n$ with $|S|=k$ we see that
\begin{align}\label{eq:beta averaged over S}
\nonumber \frac{\beta_p}{2^n\binom{n}{k}}&\sum_{\substack{S\subset\n\\|S|=k}}\sum_{\e\in \{-1,1\}^n}\sum_{x\in \Z_{4m}^n}\left|\mathcal{E}f(x+\e_S)-\mathcal{E}f(x-\e_S)\right|^p\\
&\le \frac{1}{2^n\binom{n}{k}}\sum_{\substack{S\subset\n\\|S|=k}}\sum_{\e\in \{-1,1\}^n}\sum_{x\in \Z_{4m}^n}\left|\sum_{j\in S}\e_j\left[\mathcal{E}_jf(x+e_j)-\mathcal{E}_jf(x-e_j)\right]\right|^p\nonumber\\&\qquad+\frac{k}{n}\sum_{j=1}^n\sum_{x\in \Z_{4m}^n}\left|f(x+e_j)-f(x)\right|^p.
\end{align}

Note that since $\mathcal{E}_j$ is an averaging operator,
$$
\sum_{j=1}^n\sum_{x\in \Z_{4m}^n} \left|\mathcal{E}_jf(x+e_j)-\mathcal{E}_jf(x-e_j)\right|^p\le 2^p \sum_{j=1}^n\sum_{x\in \Z_{4m}^n} \left|f(x+e_j)-f(x)\right|^p.
$$
Hence, using the linear $X_p$ inequality~\eqref{eq:k set quote}, we deduce that
\begin{align}\label{eq:use Xp beta}
\nonumber\frac{(p/\log p)^{-p}}{2^n\binom{n}{k}}&\sum_{\substack{S\subset\n\\|S|=k}}\sum_{\e\in \{-1,1\}^n}\sum_{x\in \Z_{4m}^n}\left|\sum_{j\in S}\e_j\left[\mathcal{E}_jf(x+e_j)-\mathcal{E}_jf(x-e_j)\right]\right|^p
\\\nonumber&\lesssim_p \frac{k}{n}\sum_{j=1}^n\sum_{x\in \Z_{4m}^n} \left|f(x+e_j)-f(x)\right|^p\\&\qquad+\frac{\left(k/n \right)^{\frac{p}{2}}}{2^n}\sum_{\e\in \{-1,1\}^n}\sum_{x\in \Z_{4m}^n}
\left|\sum_{j=1}^n\e_j\left[\mathcal{E}_jf(x+e_j)-\mathcal{E}_jf(x-e_j)\right]\right|^p.
\end{align}

The same reasoning that leads to the identity~\eqref{eq:rademacher projection identity} (alternatively, by~\cite[Sec.~5]{MN08}) shows that if for fixed $x\in \Z_{4m}^n$ we define $g^x:\{-1,1\}^n\to \R$ by setting $g^x(\e)=f(x+\e)-f(x)$ for every $\e\in \{-1,1\}^n$, then that Rademacher projection of $g^x$ satisfies
$$
\Rad(g^x)(\e)=\frac{i}{2}\sum_{j=1}^n\e_j\left[\mathcal{E}_jf(x+e_j)-\mathcal{E}_jf(x-e_j) \right].
$$
Hence, recalling that the $K$-convexity constant of $\R$ satisfies $K_p(\R)\lesssim \sqrt{p}$,
\begin{equation}\label{eq:use K convexity in beta ineq}
\sum_{\e\in \{-1,1\}^n}\sum_{x\in \Z_{4m}^n}
\left|\sum_{j=1}^n\e_j\left[\mathcal{E}_jf(x+e_j)-\mathcal{E}_jf(x-e_j)\right]\right|^p\lesssim_p p^{\frac{p}{2}}\sum_{\e\in \{-1,1\}^n}\sum_{x\in \Z_{4m}^n}|f(x+\e)-f(x)|^p.
\end{equation}
By combining~\eqref{eq:beta averaged over S} with~\eqref{eq:use Xp beta} and~\eqref{eq:use K convexity in beta ineq} we have
\begin{multline*}
\frac{(p^{3/2}/\log p)^{-p}\beta_p}{2^n\binom{n}{k}}\sum_{\substack{S\subset\n\\|S|=k}}
\sum_{\e\in \{-1,1\}^n}\sum_{x\in \Z_{4m}^n}\left|\mathcal{E}f(x+\e_S)-\mathcal{E}f(x-\e_S)\right|^p\\
\lesssim_p \frac{k}{n}\sum_{j=1}^n\sum_{x\in \Z_{4m}^n} \left|f(x+e_j)-f(x)\right|^p+\frac{\left(k/n \right)^{\frac{p}{2}}}{2^n}\sum_{\e\in \{-1,1\}^n}\sum_{x\in \Z_{4m}^n}|f(x+\e)-f(x)|^p.
\end{multline*}
Recalling~\eqref{eq:beta approx cal E on S} and ~\eqref{eq:get m triangle beta}, we therefore have
\begin{align*}
\frac{(p^{3/2}/\log p)^{-p}\beta_p}{2^n\binom{n}{k}}&\sum_{\substack{S\subset\n\\|S|=k}}\sum_{\e\in \{-1,1\}^n}\sum_{x\in \Z_{4m}^n} \frac{|f(x+2m \e_S)-f(x)|^p}{m^p}\\
&\lesssim_p \frac{k}{n}\sum_{j=1}^n\sum_{x\in \Z_{4m}^n} \left|f(x+e_j)-f(x)\right|^p\\&\qquad+\left(1+\frac{(p^{3/2}/\log p)^{-p}}{m^p\left(k/n \right)^{\frac{p}{2}}}\beta_p\right)\frac{\left(k/n \right)^{\frac{p}{2}}}{2^n}\sum_{\e\in \{-1,1\}^n}\sum_{x\in \Z_{4m}^n}|f(x+\e)-f(x)|^p.\qedhere
\end{align*}
\end{proof}

\section{The Schatten $p$ trace class is an $X_p$ Banach space}\label{sec:appendix}

For $p\in [1,\infty)$ and $d\in \N$, the Schatten $p$-norm of a $d$ by $d$ matrix $A\in M_d(\R)$ is defined as
$$
\|A\|_{S_p}=\left(\trace\!\left((A^*A)^{\frac{p}{2}}\right)\right)^{\frac{1}{p}}=
\left(\trace\!\left((AA^*)^{\frac{p}{2}}\right)\right)^{\frac{1}{p}}.
$$
See~\cite{PX03} for relevant background. The following theorem asserts that $S_p$ is an $X_p$ Banach space.

\begin{theorem}\label{cor:schatten short} Fix $p\in [2,\infty)$, $d,n\in \N$ and $k\in \n$. Then for every $A_1,\ldots,A_n\in M_d(\R)$,
\begin{equation*}
\frac{(p/\sqrt{\log p})^{-p}}{2^n\binom{n}{k}}\sum_{\substack{S\subset\n\\|S|=k}}\sum_{\e\in \{-1,1\}^n}\left\|\sum_{j\in S} \e_j A_j\right\|_{S_p}^p\lesssim_p \frac{k}{n}\sum_{j=1}^n \|A_j\|_{S_p}^p+\frac{\left(k/n\right)^{\frac{p}{2}}}{2^n}\sum_{\e\in \{-1,1\}^n}\left\|\sum_{j=1}^n \e_j A_j\right\|_{S_p}^p.
\end{equation*}
\end{theorem}

\begin{question}
It remains an interesting open problem to determine whether or not the quantity $p/\sqrt{\log p}$ in Theorem~\ref{cor:schatten short} can be replaced by the (sharp) quantity $p/\log p$. This was proved in the scalar case in~\cite{JSZ85}, but additional ideas seem to be required in order to carry out the proof of~\cite{JSZ85} in the above noncommutative setting.
\end{question}

The key step in the proof of Theorem~\ref{cor:schatten short} is the following proposition.

\begin{proposition}\label{cor:all k}
Fix $q\in [1,\infty)$, $d,n\in \N$ and $k\in \n$. Suppose that $B_1,\ldots,B_n\in M_d(\R)$ are symmetric and positive semidefinite. Then
\begin{equation*}
\frac{1}{\binom{n}{k}}\sum_{\substack{S\subset \n\\|S|=k}}\trace\!\left(\Big(\sum_{j\in S} B_j\Big)^{q}\right)\lesssim_q \left(\frac{q}{\log(2q)}\right)^{q} \max\left\{\frac{k}{n}\sum_{j=1}^n\trace\!\left(B_j^{q}\right),\left(\frac{k}{n}\right)^{q}\trace\!\left(\Big(\sum_{j=1}^n B_j\Big)^{q}\right)\right\}.
\end{equation*}
\end{proposition}

Before proving Proposition~\ref{cor:all k}, we assume its validity for the moment and proceed to show how it implies Theorem~\ref{cor:schatten short}.

\begin{proof}[Proof of Theorem~\ref{cor:schatten short}] Lust-Piquard's noncommutative Khinchine inequality~\cite{Lus86} asserts that for every $S\subseteq\n$ we have
\begin{equation}\label{eq:lust piquard}
\frac{p^{-\frac{p}{2}}}{2^n}\sum_{\e\in \{-1,1\}^n}\left\|\sum_{j\in S} \e_j A_j\right\|_{S_p}^p\lesssim_p  \trace\!\left(\Big(\sum_{j\in S} A_j^*A_j\Big)^{\frac{p}{2}}\right)+\trace\!\left(\Big(\sum_{j\in S} A_jA_j^*\Big)^{\frac{p}{2}}\right).
\end{equation}
The (asymptotically optimal) dependence on $p$ in the left-hand side of~\eqref{eq:lust piquard} is not stated in Lust-Piquard's original proof of~\eqref{eq:lust piquard}, but it can be found in~\cite[page 106]{Pis98}.  By averaging~\eqref{eq:lust piquard} over all those $S\subset\n$ with $|S|=k$ we see that
\begin{multline}\label{eq:Luse upper noncommutative khinchine}
\frac{p^{-\frac{p}{2}}}{2^n\binom{n}{k}}\sum_{\substack{S\subset\n\\|S|=k}}\sum_{\e\in \{-1,1\}^n}\left\|\sum_{j\in S} \e_j A_j\right\|_{S_p}^p\\\lesssim_p
\frac{1}{\binom{n}{k}}\sum_{\substack{S\subset\n\\|S|=k}}\trace\!\left(\Big(\sum_{j\in S}^n A_j^*A_j\Big)^{\frac{p}{2}}\right)+\frac{1}{\binom{n}{k}}\sum_{\substack{S\subset\n\\|S|=k}}\trace\!\left(\Big(\sum_{j\in S} A_jA_j^*\Big)^{\frac{p}{2}}\right).
\end{multline}

Two applications of Proposition~\ref{cor:all k} with $q=p/2\ge 1$, once with $B_j=A_j^*A_j$ and once with $B_j=A_jA_j^*$, so as to control the two terms that appear in the right-hand side of~\eqref{eq:Luse upper noncommutative khinchine}, yield
\begin{multline}\label{eq:use positive twice}
\frac{(p/\sqrt{\log p})^{-p}}{2^n\binom{n}{k}}\sum_{\substack{S\subset\n\\|S|=k}}\sum_{\e\in \{-1,1\}^n}\left\|\sum_{j\in S} \e_j A_j\right\|_{S_p}^p\\ \lesssim_p
\frac{k}{n}\sum_{j=1}^n \|A_j\|_{S_p}^p+ \left(\frac{k}{n}\right)^{\frac{p}{2}}\trace\!\left(\Big(\sum_{j=1}^n A_j^*A_j\Big)^{\frac{p}{2}}\right)+\left(\frac{k}{n}\right)^{\frac{p}{2}}\trace\!\left(\Big(\sum_{j=1}^n A_jA_j^*\Big)^{\frac{p}{2}}\right).
\end{multline}
The other  direction of Lust-Piquard's noncommutative Khinchine inequality~\cite{Lus86}  asserts that
\begin{equation}\label{eq:easy direction of L-P}
\trace\!\left(\Big(\sum_{j=1}^n A_j^*A_j\Big)^{\frac{p}{2}}\right)+\trace\!\left(\Big(\sum_{j=1}^n A_jA_j^*\Big)^{\frac{p}{2}}\right)\lesssim_p \frac{1}{2^n}\sum_{\e\in \{-1,1\}^n}\left\|\sum_{j=1}^n \e_j A_j\right\|_{S_p}^p.
\end{equation}
Theorem~\ref{cor:schatten short} now follows by combining~\eqref{eq:use positive twice} and~\eqref{eq:easy direction of L-P}.
\end{proof}

Lemma~\ref{lem:positive Xp} below makes the same assertion as Proposition~\ref{cor:all k}, but only for $k\le n/2$ (and an explicit universal constant that arises from our proof; we do not claim that it is optimal). This is actually  the main step in the proof of Proposition~\ref{cor:all k}, which we will show below to easily follow from Lemma~\ref{lem:positive Xp}.

\begin{lemma}\label{lem:positive Xp} Fix $q\in [1,\infty)$ and $d,k,n\in \N$ with $k\le n/2$. Then for  every $B_1,\ldots,B_n\in M_d(\R)$ that are symmetric and positive semidefinite we have
\begin{equation*}\label{eq:desired in k less n/2}
\frac{1}{\binom{n}{k}}\sum_{\substack{S\subset \n\\|S|=k}}\trace\!\left(\Big(\sum_{j\in S} B_j\Big)^{q}\right)\lesssim \left(\frac{4q}{\log(2q)}\right)^{q} \max\left\{\frac{k}{n}\sum_{j=1}^n\trace\!\left(B_j^{q}\right),\left(\frac{k}{n}\right)^{q}
\trace\!\left(\Big(\sum_{j=1}^n B_j\Big)^{q}\right)\right\}.
\end{equation*}
\end{lemma}

Assuming the validity of Lemma~\ref{lem:positive Xp} for the moment,  we proceed to deduce Proposition~\ref{cor:all k}, which amounts to removing the restriction  $k\le n/2$ in Lemma~\ref{lem:positive Xp}.

\begin{proof}[Proof of Proposition~\ref{cor:all k}]
Write $k=u+v$ with $u,v\in \N$ satisfying $u,v\le n/2$. By the triangle inequality in $S_q$, for every $S,T\subset \n$ with  $T\subseteq S$ we have
$$
\trace\!\left(\Big(\sum_{j\in S} B_j\Big)^{q}\right)=\left\|\sum_{s\in T} B_{s}+\sum_{s\in S\setminus T} B_{s}\right\|_{S_q}^q \le 2^{q-1}\left\|\sum_{s\in T} B_{s}\right\|_{S_q}^q+2^{q-1}\left\|\sum_{s\in S\setminus T} B_{s}\right\|_{S_q}^q.
$$
Consequently,
\begin{multline}\label{eq:u v less than n/2}
 \frac{1}{\binom{n}{k}}\sum_{\substack{S\subset \n\\|S|=k}} \trace\!\left(\Big(\sum_{j\in S} B_j\Big)^{q}\right)
\le\frac{1}{\binom{n}{k}}\sum_{\substack{S\subset \n\\|S|=k}} \frac{2^{q-1}}{\binom{k}{u}}\sum_{\substack{T\subset S\\|T|=u}} \left(\left\|\sum_{s\in T} B_{s}\right\|_{S_q}^q+\left\|\sum_{s\in S\setminus T} B_{s}\right\|_{S_q}^q\right)\\
=\frac{2^{q-1}}{\binom{n}{u}}\sum_{\substack{U\subset \n\\|U|=u}} \trace\!\left(\Big(\sum_{j\in U} B_j\Big)^{q}\right)+\frac{2^{q-1}}{\binom{n}{v}}\sum_{\substack{V\subset \n\\|V|=v}} \trace\!\left(\Big(\sum_{j\in V} B_j\Big)^{q}\right).
\end{multline}
Proposition~\ref{cor:all k} now follows by applying Lemma~\ref{lem:positive Xp} to each of the summands that appear in the right-hand side of~\eqref{eq:u v less than n/2}.
\end{proof}

 Our proof of Lemma~\ref{lem:positive Xp} relies on certain matrix inequalities of independent interest. These inequalities are established in the following section.

\subsection{Auxiliary trace inequalities}\label{sec:trace} Proposition~\ref{lem:p>1} and Proposition~\ref{lem:q<1} below will be used crucially in the proof of Lemma~\ref{lem:positive Xp}. Note that the same statements are trivial when matrices are replaced by scalars. See Section~\ref{sec:discussion} for a discussion on the context of these results, where it is explained in particular that Proposition~\ref{lem:p>1} was known when $q\in [1,2]$ by either directly applying the work of Carlen and Lieb~\cite{CL08}, or through a simple argument that relies on operator convexity. At the same time, it is explained in Section~\ref{sec:discussion} that when $q\in (0,1)\cup(2,\infty)$, a range of values of $q$ that is used crucially in our proof of Lemma~\ref{lem:positive Xp}  below, Proposition~\ref{lem:p>1} exhibits a phenomenon that is qualitatively different from the simpler case $q\in [1,2]$.

\begin{proposition}\label{lem:p>1} Suppose that $q\in [1,\infty)$ and $d\in \N$. Then for every $A,B\in M_d(\R)$ that are symmetric and positive semidefinite we have
\begin{equation}\label{our main trace ineq}
\left(\trace\left((A+B)^qA\right)\right)^{\frac{1}{q}}\le \left(\trace\left(A^{q+1}\right)\right)^{\frac{1}{q}}+\left(\trace\left(B^{q}A\right)\right)^{\frac{1}{q}}.
\end{equation}
\end{proposition}

Before proving Proposition~\ref{lem:p>1}, we record for future use the following H\"older-type estimate.

\begin{lemma}\label{lem:holder} Fix $d,k\in \N$ and $q\in (0,\infty)$. Suppose that $a_0,\ldots,a_{k-1},b_1,\ldots,b_k\in (0,\infty)$ satisfy $b_j+b_{j+1}\le 2qa_j$ for every $j\in \{0,\ldots,k-1\}$,  where we set $b_0=b_k$. Suppose also that
\begin{equation}\label{eq:ajbj sum}
\sum_{j=0}^{k-1}a_j+\sum_{j=1}^kb_j=q+1.
 \end{equation}
 Then for every $A,B\in M_d(\R)$ that are symmetric and positive semidefinite we have
\begin{equation*}
\trace\left(A^{a_0}\left(\prod_{j=1}^{k-1}B^{b_j}A^{a_j}\right)B^{b_k}\right)\le
\left(\trace\left(A^{q+1}\right)\right)^{1-\frac{1}{q}\sum_{j=1}^kb_j}\left(\trace\left(B^qA\right)\right)^{\frac{1}{q}\sum_{j=1}^k b_j}.
\end{equation*}
\end{lemma}

\begin{proof}
By applying an arbitrarily small perturbation, we may assume that $a_j-(b_j+b_{j+1})/(2q)>0$ for every $j\in \{0,\ldots,k-1\}$. We can then define $p_0,r_0,\ldots,p_{k-1},r_{k-1}\in (0,\infty)$ by
\begin{equation}\label{eq:def pj rj}
\forall\, j\in \{0,\ldots,k-1\},\qquad p_j\eqdef \frac{q+1}{a_j-\frac{b_j+b_{j+1}}{2q}}\qquad \mathrm{and}\qquad r_j\eqdef \frac{q}{b_{j+1}}.
\end{equation}
Using the cyclicity of the trace, the choices in~\eqref{eq:def pj rj} imply that we have
\begin{equation}\label{eq:product decomposition}
\trace\left(A^{a_0}\left(\prod_{j=1}^{k-1}B^{b_j}A^{a_j}\right)B^{b_k}\right)=
\trace\left(\prod_{j=0}^{k-1}A^{\frac{q+1}{p_j}}
\left(A^{\frac{1}{2r_j}}B^{\frac{q}{r_j}}A^{\frac{1}{2r_j}}\right)
\right),
\end{equation}
Moreover,
\begin{equation*}
\sum_{j=0}^{k-1}\frac{1}{p_j}+\sum_{j=0}^{k-1} \frac{1}{r_j}=\frac{1}{q+1}\sum_{j=0}^{k-1}a_j+\left(\frac{1}{q}-\frac{1}{q(q+1)}\right)\sum_{j=1}^k b_j
\stackrel{\eqref{eq:ajbj sum}}{=}1.
\end{equation*}
Therefore $p_j,r_j\in (1,\infty)$ for all $j\in \{0,\ldots,k-1\}$ and we may use H\"older's inequality for traces (Th\'eor\`eme 6 in~\cite{Dix53}) to deduce from~\eqref{eq:product decomposition} that
\begin{equation}\label{eq:first use holder}
\trace\left(A^{a_0}\left(\prod_{j=1}^{k-1}B^{b_j}A^{a_j}\right)B^{b_k}\right)\le \prod_{j=0}^{k-1}\left(\trace\left(A^{q+1}\right)\right)^{\frac{1}{p_j}}
\left(\trace\left(\left(A^{\frac{1}{2r_j}}B^{\frac{q}{r_j}}A^{\frac{1}{2r_j}}\right)^{r_j}\right)\right)^{\frac{1}{r_j}}.
\end{equation}

The Lieb--Thirring inequality~\cite{LT76} asserts that $\trace((XYX)^r)\le \trace(X^rY^rX^r)$ for every $r\in [1,\infty)$ and for every symmetric and positive semidefinite matrices $X,Y\in M_d(\R)$. Recalling the definition of $r_0,\ldots,r_{k-1}$ in~\eqref{eq:def pj rj}, for every $j\in \{0,\ldots,k-1\}$ we therefore have
$$
\trace\left(\left(A^{\frac{1}{2r_j}}B^{\frac{q}{r_j}}A^{\frac{1}{2r_j}}\right)^{r_j}\right)\le \trace\left(\sqrt{A}B^q\sqrt{A}\right)=\trace\left(B^qA\right).
$$
A substitution of this estimate into~\eqref{eq:first use holder} gives
\begin{align}\label{eq:all the same bound 1}
\nonumber \trace\left(A^{a_0}\left(\prod_{j=1}^{k-1}B^{b_j}A^{a_j}\right)B^{b_k}\right)&\le \left(\trace\left(A^{q+1}\right)\right)^{\sum_{j=0}^{k-1}\frac{1}{p_j}}\left(\trace\left(B^qA\right)\right)^{\sum_{j=0}^{k-1}\frac{1}{r_j}}\\&=
\left(\trace\left(A^{q+1}\right)\right)^{1-\frac{1}{q}\sum_{j=1}^kb_j}\left(\trace\left(B^qA\right)\right)^{\frac{1}{q}\sum_{j=1}^k b_j},
\end{align}
where we used the fact that, due to~\eqref{eq:def pj rj}, we have $\sum_{j=0}^{k-1}\frac{1}{r_j}=\frac{1}{q}\sum_{j=1}^k b_j$.
\end{proof}

\begin{remark}\label{rem:aj>1}
For future use, note that if $q,a_0,\ldots,a_{k-1},b_1,\ldots,b_k\in (0,\infty)$ satisfy~\eqref{eq:ajbj sum} and we also know that $a_0,\ldots,a_{k-1}\ge 1$ then the assumptions of Lemma~\ref{lem:holder} hold true, i.e.,  $b_{j}+b_{j+1}\le 2q a_j$ for every $j\in \{0,\ldots,k-1\}$. Indeed,  by~\eqref{eq:ajbj sum} we have $\max\{b_j,b_{j+1}\}\le q+1-a_j\le q a_j$, and consequently $b_j+b_{j+1}\le 2\max\{b_j,b_{j+1}\}\le 2q a_j$.
\end{remark}

\begin{proof}[Proof of Proposition~\ref{lem:p>1}]
Write $q=2m+\theta$, where $m\in \N\cup\{0\}$ and $\theta\in (0,2]$. The proof of~\eqref{our main trace ineq} treats the cases $\theta\in (0,1)$ and $\theta\in [1,2]$ differently.

\medskip

\noindent{\bf Case 1: $\theta\in [1,2]$.} In this range the mapping $t\to t^\theta$ is operator-convex (see Theorem 2.6 in~\cite{Car10}). This means that for every $s\in (0,1)$ we have
\begin{equation}\label{eq:use PSD theta}
(A+B)^\theta=\left(s\frac{A}{s}+(1-s)\frac{B}{1-s}\right)^\theta\le \frac{A^\theta}{s^{\theta-1}}+\frac{B^\theta}{(1-s)^{\theta-1}},
\end{equation}
where, as usual, we interpret the inequality~\eqref{eq:use PSD theta} in terms of the PSD order of matrices, i.e., that the right-hand side of~\eqref{eq:use PSD theta} minus the left-hand side of~\eqref{eq:use PSD theta} is a positive semidefinite matrix.

It follows from~\eqref{eq:use PSD theta} that
\begin{equation*}
\sqrt{A}(A+B)^q\sqrt{A}\le  \frac{\sqrt{A}(A+B)^mA^\theta(A+B)^m\sqrt{A}}{s^{\theta-1}}+
\frac{\sqrt{A}(A+B)^mB^\theta(A+B)^m\sqrt{A}}{(1-s)^{\theta-1}}.
\end{equation*}
So, by taking traces while making use of the cyclicity of the trace, we see that
\begin{equation}\label{eq:took traces theta}
\trace\left((A+B)^qA\right)\le \frac{\trace\left((A+B)^mA^\theta(A+B)^mA\right)}{s^{\theta-1}}+
\frac{\trace\left((A+B)^mB^\theta(A+B)^mA\right)}{(1-s)^{\theta-1}}.
\end{equation}
By choosing $s$ so as to minimize the quantity appearing in the right-hand side of~\eqref{eq:took traces theta}, we have
\begin{multline}\label{eq:s chosen}
\left(\trace\left((A+B)^qA\right)\right)^{\frac{1}{\theta}}\\\le \left(\trace\left((A+B)^mA^\theta(A+B)^mA\right)\right)^{\frac{1}{\theta}}
+\left(\trace\left((A+B)^mB^\theta(A+B)^mA\right)\right)^{\frac{1}{\theta}}.
\end{multline}

We shall now proceed to estimate each of the terms that appear in the right-hand side of~\eqref{eq:s chosen} separately. By expanding the $m$th powers appearing in the matrix $(A+B)^mA^\theta(A+B)^mA$, and using the cyclicity of the trace, we see that $\trace\left((A+B)^mA^\theta(A+B)^mA\right)$ equals the sum of $2^{2m}$ terms, each of which is of the form
\begin{equation}\label{eq:alpha beta term}
\trace\left(A^{a_0}\left(\prod_{j=1}^{k-1}B^{b_j}A^{a_j}\right)B^{b_k}\right),
\end{equation}
for some $k\in \N\cup\{0\}$ and $a_0,\ldots,a_{k-1},b_1,\ldots,b_k\in (0,\infty)$ that satisfy~\eqref{eq:ajbj sum} (recall that $q=2m+\theta$). Here we use the convention that when $k=0$ the quantity appearing in~\eqref{eq:alpha beta term} equals $\trace(A^{q+1})$. Note that $b_j$ is an integer for every $j\in \{1,\ldots,k\}$, and for every $r\in \{0,\ldots,2m\}$ the number of terms of the form~\eqref{eq:alpha beta term} that appear in the above expansion of $\trace\left((A+B)^mA^\theta(A+B)^mA\right)$ with $\sum_{j=1}^k b_j=r$ equals $\binom{2m}{r}$; this is because $\sum_{j=1}^k b_j$ is the total number of times that $B$ was chosen when one expands the two occurrences of $(A+B)^m$ in $(A+B)^mA^\theta(A+B)^mA$ as a product of matrices, each of which is either $A$ or $B$. Note also that $a_0,\ldots,a_{k-1}\ge 1$, since $\theta\ge 1$. Recalling Remark~\ref{rem:aj>1}, we may therefore use Lemma~\ref{lem:holder} to deduce that
\begin{align}\label{eq:all the same bound 1}
 \trace\left(A^{a_0}\left(\prod_{j=1}^{k-1}B^{b_j}A^{a_j}\right)B^{b_k}\right)\le
\left(\trace\left(A^{q+1}\right)\right)^{1-\frac{1}{q}\sum_{j=1}^kb_j}\left(\trace\left(B^qA\right)\right)^{\frac{1}{q}\sum_{j=1}^k b_j}.
\end{align}
Hence,
\begin{align}\label{eq:final theta trace bound A}
\nonumber \trace\left((A+B)^mA^\theta(A+B)^mA\right)&\le \sum_{r=0}^{2m} \binom{2m}{r}\left(\trace\left(A^{q+1}\right)\right)^{1-\frac{r}{q}}\left(\trace\left(B^qA\right)\right)^{\frac{r}{q}}\\
&= \left(\trace\left(A^{q+1}\right)\right)^{1-\frac{2m}{q}}\left(\left(\trace\left(A^{q+1}\right)\right)^{\frac{1}{q}}+
\left(\trace\left(B^{q}A\right)\right)^{\frac{1}{q}}\right)^{2m}\nonumber\\
&=\left(\trace\left(A^{q+1}\right)\right)^{\frac{\theta}{q}}\left(\left(\trace\left(A^{q+1}\right)\right)^{\frac{1}{q}}+
\left(\trace\left(B^{q}A\right)\right)^{\frac{1}{q}}\right)^{2m},
\end{align}
where in the final step we used the fact that $2m+\theta=q$.

The second term in the right-hand side of~\eqref{eq:s chosen} is bounded using similar reasoning. As before,  $\trace\left((A+B)^mB^\theta(A+B)^mA\right)$ equals the sum of terms as in~\eqref{eq:alpha beta term}, for some $k\in \N\cup\{0\}$ and $a_0,\ldots,a_{k-1},b_1,\ldots,b_k\in (0,\infty)$ that satisfy~\eqref{eq:ajbj sum}. However, now we know that $a_1,\ldots,a_{k-1}\in \N$ and $\sum_{j=0}^{k-1}b_j-\theta\in \{0,\ldots,2m\}$. By Lemma~\ref{lem:holder} (and Remark~\ref{rem:aj>1}), the estimate~\eqref{eq:all the same bound 1} holds true for the terms of the form~\eqref{eq:alpha beta term}   that appear in the expansion of $\trace\left((A+B)^mB^\theta(A+B)^mA\right)$. For every $r\in \{0,\ldots,2m\}$, the number of terms of the form~\eqref{eq:alpha beta term} that appear in the expansion of $\trace\left((A+B)^mB^\theta(A+B)^mA\right)$ with $\sum_{j=1}^k b_j=r+\theta$ equals $\binom{2m}{r}$, so by~\eqref{eq:all the same bound 1} we have
\begin{align}\label{eq:final theta trace bound B}
 \nonumber\trace&\left((A+B)^mB^\theta(A+B)^mA\right)\\\nonumber &\le \sum_{r=0}^{2m} \binom{2m}{r}\left(\trace\left(A^{q+1}\right)\right)^{1-\frac{r+\theta}{q}}\left(\trace\left(B^qA\right)\right)^{\frac{r+\theta}{q}}\\
&= \left(\trace\left(A^{q+1}\right)\right)^{1-\frac{2m+\theta}{q}}\left(\trace\left(B^{q}A\right)\right)^{\frac{\theta}{q}}\left(\left(\trace\left(A^{q+1}\right)\right)^{\frac{1}{q}}+
\left(\trace\left(B^{q}A\right)\right)^{\frac{1}{q}}\right)^{2m}\nonumber\\
&=\left(\trace\left(B^{q}A\right)\right)^{\frac{\theta}{q}}\left(\left(\trace\left(A^{q+1}\right)\right)^{\frac{1}{q}}+
\left(\trace\left(B^{q}A\right)\right)^{\frac{1}{q}}\right)^{2m},
\end{align}
where the last step uses the fact that $2m+\theta=q$.

By substituting~\eqref{eq:final theta trace bound A} and~\eqref{eq:final theta trace bound B} into~\eqref{eq:s chosen} we see that
$$
\left(\trace\left((A+B)^qA\right)\right)^{\frac1{\theta}}\le \left(\left(\trace\left(A^{q+1}\right)\right)^{\frac{1}{q}}+
\left(\trace\left(B^{q}A\right)\right)^{\frac{1}{q}}\right)^{1+\frac{2m}{\theta}}=\left(\left(\trace\left(A^{q+1}\right)\right)^{\frac{1}{q}}+
\left(\trace\left(B^{q}A\right)\right)^{\frac{1}{q}}\right)^{\frac{q}{\theta}},
$$
using $2m+\theta=q$ once more. This completes  the proof of the desired estimate~\eqref{our main trace ineq} in Case 1.

\medskip
\noindent{\bf Case 2: $\theta\in (0,1)$.} Note that since the underlying assumption of Proposition~\ref{lem:p>1}  is that $q\ge 1$, the facts that $q=2m+\theta$ and $\theta\in (0,1)$ imply that the integer $m$ is positive. Moreover, in the range $\theta\in (0,1)$ the mapping $t\to t^\theta$ is no longer operator-convex but we have the following commonly used (see e.g.~\cite{Eps73})  integral representation at our disposal. Since for every $a\in (0,\infty)$ we have
$$
a^\theta=\frac{\sin(\pi \theta)}{\pi}\int_0^\infty t^\theta\left(\frac{1}{t}-\frac{1}{t+a}\right)dt,
$$
it follows that for every $s\in (0,\infty)$,
\begin{equation}\label{eq:integral identity less than 1}
(sA+B)^\theta=\frac{\sin(\pi \theta)}{\pi}\int_0^\infty t^{\theta}\left(\frac{1}{t}I-\left(tI+sA+B\right)^{-1}\right)dt.
\end{equation}
Since $\frac{d}{dt}X(t)^{-1}=-X(t)^{-1}X'(t)X(t)^{-1}$ for every differentiable $X:(0,\infty)\to M_d(\R)$ such that $X(t)$ is an invertible matrix for every $t\in (0,\infty)$ (simply differentiate the identity $X(t)^{-1}X(t)=I$), it follows from~\eqref{eq:integral identity less than 1} that
\begin{equation}\label{eq:derivative less than 1}
\frac{d}{ds}(sA+B)^\theta =\frac{\sin(\pi \theta)}{\pi}\int_0^\infty t^\theta\left(tI+sA+B\right)^{-1}A\left(tI+sA+B\right)^{-1} dt.
\end{equation}

By integrating over $s\in [0,1]$, in order to prove~\eqref{our main trace ineq} it will suffice to show that
$$
\forall\, s\in (0,1),\qquad \frac{d}{ds} \left(\trace\left((sA+B)^qA\right)\right)^{\frac{1}{q}}\le \left(\trace\left(A^{q+1}\right)\right)^{\frac{1}{q}}.
$$
Equivalently, we want to prove that
\begin{equation}\label{eq:goal derivative}
\forall\, s\in (0,1),\qquad \frac{d}{ds} \trace\left((sA+B)^qA\right)\le q\left(\trace\left(A^{q+1}\right)\right)^{\frac{1}{q}}\left(\trace\left((sA+B)^qA\right)\right)^{1-\frac{1}{q}}.
\end{equation}

Define for every $s\in (0,1)$,
$$
f(s)\eqdef \trace\left(\left(\frac{d}{ds}(sA+B)^m\right)(sA+B)^{m+\theta}A\right),
$$
\begin{equation*}
g(s)\eqdef \trace\left((sA+B)^{m}\left(\frac{d}{ds}(sA+B)^\theta\right)(sA+B)^{m}A\right),
\end{equation*}
and
$$
h(s)\eqdef \trace\left((sA+B)^{m+\theta}\left(\frac{d}{ds}(sA+B)^m\right)A\right).
$$
Then, since $(sA+B)^q=(sA+B)^m(sA+B)^\theta(sA+B)^m$, we have
$$
\frac{d}{ds} \trace\left((sA+B)^qA\right)=f(s)+g(s)+h(s).
$$
Hence, because $q=2m+\theta$, in order to establish the validity of~\eqref{eq:goal derivative} it suffice to show that for every $s\in [0,1]$ we have
\begin{equation}\label{eq:max f h}
\max\{f(s),h(s)\}\le m\left(\trace\left(A^{q+1}\right)\right)^{\frac{1}{q}}\left(\trace\left((sA+B)^qA\right)\right)^{1-\frac{1}{q}},
\end{equation}
and
\begin{equation}\label{eq:goal g upper}
g(s)\le \theta\left(\trace\left(A^{q+1}\right)\right)^{\frac{1}{q}}\left(\trace\left((sA+B)^qA\right)\right)^{1-\frac{1}{q}}.
\end{equation}

Observe that
\begin{align}\label{eq:f der}
\nonumber f(s)
&=\sum_{r=1}^m \trace\left((sA+B)^{r-1}A(sA+B)^{m-r}(sA+B)^{m+\theta}A\right)\\&=\sum_{r=1}^m \trace\left((sA+B)^{r-1}A(sA+B)^{q-r}A\right).
\end{align}
Similarly, using the cyclicity of the trace, we have
\begin{equation}\label{eq:h=f}
h(s)=\sum_{r=1}^m \trace\left((sA+B)^{q-r}A(sA+B)^{r-1}A\right)=f(s).
\end{equation}
Finally, by the integral representation~\eqref{eq:derivative less than 1}  we have
\begin{equation}\label{eq:g der}
g(s)=\frac{\sin(\pi \theta)}{\pi}\int_0^\infty t^\theta \trace\left((sA+B)^m\left(tI+sA+B\right)^{-1}A\left(tI+sA+B\right)^{-1}(sA+B)^mA\right)dt.
\end{equation}

By denoting $C=sA+B$, it follows from~\eqref{eq:f der}, \eqref{eq:h=f} and~\eqref{eq:g der} that the desired estimates~\eqref{eq:max f h} and~\eqref{eq:goal g upper} will be proven once we show that for every $C\in M_d(\R)$ that is symmetric and positive semidefinite we have
\begin{equation}\label{eq:goal each r in sum}
\forall\, r\in \{1,\ldots,m\},\qquad \trace\left(C^{r-1}AC^{q-r}A\right)\le \left(\trace\left(A^{q+1}\right)\right)^{\frac{1}{q}}\left(\trace\left(C^qA\right)\right)^{1-\frac{1}{q}},
\end{equation}
and
\begin{equation}\label{eq:goal integral r}
\int_0^\infty t^\theta \trace\left(C^m\left(tI+C\right)^{-1}A\left(tI+C\right)^{-1}C^mA\right)dt\le \frac{\pi\theta}{\sin(\pi \theta)} \left(\trace\left(A^{q+1}\right)\right)^{\frac{1}{q}}\left(\trace\left(C^qA\right)\right)^{1-\frac{1}{q}}.
\end{equation}

\eqref{eq:goal each r in sum} is a consequence of Lemma~\ref{lem:holder} (with $B$ replaced by $C$). It therefore remains to establish the validity of~\eqref{eq:goal integral r}. To this end, note that for every $t\in (0,\infty)$, since $(tI+C)^{-1}$ and $C^m$ commute, by the cyclicity of the trace we have
\begin{multline}\label{eq:lieb thirring 2}
\trace\left(C^m\left(tI+C\right)^{-1}A\left(tI+C\right)^{-1}C^mA\right)=\trace\left(\left(\sqrt{A}C^m\left(tI+C\right)^{-1}\sqrt{A}\right)^2\right)\\\le
\trace\left(AC^{2m}\left(tI+C\right)^{-2}A\right)=\trace\left(C^{2m}\left(tI+C\right)^{-2}A^2\right),
\end{multline}
where for the inequality in~\eqref{eq:lieb thirring 2} we used the Lieb--Thirring inequality. This upper bound on the integrand in the left-hand side of~\eqref{eq:goal integral r} yields the following estimate.
\begin{equation}\label{eq:integral -2}
\int_0^\infty t^\theta \trace\left(C^m\left(tI+C\right)^{-1}A\left(tI+C\right)^{-1}C^mA\right)dt\le \trace\left(C^{2m}\left(\int_0^\infty t^\theta (tI+C)^{-2}dt\right)A^2\right).
\end{equation}
Note that for every $c\in (0,\infty)$ we have
$$
\int_0^\infty \frac{t^\theta}{(t+c)^2}dt=c^{\theta-1}\int_0^\infty \frac{s^\theta}{(s+1)^2}ds=\frac{\pi \theta}{\sin(\pi \theta)}c^{\theta-1}.
$$
Consequently,
\begin{multline}\label{eq:use integral -2 for C}
\trace\left(C^{2m}\left(\int_0^\infty t^\theta (tI+C)^{-2}dt\right)A^2\right)=\frac{\pi \theta}{\sin(\pi \theta)}\trace\left(C^{2m+\theta-1}A^2\right)
=\frac{\pi \theta}{\sin(\pi \theta)}\trace\left(C^{q-1}A^2\right)\\=\frac{\pi \theta}{\sin(\pi \theta)}\trace\left(AC^{q-1}A\right)\stackrel{\eqref{eq:goal each r in sum}}{\le}
\frac{\pi \theta}{\sin(\pi \theta)}\left(\trace\left(A^{q+1}\right)\right)^{\frac{1}{q}}\left(\trace\left(C^qA\right)\right)^{1-\frac{1}{q}}.
\end{multline}
A substitution of~\eqref{eq:use integral -2 for C} into~\eqref{eq:integral -2} yields the desired inequality~\eqref{eq:goal integral r}.
\end{proof}

The following Proposition is a variant of Proposition~\ref{lem:p>1} when $q\in (0,1)$.

\begin{proposition}\label{lem:q<1}
Suppose that $q\in (0,1)$ and $d\in \N$. Then for every $A,B\in M_d(\R)$ that are symmetric and positive semidefinite we have
$$
\trace\left((A+B)^qA\right)\le \trace\left(A^{q+1}\right)+\trace\left(B^qA\right).
$$
\end{proposition}

\begin{proof} By the integral identity~\eqref{eq:derivative less than 1}, with $\theta$ replaced by $q$ (which is allowed since $0<q<1$), for every $s\in (0,\infty)$ we have
\begin{equation}\label{eq:derivative less than 1-new}
\frac{d}{ds}\trace\left((sA+B)^qA\right)=\frac{\sin(\pi q)}{\pi}\int_0^\infty t^q\trace\left(\left(tI+sA+B\right)^{-1}A\left(tI+sA+B\right)^{-1}A\right) dt.
\end{equation}

Fix $s,t\in (0,\infty)$ and define $F:[0,\infty)\to \R$ by
$$
\forall\, w\in [0,\infty),\qquad F(w)\eqdef \trace\left(\left(tI+sA+wB\right)^{-1}A\left(tI+sA+wB\right)^{-1}A\right).
$$
This mapping was investigated in Section III of~\cite{BCL94}, where it was shown to be convex. Here we need to know that it is non-increasing,  which follows from the following computation.
\begin{align*}
F'(w)&=-\trace\left(\left(tI+sA+wB\right)^{-1}B\left(tI+sA+wB\right)^{-1}A\left(tI+sA+wB\right)^{-1}A\right)\\&\qquad-
\trace\left(\left(tI+sA+wB\right)^{-1}A\left(tI+sA+wB\right)^{-1}B\left(tI+sA+wB\right)^{-1}A\right)\\
&=-\trace(CD)-\trace(DC)=-2\trace(CD)\le 0,
\end{align*}
where $C,D\in M_d(\R)$ are the symmetric and positive semidefinite matrices given by
$$
C\eqdef \left(tI+sA+wB\right)^{-1}B\left(tI+sA+wB\right)^{-1}\qquad\mathrm{and}\qquad D\eqdef A\left(tI+sA+wB\right)^{-1}A.
$$
It follows from these considerations that
$$
\trace\left(\left(tI+sA+B\right)^{-1}A\left(tI+sA+B\right)^{-1}A\right)=F(1)\le F(0)=\trace\left(\left(tI+sA\right)^{-1}A\left(tI+sA\right)^{-1}A\right).
$$
A substitution of this estimate into~\eqref{eq:derivative less than 1-new} shows that
\begin{multline}\label{eq:to integrate s}
\frac{d}{ds}\trace\left((sA+B)^qA\right)\le \frac{\sin(\pi q)}{\pi}\int_0^\infty t^q\trace\left(\left(tI+sA\right)^{-1}A\left(tI+sA\right)^{-1}A\right) dt\\\stackrel{\eqref{eq:integral identity less than 1}}{=}\frac{d}{ds}\trace\left((sA)^qA\right)=qs^{q-1}\trace\left(A^{q+1}\right).
\end{multline}
By integrating~\eqref{eq:to integrate s} over $[0,1]$ we therefore see that
$
\trace\left((A+B)^qA\right)-\trace\left(B^qA\right)\le \trace\left(A^{q+1}\right).
$
\end{proof}

We record for future use the following simple reformulation of Proposition~\ref{lem:p>1} and Proposition~\ref{lem:q<1}. When $q\in [1,2]$ it follows from Proposition~\ref{lem:q<1} (with $q$ replaced by $q-1$), and when $q>2$ it follows from Proposition~\ref{lem:p>1} (with $q$ replaced by $q-1$) and the convexity of $t\mapsto t^{q-1}$ on $[0,\infty)$.

\begin{corollary}\label{cor:reformulation of lemmas for use}
Suppose that $q\in [1,\infty)$ and $d\in \N$. Set $r\eqdef\max\{q-2,0\}$. For every $A,B\in M_d(\R)$ that are symmetric and positive semidefinite we have
\begin{equation*}
\trace\left((A+B)^{q-1}A\right)\le  \min\left\{\frac{\trace\left(A^{q}\right)}{\lambda^r}+\frac{\trace\left(B^{q-1}A\right)}{(1-\lambda)^r}:\  \lambda\in (0,1)\right\}.
\end{equation*}
\end{corollary}

\subsubsection{Discussion and counterexamples}\label{sec:discussion} An inspection of our proof of Lemma~\ref{lem:positive Xp} below shows that, for $p\in (2,\infty)$,  what we really need in order to show that $S_p$ is an $X_p$ Banach space is that there exists $K=K_p\in (0,\infty)$ such that if $A,B\in M_d(\R)$ are symmetric and positive semidefinite then
\begin{equation}\label{eq:trace ineq need}
\trace\!\left((A+B)^{\frac{p}{2}-1}A\right)\le K\left(\trace\!\left(A^{\frac{p}{2}-1}A\right)+\trace\!\left(B^{\frac{p}{2}-1}A\right)\right).
\end{equation}
Specifically, \eqref{eq:trace ineq need} implies Theorem~\ref{cor:schatten short} with the term $p/\sqrt{\log p}$ replaced by a constant that depends only on $K$ and $p$. By Corollary~\ref{cor:reformulation of lemmas for use}, \eqref{eq:trace ineq need} holds true
with $K=2^{\max\{0,(p-4)/2\}}$.

 Setting $q=(p-2)/2>0$, it is natural to ask whether multiplication by $A$ is crucial for~\eqref{eq:trace ineq need} to hold true. Specifically, one would naturally investigate whether for every $A,B,C\in M_d(\R)$ that are symmetric and positive semidefinite we have
\begin{equation}\label{eq:trace ineq need-C}
\trace\!\left((A+B)^{q}C\right)\le K\left(\trace\!\left(A^{q}C\right)+\trace\!\left(B^{q}C\right)\right),
\end{equation}
with $K\in (0,\infty)$ independent of $A,B,C$. By a simple duality argument (e.g. Lemma~5.12 in~\cite{Car10}), the above requirement is equivalent to the matrix inequality
\begin{equation}\label{eq:hope subadditive PSD}
(A+B)^q\le K\left(A^q+B^q\right),
\end{equation}
where, as usual, we interpret the inequality~\eqref{eq:hope subadditive PSD} in terms of the PSD order of matrices.

Since for $q\in [1,2]$ the function $t\mapsto t^q$ is operator-convex (see e.g.~\cite{Bha97}), for such $q$ the PSD inequality~\eqref{eq:hope subadditive PSD} holds true with $K=2^{q-1}$ (recall~\eqref{eq:use PSD theta}). This yields a simple proof of~\eqref{eq:trace ineq need} when $4\le p\le 6$. Moreover, when  $q\in [1,2]$  the operator convexity of the function $t\mapsto t^q$ shows that  if $A,B,C\in M_d(\R)$ are symmetric and positive semidefinite then for every $\lambda\in (0,1)$ we have
\begin{equation}\label{eq:to minimize lambda}
\trace\!\left((A+B)^{q}C\right)\le \frac{\trace\!\left(A^{q}C\right)}{\lambda^{q-1}}+\frac{\trace\!\left(B^{q}C\right)}{(1-\lambda)^{q-1}}.
\end{equation}
By choosing $\lambda$ so as to minimize the right hand side of~\eqref{eq:to minimize lambda} we see that
\begin{equation}\label{eq:homogoeneous with C}
\left(\trace\left((A+B)^qC\right)\right)^{\frac{1}{q}}\le \left(\trace\left(A^{q}C\right)\right)^{\frac{1}{q}}+\left(\trace\left(B^{q}C\right)\right)^{\frac{1}{q}}.
\end{equation}
The inequality~\eqref{eq:homogoeneous with C}  is  a strengthening of Proposition~\ref{lem:p>1}  in the special case $q\in [1,2]$, showing that when $q$ belongs to this range  Proposition~\ref{lem:p>1}  is  a simple consequence of the operator convexity of the function $t\mapsto t^q$ (alternatively, one can deduce Proposition~\ref{lem:p>1} directly from the work of Carlen and Lieb~\cite{CL08}; see specifically Theorem~1.1 and Remark~1.2 in~\cite{CL08}). However, the above argument is  special to the range $q\in [1,2]$ since, as we shall explain below, if $q\in (0,1)\cup(2,\infty)$ then~\eqref{eq:hope subadditive PSD} fails to hold true with {\em any} constant $K$ that is independent of $A,B$.

The failure of such PSD subadditivity inequalities prompted much work in search for substitutes (note, however, that the literature did not focus on inequalities that allow for an arbitrary constant $K$ in~\eqref{eq:hope subadditive PSD}, but was rather devoted to, e.g., finding substitutes for~\eqref{eq:hope subadditive PSD} with $q\in (0,1)$ and $K=1$). One such substitute allows for conjugation by unitary matrices, as initiated in~\cite{AAP82}. A satisfactory recent result~\cite{AB07} along these lines asserts that if $f:[0,\infty)\to \R$ is nondecreasing, concave, and $f(0)\ge 0$, then for every $A,B\in M_d(\R)$ there exist unitary matrices $U,V\in M_d(\C)$ such that $f(A+B)\le Uf(A)U^*+Vf(B)V^*$. Another substitute for PSD subadditivity is a subadditivity inequality for unitarily invariant norms. Recall that a norm $\|\cdot\|$ on $M_d(\C)$ is unitarily invariant if $\|UXV\|=\|X\|$ for every $X,U,V\in M_d(\C)$ such that $U,V$ are unitary. The papers~\cite{AZ99,BU07} contain satisfactory results along these lines, obtaining inequalities of the form $\|f(A+B)\|\le \|f(A)+f(B)\|$. For $q\in (0,1)$, when $f(t)=t^q$ and $\|\cdot\|$ is the Schatten $1$ norm, the resulting inequality goes back to~\cite{McC67} and it corresponds to~\eqref{eq:trace ineq need-C} with $C=I$ (and $K=1$).

Here we study a different type of substitute for~\eqref{eq:hope subadditive PSD}. For example, when $A\in M_d(\R)$ is symmetric and positive semidefinite define $F_A:M_d(\R)\to \R$ by $F_A(X)=(\trace(|X|^qA))^{1/q}$ ($F_A$ need not be unitarily invariant). Proposition~\ref{lem:p>1} asserts that if $q>1$ then $F_A(X+Y)\le F_A(X)+F_A(Y)$ for symmetric and positive semidefinite $X,Y\in M_d(\R)$, provided that either $X$ or $Y$ equals $A$. Weakenings  of~\eqref{eq:trace ineq need-C} (the special case $C=A$)  suffice for our application (i.e., proving the $X_p$ inequality for $S_p$, and consequently obtaining various nonembeddability results), but we believe that they are interesting in their own right and deserve further investigation. Possible extensions include understanding inequalities of the form $\trace(f(A+B)A)\le K\trace(f(A)A)+K\trace(f(B)A)$.

\medskip

We shall end this discussion by presenting the aforementioned example that exhibits the failure of~\eqref{eq:hope subadditive PSD} for every $q\in (0,1)\cup(2,\infty)$ and $K\in (0,\infty)$. Fix $s\in (0,\infty)$ which we will eventually take to be sufficiently small. Define $A_s,B_s\in M_2(\R)$ and $w_s\in \R^2$ by
$$
A_s\eqdef \left(\begin{array}{cccc}
s^2 & 0 \\
0 &  0
\end{array}\right), \qquad  B_s\eqdef\left(\begin{array}{cccccc}
1 & s  \\
s &  s^2
\end{array}\right)\qquad \mathrm{and}\qquad w_s\eqdef \left(\begin{array}{cccccc}
-s   \\
1
\end{array}\right).
$$
$A_s$ and $B_s$ are symmetric and positive semidefinite, yet by direct computation for every $K\in (0,\infty)$,
$$
\left\langle \left(K(A_s^4+B_s^4)-(A_s+B_s)^4\right)w_s,w_s\right\rangle=-s^6-3s^8+(K-1)s^{10}.
$$
The above quantity is negative for $s<1/\sqrt[4]{K}$, in which case the matrix  $K(A_s^4+B_s^4)-(A_s+B_s)^4$ is not positive semidefinite. This shows that~\eqref{eq:hope subadditive PSD} fails to hold true for $q=4$ with any constant $K\in (0,\infty)$ that is independent of $A$ and $B$ (this corresponds to the failure of~\eqref{eq:trace ineq need} when $p=10$). A similar, though more tedious, computation shows that~\eqref{eq:hope subadditive PSD} also fails for every $q\in (0,1)\cup(2,\infty)$. Indeed, direct computation (via diagonalization) yields that $
A_s^q=s^{2q}A_s$, $B_s^q=(1+s^2)^{q-1}B_s$ and
$$
(A_s+B_s)^q=
\left(\begin{array}{cccccc}
\frac{a(s)^q\left(\sqrt{1+4s^2}+1\right)+b(s)^q\left(\sqrt{1+4s^2}-1\right)}{{2\sqrt{1+4s^2}}} & \frac{s\left(a(s)^q-b(s)^q\right)}{\sqrt{1+4s^2}}  \\
\frac{s\left(a(s)^q-b(s)^q\right)}{\sqrt{1+4s^2}} &  \frac{a(s)^q\left(\sqrt{1+4s^2}-1\right)+b(s)^q\left(\sqrt{1+4s^2}+1\right)}{2\sqrt{1+4s^2}}
\end{array}\right),
$$
where
$$
a(s)\eqdef s^2+\frac12+\frac{\sqrt{1+4s^2}}{2}\qquad\mathrm{and}\qquad b(s)\eqdef s^2+\frac12-\frac{\sqrt{1+4s^2}}{2}.
$$
One then directly computes that as $s\to 0$,
\begin{equation}\label{eq:3 dependencies}
\left\langle \left(K(A_s^q+B_s^q)-(A_s+B_s)^q\right)w_s,w_s\right\rangle=\left(Ks^{2(q+1)}-s^6-s^{4q}\right)\left(1+O_{q,K}(s^2)\right).
\end{equation}
When $q\in (0,1)$ we have $4q<\min\{2(q+1),6\}$ and when $q\in (2,\infty)$ we have $6<\min\{2(q+1),4q\}$. Consequently, for $q\in (0,1)\cup (2,\infty)$ the quantity appearing in~\eqref{eq:3 dependencies} is negative for small enough $s$, which means that the matrix $K(A_s^q+B_s^q)-(A_s+B_s)^q$ is not positive semidefinite.

\subsection{Proof of Lemma~\ref{lem:positive Xp}} For the sake of simplicity denote
\begin{equation}\label{eq:def UVW}
U\eqdef \frac{1}{\binom{n}{k}}\sum_{\substack{S\subset \n\\|S|=k}} \trace\!\left(\Big(\sum_{j\in S} B_j\Big)^{q}\right),\quad V\eqdef \frac{k}{n}\sum_{j=1}^n \trace\!\left(B_j^{q}\right),\quad W\eqdef  \trace\!\left(\Big(\frac{k}{n}\sum_{j=1}^n B_j\Big)^{q}\right).
\end{equation}
Our goal is therefore to show that
\begin{equation}\label{eq:goal with UVW}
U\le \left(\frac{4q}{\log(2q)}\right)^q\max\{V,W\}.
\end{equation}

Fix $\lambda \in (0,1)$ to be specified later. For every $S\subset \n$ and $j\in S$, by Corollary~\ref{cor:reformulation of lemmas for use}, with $A=B_j$ and
$B=\sum_{s\in S\setminus\{j\}} B_s$, we have
$$
\trace\!\left(\Big(\sum_{s\in S} B_s\Big)^{q-1}B_j\right)\le \frac{1}{\lambda^r}\trace\left(B_j^q\right)+\frac{1}{(1-\lambda)^r}\trace\!\left(\Big(\sum_{s\in S\setminus\{j\}} B_s\Big)^{q-1}B_j\right),
$$
where, as denoted in Corollary~\ref{cor:reformulation of lemmas for use}, $r=\max\{q-2,0\}$. Hence,
\begin{align}\label{eq:lambda convexity}
\trace\!\left(\Big(\sum_{j\in S} B_j\Big)^{q}\right)&=\sum_{j=1}^n\trace\!\left(\Big(\sum_{j\in S} B_j\Big)^{q-1}B_j\right)\nonumber\\&\le \frac{1}{\lambda^r}\sum_{j\in S}\trace\left(B_j^q\right)+\frac{1}{(1-\lambda)^r}\sum_{j\in S}  \trace\!\left(\Big(\sum_{s\in S\setminus\{j\}} B_s\Big)^{q-1}B_j\right).
\end{align}
By averaging~\eqref{eq:lambda convexity} over all of those $S\subset \n$ with $|S|=k$, and recalling~\eqref{eq:def UVW}, we see that
\begin{equation}\label{eq:for bijection}
U\le \frac{V}{\lambda^r}+\frac{1}{(1-\lambda)^r\binom{n}{k}}\sum_{\substack{S\subset \n\\|S|=k}} \sum_{j\in S}  \trace\!\left(\Big(\sum_{s\in S\setminus\{j\}} B_s\Big)^{q-1}B_j\right).
\end{equation}

Now,
\begin{multline}\label{eq:string ST1}
\sum_{\substack{S\subset \n\\|S|=k}} \sum_{j\in S}  \trace\!\left(\Big(\sum_{s\in S\setminus\{j\}} B_s\Big)^{q-1}B_j\right)=
\sum_{\substack{T\subset \n\\|T|=k-1}} \sum_{j\in \n\setminus T}  \trace\!\left(\Big(\sum_{t\in T} B_t\Big)^{q-1}B_j\right)\\
= \sum_{\substack{T\subset \n\\|T|=k-1}} \trace\!\left(\Big(\sum_{t\in T} B_t\Big)^{q-1}\Big(\sum_{j\in \n\setminus T} B_j\Big)\right)\le
\sum_{\substack{T\subset \n\\|T|=k-1}} \trace\!\left(\Big(\sum_{t\in T} B_t\Big)^{q-1}\Big(\sum_{j=1}^n B_j\Big)\right),
\end{multline}
where in the last step of~\eqref{eq:string ST1} we used the fact that if $A,B,C\in M_d(\R)$ are symmetric and positive semidefinite then $\trace(AB)\le \trace(A(B+C))$.  To bound the final term in~\eqref{eq:string ST1}, use H\"older's inequality for traces to deduce that for every $T\subset\n$ we have

\begin{align}\label{eq:use holder q}
\nonumber \trace\!\left(\Big(\sum_{t\in T} B_t\Big)^{q-1}\Big(\sum_{j=1}^n B_j\Big)\right)&\le \left(\trace\!\left(\Big(\sum_{j=1}^n B_j\Big)^q\right)\right)^{\frac{1}{q}}\left(\trace\!\left(\Big(\sum_{t\in T} B_t\Big)^q\right)\right)^{1-\frac{1}{q}}\\
&= \frac{nW^{\frac{1}{q}}}{k}\left(\trace\!\left(\Big(\sum_{t\in T} B_t\Big)^q\right)\right)^{1-\frac{1}{q}},
\end{align}
where we recall the definition of $W$ in~\eqref{eq:def UVW}.

The function $t\mapsto t^q$ is operator trace-increasing (see Theorem~2.10 in~\cite{Car10}), i.e., if $C,D\in M_d(\R)$ are symmetric and positive semidefinite with $C\le D$ then $\trace(C^q)\le \trace(D^q)$. Consequently,  for every $T\subsetneq \n$ and $u\in \n$ we have $\trace\left(\left(\sum_{t\in T} B_t\right)^q\right)\le \trace\left(\left(B_u+\sum_{t\in T} B_t\right)^q\right)$. By raising this inequality to the power $(q-1)/q$ and averaging over all $u\in \n\setminus T$ we see that
\begin{equation}\label{eq:averaged over complement of T}
\left(\trace\!\left(\Big(\sum_{t\in T} B_t\Big)^q\right)\right)^{1-\frac{1}{q}}\le \frac{1}{n-|T|}\sum_{u\in \n\setminus T}\left(\trace\!\left(\Big(\sum_{t\in T\cup\{u\}} B_t\Big)^q\right)\right)^{1-\frac{1}{q}}.
\end{equation}
Hence, by combining~\eqref{eq:use holder q} and~\eqref{eq:averaged over complement of T} with~\eqref{eq:string ST1}, we see that
\begin{align}
\nonumber \sum_{\substack{S\subset \n\\|S|=k}}& \sum_{j\in S}  \trace\!\left(\Big(\sum_{s\in S\setminus\{j\}} B_s\Big)^{q-1}B_j\right)\\ \label{eq:T version with u}
&\le
\frac{nW^{\frac{1}{q}}}{k(n-k+1)}\sum_{\substack{T\subset \n\\|T|=k-1}} \sum_{u\in \n\setminus T}\left(\trace\!\left(\Big(\sum_{t\in T\cup\{u\}} B_t\Big)^q\right)\right)^{1-\frac{1}{q}}\\&=
\frac{nW^{\frac{1}{q}}}{n-k+1}\sum_{\substack{S\subset \n\\|S|=k}} \left(\trace\!\left(\Big(\sum_{s\in S} B_s\Big)^q\right)\right)^{1-\frac{1}{q}}, \label{eq:bijection}
\end{align}
where for~\eqref{eq:bijection} note that for every $S\subset \n$ with $|S|=k$ the term corresponding to $\sum_{s\in S} B_s$ occurs in the sum that appears in~\eqref{eq:T version with u} with multiplicity $k$, once for each $u\in S$.

Recalling the definition of $U$ in~\eqref{eq:def UVW}, by Jensen's inequality we see that
\begin{equation}\label{eq:jensen U}
\frac{1}{\binom{n}{k}}\sum_{\substack{S\subset \n\\|S|=k}} \left(\trace\!\left(\Big(\sum_{s\in S} B_s\Big)^q\right)\right)^{1-\frac{1}{q}}\le U^{1-\frac{1}{q}}.
\end{equation}
By substituting~\eqref{eq:jensen U} into~\eqref{eq:bijection} and using $k\le n/2$, we have
\begin{equation}\label{eq:WU}
\frac{1}{\binom{n}{k}} \sum_{\substack{S\subset \n\\|S|=k}} \sum_{j\in S}  \trace\!\left(\Big(\sum_{s\in S\setminus\{j\}} B_s\Big)^{q-1}B_j\right)\le 2W^{\frac{1}{q}}U^{\frac{q-1}{q}}.
\end{equation}
In conjunction with~\eqref{eq:WU}, it follows from~\eqref{eq:for bijection}   that
\begin{equation}\label{eq:to simplify UVW}
U\le \min \left\{\frac{V}{\lambda^r}+\frac{2W^{\frac{1}{q}}U^{\frac{q-1}{q}}}{(1-\lambda)^r}:\ \lambda\in (0,1)\right\}\le \left(V^{\frac{1}{r+1}}+2^{\frac{1}{r+1}}W^{\frac{1}{q(r+1)}}U^{\frac{q-1}{q(r+1)}}\right)^{r+1},
\end{equation}
where the final inequality in~\eqref{eq:to simplify UVW} is seen by choosing $1/\lambda=1+\left(2W^{\frac{1}{q}}U^{\frac{q-1}{q}}/V\right)^{\frac{1}{r+1}}$. By~\eqref{eq:to simplify UVW},
\begin{equation}\label{eq:UVW bootstrap}
U^{\frac{1}{r+1}}\le V^{\frac{1}{r+1}}+2^{\frac{1}{r+1}}W^{\frac{1}{q(r+1)}}U^{\frac{q-1}{q(r+1)}}.
\end{equation}
The desired inequality~\eqref{eq:goal with UVW} is a formal consequence of~\eqref{eq:UVW bootstrap}, as follows. If $U\le (4q/\log(2q))^{r+1}V$ then~\eqref{eq:goal with UVW} holds true because $r+1\le q$. We may therefore assume that $U> (4q/\log(2q))^{r+1}V$, in which case~\eqref{eq:UVW bootstrap} implies that
\begin{equation}\label{eq:to divide by U}
\frac{U^{\frac{1}{r+1}}}{(2q)^{\frac{1}{2q}}}\le\left(1-\frac{\log(2q)}{4
q}\right)U^{\frac{1}{r+1}}\le 2^{\frac{1}{r+1}}W^{\frac{1}{q(r+1)}}U^{\frac{q-1}{q(r+1)}},
\end{equation}
where we used the fact that $(1-t)\ge e^{-2t}$ for every $t\in [0,1/2]$. The estimate~\eqref{eq:to divide by U} simplifies to
$$
U\le 2^q(2q)^{\frac{r+1}{2}}W\le 2^q(2q)^{\frac{q}{2}}W\le\left(\frac{4q}{\log(2q)}\right)^qW,
$$
where we used the elementary inequality $\log t\le \sqrt{t}$, which holds true for every $t\in (0,\infty)$.
\qed

\medskip

\noindent{\bf Acknowledgements.} We are grateful to Eric Carlen and Oded Regev for helpful pointers to the literature related to Section~\ref{sec:trace}, as well as showing us  counter-examples to~\eqref{eq:hope subadditive PSD} when $q=4$. We also thank the anonymous referee for carefully reading our manuscript and many helpful comments. A.~N. was supported in part by the NSF, the BSF, the Packard Foundation and the Simons Foundation. G.~S. was supported in part by the ISF and the BSF.

\bibliographystyle{abbrv}
\bibliography{Xp}

 \end{document}